\documentclass[12pt,twoside]{preprint}
\usepackage{amssymb}
\usepackage{amsmath}
\usepackage{hyperref}
\usepackage{breakurl}
\usepackage{mhenvs}
\usepackage{mhequ}
\usepackage{mhsymb}
\usepackage{booktabs}
\usepackage{tikz}
\usepackage{mathrsfs}
\usepackage{times}
\usepackage{microtype}
\usepackage{comment}
\usepackage{wasysym}
\usepackage{centernot}

\makeatletter
\pgfdeclareshape{crosscircle}
{
  \inheritsavedanchors[from=circle] 
  \inheritanchorborder[from=circle]
  \inheritanchor[from=circle]{north}
  \inheritanchor[from=circle]{north west}
  \inheritanchor[from=circle]{north east}
  \inheritanchor[from=circle]{center}
  \inheritanchor[from=circle]{west}
  \inheritanchor[from=circle]{east}
  \inheritanchor[from=circle]{mid}
  \inheritanchor[from=circle]{mid west}
  \inheritanchor[from=circle]{mid east}
  \inheritanchor[from=circle]{base}
  \inheritanchor[from=circle]{base west}
  \inheritanchor[from=circle]{base east}
  \inheritanchor[from=circle]{south}
  \inheritanchor[from=circle]{south west}
  \inheritanchor[from=circle]{south east}
  \inheritbackgroundpath[from=circle]
  \foregroundpath{
    \centerpoint%
    \pgf@xc=\pgf@x%
    \pgf@yc=\pgf@y%
    \pgfutil@tempdima=\radius%
    \pgfmathsetlength{\pgf@xb}{\pgfkeysvalueof{/pgf/outer xsep}}%
    \pgfmathsetlength{\pgf@yb}{\pgfkeysvalueof{/pgf/outer ysep}}%
    \ifdim\pgf@xb<\pgf@yb%
      \advance\pgfutil@tempdima by-\pgf@yb%
    \else%
      \advance\pgfutil@tempdima by-\pgf@xb%
    \fi%
    \pgfpathmoveto{\pgfpointadd{\pgfqpoint{\pgf@xc}{\pgf@yc}}{\pgfqpoint{-0.707107\pgfutil@tempdima}{-0.707107\pgfutil@tempdima}}}
    \pgfpathlineto{\pgfpointadd{\pgfqpoint{\pgf@xc}{\pgf@yc}}{\pgfqpoint{0.707107\pgfutil@tempdima}{0.707107\pgfutil@tempdima}}}
    \pgfpathmoveto{\pgfpointadd{\pgfqpoint{\pgf@xc}{\pgf@yc}}{\pgfqpoint{-0.707107\pgfutil@tempdima}{0.707107\pgfutil@tempdima}}}
    \pgfpathlineto{\pgfpointadd{\pgfqpoint{\pgf@xc}{\pgf@yc}}{\pgfqpoint{0.707107\pgfutil@tempdima}{-0.707107\pgfutil@tempdima}}}
  }
}
\makeatother

\colorlet{symbols}{black}      
\colorlet{testcolor}{green!60!black}
\colorlet{darkblue}{blue!60!black}

\def\symbol#1{\textcolor{symbols}{#1}}
\def\1{\mathbf{\symbol{1}}}

\usetikzlibrary{shapes.misc}
\usetikzlibrary{shapes.symbols}
\usetikzlibrary{shapes.geometric}
\usetikzlibrary{snakes}
\usetikzlibrary{decorations}
\usetikzlibrary{decorations.markings}

\def\drawx{\draw[-,solid] (-3pt,-3pt) -- (3pt,3pt);\draw[-,solid] (-3pt,3pt) -- (3pt,-3pt);}
\tikzset{
	root/.style={circle,fill=testcolor,inner sep=0pt, minimum size=2mm},
	dot/.style={circle,fill=black,inner sep=0pt, minimum size=0.7mm},
	circ/.style={circle,draw=black,inner sep=0pt, minimum size=1mm},
	var/.style={circle,fill=black!10,draw=black,inner sep=0pt, minimum size=2mm},
	dotred/.style={circle,fill=black!50,inner sep=0pt, minimum size=2mm},
	generic/.style={semithick,shorten >=1pt,shorten <=1pt},
	oddfunc/.style={generic, dotted},
	dist/.style={ultra thick,draw=testcolor,shorten >=1pt,shorten <=1pt},
	testfcn/.style={ultra thick,testcolor,shorten >=1pt,shorten <=1pt,<-},
	testfunction/.style={ultra thick,testcolor,shorten >=1pt,shorten <=1pt},
	testfcnx/.style={ultra thick,testcolor,shorten >=1pt,shorten <=1pt,<-,
		postaction={decorate,decoration={markings,mark=at position 0.6 with {\drawx}}}},
	kprime/.style={semithick,shorten >=1pt,shorten <=1pt,densely dashed,->},
	kprimex/.style={semithick,shorten >=1pt,shorten <=1pt,densely dashed,->,
		postaction={decorate,decoration={markings,mark=at position 0.4 with {\drawx}}}},
	kernel/.style={semithick,shorten >=1pt,shorten <=1pt,->,draw=black},
	multx/.style={shorten >=1pt,shorten <=1pt,
		postaction={decorate,decoration={markings,mark=at position 0.5 with {\drawx}}}},
	kernelx/.style={semithick,shorten >=1pt,shorten <=1pt,->,
		postaction={decorate,decoration={markings,mark=at position 0.4 with {\drawx}}}},
	kernel1/.style={->,semithick,shorten >=1pt,shorten <=1pt,postaction={decorate,decoration={markings,mark=at position 0.45 with {\draw[-] (0,-0.1) -- (0,0.1);}}}},
	kernel2/.style={->,semithick,shorten >=1pt,shorten <=1pt,postaction={decorate,decoration={markings,mark=at position 0.45 with {\draw[-] (0.05,-0.1) -- (0.05,0.1);\draw[-] (-0.05,-0.1) -- (-0.05,0.1);}}}},
	kernelBig/.style={semithick,shorten >=1pt,shorten <=1pt,decorate, decoration={zigzag,amplitude=1.5pt,segment length = 3pt,pre length=2pt,post length=2pt}},
	rho/.style={dotted,semithick,shorten >=1pt,shorten <=1pt},
	renorm/.style={shape=circle,fill=white,inner sep=1pt},
	labl/.style={shape=rectangle,fill=white,inner sep=1pt},
cumu2n/.style={inner sep=3pt},
cumu2/.style={draw=red!80,fill=red!40},
cumu2b/.style={draw=blue!80,fill=blue!40},
cumu2nv/.style={inner sep=3pt},
cumu2v/.style={draw=red!80,fill=white,very thick},
cumu3/.style={regular polygon, regular polygon sides=3,draw=red!80,rounded corners=3pt,fill=red!40,minimum size=5mm},
cumu4/.style={regular polygon, regular polygon sides=4,draw=red!80,rounded corners=3pt,fill=red!40,minimum size=7mm},
cumu5/.style={regular polygon, regular polygon sides=5,draw=red!80,rounded corners=3pt,fill=red!40,minimum size=7mm},
	xi/.style={circle,fill=symbols!10,draw=symbols,inner sep=0pt,minimum size=1.2mm},
	xix/.style={crosscircle,fill=symbols!10,draw=symbols,inner sep=0pt,minimum size=1.2mm},
	xib/.style={circle,fill=symbols!10,draw=symbols,inner sep=0pt,minimum size=1.6mm},
	xibx/.style={crosscircle,fill=symbols!10,draw=symbols,inner sep=0pt,minimum size=1.6mm},
	not/.style={circle,fill=symbols,draw=symbols,inner sep=0pt,minimum size=0.5mm},
	>=stealth,
	}
\makeatletter
\def\DeclareSymbol#1#2#3{\expandafter\gdef\csname MH@symb@#1\endcsname{\tikz[baseline=#2,scale=0.12,draw=symbols]{#3}}\expandafter\gdef\csname MH@symb@#1s\endcsname{\scalebox{0.7}{\tikz[baseline=#2,scale=0.12,draw=symbols]{#3}}}}
\def\<#1>{\csname MH@symb@#1\endcsname}
\makeatother

\DeclareSymbol{Xi22}{0.5}{\draw (0,0) node[xi] {} -- (-1,1) node[not] {} -- (0,2) node[xi] {};}

\DeclareSymbol{Xi2}{-2}{\draw (0,-0.25) node[xi] {} -- (-1,1) node[xi] {};}
\DeclareSymbol{Xi3}{0}{\draw (0,0) node[xi] {} -- (-1,1) node[xi] {} -- (0,2) node[xi] {};}
\DeclareSymbol{Xi4}{2}{\draw (0,0) node[xi] {} -- (-1,1) node[xi] {} -- (0,2) node[xi] {} -- (-1,3) node[xi] {};}
\DeclareSymbol{Xi2X}{-2}{\draw (0,-0.25) node[xi] {} -- (-1,1) node[xix] {};}
\DeclareSymbol{XXi2}{-2}{\draw (0,-0.25) node[xix] {} -- (-1,1) node[xi] {};}

\DeclareSymbol{IXi2}{0}{\draw (0,-0.25) node[not] {} -- (-1,1) node[xi] {} -- (0,2) node[xi] {};}
\DeclareSymbol{IXi^2}{-1}{\draw (-1,1) node[xi] {} -- (0,0) node[not] {} -- (1,1) node[xi] {};}

\DeclareSymbol{XiX}{-2.8}{\node[xibx] {};}
\DeclareSymbol{Xi}{-2.8}{\node[xib] {};}
\DeclareSymbol{IXiX}{-1}{\draw (0,-0.25) node[not] {} -- (-1,1) node[xix] {};}

\DeclareSymbol{Xi3b}{-1}{\draw (-1,1) node[xi] {} -- (0,0) node[xi] {} -- (1,1) node[xi] {};}

\DeclareSymbol{IXi3}{2}{\draw (0,-0.25) node[not] {} -- (-1,1) node[xi] {} -- (0,2) node[xi] {} -- (-1,3) node[xi] {};}
\DeclareSymbol{IXi}{-2}{\draw (0,-0.25) node[not] {} -- (-1,1) node[xi] {};}
\DeclareSymbol{XiI}{-2}{\draw (0,-0.25) node[xi] {} -- (-1,1) node[not] {};}

\DeclareSymbol{Xi4b}{-1}{\draw(0,1.5) node[xi] {} -- (0,0); \draw (-1,1) node[xi] {} -- (0,0) node[xi] {} -- (1,1) node[xi] {};}
\DeclareSymbol{Xi4b'}{-1}{\draw(0,1.5) node[xi] {} -- (0,-0.2); \draw (-1,1) node[xi] {} -- (0,-0.2) node[not] {} -- (1,1) node[xi] {};}
\DeclareSymbol{Xi4c}{0}{\draw (0,1) -- (0.8,2.2) node[xi] {};\draw (0,-0.25) node[xi] {} -- (0,1) node[xi] {} -- (-0.8,2.2) node[xi] {};}
\DeclareSymbol{Xi4d}{-4.5}{\draw (0,-1.5) node[not] {} -- (0,0); \draw (-1,1) node[xi] {} -- (0,0) node[xi] {} -- (1,1) node[xi] {};}
\DeclareSymbol{Xi4e}{0}{\draw (0,2) node[xi] {} -- (-1,1) node[xi] {} -- (0,0) node[xi] {} -- (1,1) node[xi] {};}
\DeclareSymbol{Xi4e'}{0}{\draw (0,2) node[xi] {} -- (-1,1) node[xi] {} -- (0,-0.2) node[not] {} -- (1,1) node[xi] {};}

\newtheorem{example}[lemma]{Example}
\newtheorem{assumption}[lemma]{Assumption}

\def\s{\mathfrak{s}}
\def\KK{\mathfrak{K}}

\let\eref\eqref

\def\DeltaM{\Delta^{\!M}}
\def\hDeltaM{\hat \Delta^{\!M}}

\def\MM{\mathscr{M}}
\def\TT{\mathscr{T}}
\def\DD{\mathscr{D}}

\def\${|\!|\!|}
\def\Wick#1{\,\colon\!\! #1 \colon}
\def\E{\mathbf{E}}
\def\T{\mathbf{T}}
\def\Cum{\mathbf{E}_c}

\def\powerset{\mathscr{P}}

\def\emptyset{\mathop{\centernot\ocircle}}

\DeclareSymbol{X}{-2.4}{\node[dot] {};}
\DeclareSymbol{1}{0}{\draw[white] (-.4,0) -- (.4,0); \draw (0,0)  -- (0,1.2) node[dot] {};}
\DeclareSymbol{2}{0}{\draw (-0.6,1.6) node[dot] {} -- (0,0) -- (0.6,1.6) node[dot] {}; \node at (0,0) {};}
\DeclareSymbol{Kxi}{0}{\draw (0,2.2) node[dot] {} -- (0,-0.2); \node at (0,0) {};}
\DeclareSymbol{KKxi}{-3}{
	\draw (-0.5, 0.3) -- (0.4,-1.1) ;
	\draw  (-0.5, 0.3) -- (0.5,1.3) node[dot] {};
	\node at (0.3,0) {};}
\DeclareSymbol{K2}{0}{
	\draw (-0.9,2.2) node[dot] {} -- (0,1) -- (0.9,2.2) node[dot] {};
	\draw (0,1) -- (0,-0.2);
	\node at (0,0) {};
}

\DeclareSymbol{3}{0}{\draw (0,0) -- (0,1.2) node[dot] {}; \draw (-.7,1) node[dot] {} -- (0,0) -- (.7,1) node[dot] {};}

\DeclareSymbol{31}{-3}{\draw (0,0) -- (0,-1) -- (1,0) node[dot] {}; \draw (0,0) -- (0,1.2) node[dot] {}; \draw (-.7,1) node[dot] {} -- (0,0) -- (.7,1) node[dot] {};}
\DeclareSymbol{30}{-3}{\draw (0,0) -- (0,-1); \draw (0,0) -- (0,1.2) node[dot] {}; \draw (-.7,1) node[dot] {} -- (0,0) -- (.7,1) node[dot] {};}
\DeclareSymbol{32}{-3}{\draw (0,0) -- (0,-1) -- (1,0) node[dot] {}; \draw (0,0) -- (0,-1) -- (-1,0) node[dot] {}; \draw (0,0) -- (0,1.2) node[dot] {}; \draw (-.7,1) node[dot] {} -- (0,0) -- (.7,1) node[dot] {};}
\DeclareSymbol{22}{-3}{\draw (0,0.3) -- (0,-1) -- (1,0) node[dot] {}; \draw (0,0.3) -- (0,-1) -- (-1,0) node[dot] {};\draw (-.7,1) node[dot] {} -- (0,0.3) -- (.7,1) node[dot] {};}
\DeclareSymbol{20}{-3}{\draw (0,0) -- (0,-1);\draw (-.7,1) node[dot] {} -- (0,0) -- (.7,1) node[dot] {};}
\DeclareSymbol{12}{-3}{\draw (0,0.3) -- (0,-1) -- (1,0) node[dot] {}; \draw (0,0.3) -- (0,-1) -- (-1,0) node[dot] {};\draw (-.7,1) node[dot] {} -- (0,0.3);}
\DeclareSymbol{10}{-3}{\draw (0,0.3) -- (0,-1);\draw (-.7,1) node[dot] {} -- (0,0.3);}
\DeclareSymbol{21}{-3}{
	\draw (0,0.3) -- (0,-1) -- (1,0) node[dot] {};
	\draw (-.7,1) node[dot] {} -- (0,0.3) -- (.7,1) node[dot] {};
	\node at (0,0) {};}
\DeclareSymbol{21a}{-3}{
	\draw (-0.5, 0.3) -- (0,-1) -- (1,0) node[dot] {};
	\draw  (-0.5, 0.3) -- (0.5,1.3) node[dot] {};
	\node at (0.3,0) {};}
\DeclareSymbol{211}{1.5}{
	\draw (0,0) -- (-2,3) node[dot] {};
	\draw (0,0)  -- (1,1) node[dot] {};
	\draw (-0.67,1)  -- (.33,2) node[dot] {};
	\draw (-1.33,2)  -- (-0.33,3) node[dot] {};
	\node at (0,0) {};
}
\DeclareSymbol{4}{-4}{
	\draw (-1,0)  -- (0,-1.2) -- (1,0) ;
	\draw (-1.5,1) node[dot] {} -- (-1,0) -- (-0.5,1) node[dot] {};  
	\draw (1.5,1) node[dot] {} -- (1,0) -- (0.5,1) node[dot] {};
	\node at (0,0) {};
}

\begin{document}

\title{A central limit theorem for the KPZ equation}
\author{Martin Hairer$^1$ and Hao Shen$^2$}
\institute{University of Warwick, UK, \email{M.Hairer@Warwick.ac.uk}
\and University of Warwick, UK, \email{pkushenhao@gmail.com}}

\maketitle

\begin{abstract}
We consider the KPZ equation in one space dimension 
driven by a stationary centred space-time random field,
which is sufficiently integrable and mixing, but not necessarily Gaussian.
We show that, in the weakly asymmetric regime, the solution to this equation
considered at a suitable large scale and in a suitable reference frame 
converges to the Hopf-Cole solution 
to the KPZ equation driven by space-time Gaussian white noise.
While the limiting process depends only on the integrated variance of 
the driving field, the diverging constants appearing in the definition of
the reference frame also depend on higher order moments. 
\end{abstract}


\setcounter{tocdepth}{2}
\tableofcontents

\section{Introduction}

Over the past fifteen years or so, there has been a great
deal of interest in reaching a mathematical understanding of the ``KPZ fixed point'',
a conjectured self-similar stochastic process describing 
the large-scale behaviour of one-dimensional interface growth models \cite{KPZOrig}.
While substantial progress has been achieved for microscopic models presenting some form of
``integrable structure'' (see 
\cite{MR1682248,MR1887169,MR2628936,MR2784327,MR2796514,MR3152785} and references therein for some of the major
milestones), the underlying mechanisms leading to the universality of the KPZ fixed point
in general are still very poorly understood.

A related question is that of the universality of the ``crossover regime'' in
weakly asymmetric interface fluctuation models. Recall that the large-scale
behaviour of perfectly symmetric models
is expected to be governed by the Edwards-Wilkinson model \cite{EW82}, i.e.\ the
additive stochastic heat equation. Weakly asymmetric models are thus expected to exhibit a 
crossover regime where they transition from a behaviour described by the
Edwards-Wilkinson model to one described by the KPZ fixed point.
It is conjectured that this behaviour is also universal and is described
by the KPZ equation \cite{KPZOrig}, a nonlinear stochastic PDE formally given by
\begin{equ} [e:KPZ]
\partial_t h=\partial_x^2 h + \lambda \,(\partial_x h)^2 +\xi \;.
\end{equ}
Here, $\xi$ denotes space-time Gaussian white noise, $\lambda\in\R$, and the spatial variable
$x$ takes values in the circle $S^1$.
One then recovers the Edwards-Wilkinson model for $\lambda = 0$, while 
the KPZ fixed point corresponds to $\lambda = \infty$ (modulo rescaling).

Mathematically rigorous evidence for this was first 
given in \cite{MR1462228}, where the authors show that the
height function of the weakly asymmetric simple exclusion process converges to \eqref{e:KPZ}
(interpreted in a suitable sense). More recently, this result was extended
to several other discrete models in \cite{AKQ,MR3176353,2013arXiv1302,2015arXiv1505}.
One difficulty arising is that it is not clear \textit{a priori} what \eqref{e:KPZ}
even means: solutions are not differentiable, so that $(\partial_x h)^2$ does not exist
in the classical sense. The trick used in \cite{MR1462228} to circumvent this is that 
the authors consider the so-called ``Hopf-Cole solution'' which, \textit{by definition}, is
given by $h = \lambda^{-1}\log(Z)$, where $Z$ solves
$dZ = \d_x^2 Z\,dt + \lambda Z\,dW$ in the It\^o sense, see \cite{DPZ}.

A more direct interpretation of \eqref{e:KPZ} was recently given in \cite{KPZ}
using the theory of rough paths, where it
was also shown that this interpretation generalises the Hopf-Cole solution.
(In the sense that the solution theory relies on more input data than just a realisation of
the driving noise $\xi$, but there is a canonical choice for this additional data for which the
solution then coincides with the Hopf-Cole solution.)
A somewhat different approach leading to the same result 
is given by the theory of regularity structures developed in \cite{Regularity} (see also \cite{IntroRegularity} for an introduction to the theory). Its application to the KPZ
equation is explained for example in \cite{FrizHairer}, but was already implicit in
\cite{Regularity,WongZakai}.

In \cite{KPZJeremy}, the authors exploit this theory to show that
a rather large class of continuous weakly asymmetric interface fluctuation models 
is indeed described by the KPZ equation in the crossover regime. 
These models are of the type
\begin{equ}
\d_t h = \d_x^2 h + \sqrt \eps P(\d_x h) + \zeta\;,
\end{equ}
for $\zeta$ a smooth space-time Gaussian field with finite range correlations
and $P$ an even polynomial.
In that article, the Gaussianity assumption on the noise 
is used in an essential way to show convergence of an associated ``model" (in the sense of the 
theory of regularity structures) to the limiting model used in \cite{KPZ,FrizHairer}.

The purpose of the present article is to remove the Gaussianity assumption
of \cite{KPZJeremy}.
Since we focus on dealing with the non-Gaussian character of our model,
we only consider the case of quadratic nonlinearities, but we expect that a result analogous to
that of \cite{KPZJeremy} also holds in our context.
We assume that the driving noise is given by 
a stationary and centred space-time random field $\zeta$
with suitable regularity, integrability, and mixing conditions given below.
It is important to note that we do \textit{not} assume that $\zeta$
is white in time, so martingale-based techniques as used for example in \cite{MR1462228} 
and more recently in \cite{MourratWeber} do not apply here.
We also remark that although we only consider the KPZ equation in this article,
most of the methods such as the non-Gaussian Wick calculus, non-Gaussian  diagrammatic techniques etc. developed here should also apply to other equations such as the dynamical $\Phi^4$ equation, the generalised parabolic Anderson model (gPAM) and so on.

Given a non-Gaussian random field $\zeta$ as above, we consider the equation
\begin{equ}
\partial_t h=\partial_x^2 h + \sqrt{\eps} \lambda \,(\partial_x h)^2 +\zeta \;.
\end{equ}
Rescaling it by setting $ \tilde h_\eps(t,x)=\eps^{1\over 2} h \big( \eps^{-2}t,\eps^{-1}x \big)$ yields 
\begin{equ} [e:rescaledKPZ]
\partial_t  \tilde h_\eps=\partial_x^2  \tilde h_\eps +\lambda (\partial_x \tilde h_\eps)^2 
	+\tilde \zeta_\eps \;,
\end{equ}
where $\tilde\zeta_\eps = \eps^{-3/2} \zeta(\eps^{-2}t,\eps^{-1}x)$.
Our main result, Theorem~\ref{theo:main} below, then states that there exists a suitable choice
of inertial reference frame such that, in this frame, 
$\tilde h_\eps$ converges as $\eps \to 0$ to the Hopf-Cole solution 
of the KPZ equation.
One interesting fact is 
that the ``vertical speed'' $C_\eps$ of this reference frame
differs by a term of order $\eps^{-1/2}$ from what it would be if $\zeta$ were replaced by a centred Gaussian
field with the same covariance. This difference depends to leading order 
on the third cumulants of $\zeta$ and to order $1$ on its fourth cumulants.
The limiting solution however depends only on the covariance of $\zeta$, which is why we
call our result a ``Central Limit Theorem''.
We will also see that unless the covariance of $\zeta$ is invariant 
under spatial inversion $x\mapsto -x$, the frame in which convergence takes place
also has a ``horizontal speed'' of order one.
(This is also the case when the driving noise is Gaussian, but it didn't arise in
\cite{KPZJeremy,WongZakai} since symmetry under spatial inversion was assumed there.)


Our assumptions on the space-time random field $\zeta$ are twofold. First, we
need translation invariance and sufficient regularity / integrability.

\begin{assumption} \label{assump-0}
The random field $\zeta$ is stationary, centred and Lipschitz continuous. Furthermore,
$\E \left( |\zeta(z)|^p + |\nabla \zeta(z)|^p \right) < \infty$ for all $p>0$,
and $\zeta$ is normalised so that $\E(\zeta(0)\zeta(z))$ integrates to $1$.
\end{assumption}

Our second assumption is that there is enough independence.
In order to formulate it, given $K \subset \R^2$, we denote by $\CF_K$
the $\sigma$-algebra generated by the point evaluations $\{\zeta(z)\,:\, z \in K\}$.

\begin{assumption} \label{assump-mix}
For any two compact sets $K_1,K_2 \subset \R^2$
such that $\inf_{z_i\in K_i} |z_1-z_2| \geq 1$, $\CF_{K_1}$ and $\CF_{K_2}$
are independent.
\end{assumption}

The strict independency condition in Assumption~\ref{assump-mix}  is only here for simplicity and
may be relaxed. We expect both 
our results and our technique of proof to still hold when the dependency 
decays sufficiently fast (e.g.\ exponentially or polynomially with a sufficiently high power). 
Also, in Assumption~\ref{assump-0}, the condition that all moments of $\zeta$ 
are finite can be relaxed: our proof only requires us to have finite $p$th moments for some
sufficiently large but fixed $p$.

One typical example of a random field $\zeta$ satisfying these assumptions is as follows.
Let $\mu$ be a finite positive measure on $\CC_0^1(\R^2)$ which is supported on the set of functions
with support in the centred ball of radius $1/2$, and such that
$\int \|\phi\|_{\CC^1}^p\,\mu(d\phi) < \infty$ for every $p$. Denote by 
$\hat \mu$ a realisation of the Poisson point process  on
$\CC_0^1(\R^2) \times \R^2$ with intensity $\mu \otimes dz$ and set
\begin{equ}
\zeta(z) = \int \phi(z-z')\, \hat \mu(d\phi\otimes dz') - \int \int \phi(z)\,dz \mu(d\phi)\;.
\end{equ}
Then it is immediate that $\zeta$ satisfies Assumptions~\ref{assump-0} and \ref{assump-mix}.
Furthermore, nonlinear combinations $F(\zeta_1,\ldots,\zeta_k)$
of fields of this type still satisfy these assumptions provided
that $F$ grows at most polynomially at infinity.

Given this setting, we state the main result of this article as the following theorem. 
Note that since we consider periodic boundary conditions, we need to replace the field  $\zeta$ by
 a suitable ``periodisation", which is formulated in Assumption~\ref{ass:approxField} below. For simplicity of notation we still denote by $\tilde\zeta_\eps$ the noise
 obtained from rescaling the ``periodised" random field.

\begin{theorem} \label{theo:main}
Let $\zeta$ be a random field satisfying Assumptions~\ref{assump-0} and \ref{assump-mix}. 
Let $h^{(\eps)}_0$ be a sequence of smooth functions on $S^1$ that converge in 
$\CC^\beta$ as $\eps\to 0$ to a limit $h_0 \in \CC^\beta$ with $\beta\in(0,1)$.
Let $ \tilde h_\eps$ be the solution to \eref{e:rescaledKPZ}
on $S^1$ with initial condition $h_0^{(\eps)}$ and with $\tilde\zeta_\eps$ as in \eqref{e:def-zeta-eps} below 
satisfying Assumption~\ref{ass:approxField}. 
Then there exist velocities $v_h$ and $v_v^{(\eps)}$ 
such that, for every $T>0$, the family of random functions
$\tilde h_\eps (t, x - v_h t)-v_v^{(\eps)} t$ converges in law as $\eps \to 0$ to 
the Hopf-Cole solution of the KPZ equation \eref{e:KPZ} 
with initial data $h_0$ in the space $\CC^\eta([0,T]\times S^1)$,
for any $\eta\in(0,{1\over 2}\wedge \beta)$.

Furthermore, one has
\begin{equ}  [e:vxvy]
v_v^{(\eps)} = \lambda \eps^{-1} C_0 + 2 \lambda^2 \eps^{-1/2} C_1 +  \lambda^3 c\;,\quad 
v_h = 4\lambda^2 \hat c\;,
\end{equ}
where $C_0$ and $\hat c$
only depend on the second moment of 
the random field $\zeta$,
while $C_1$ depends on its third 
moment and $c$ depends on the second and fourth moments. 
If $\E(\zeta(0,0)\zeta(t,x))$ is even as a function of $x$, then $\hat c=0$.
\end{theorem}

Here $\CC^\alpha$ is the H\"{o}lder space of functions or distributions with regularity $\alpha$ as 
in \cite{Regularity}. (This coincides with the parabolic Besov space 
$B^\alpha_{\infty,\infty,\mathrm{loc}}$.)
The subscripts $h$ and $v$ in our notation only indicate that
$v_h$ and $v_v^{(\eps)}$ are the ``horizontal" and ``vertical" speeds respectively as 
mentioned after \eqref{e:rescaledKPZ}.
The choice of the constants appearing in Theorem~\ref{theo:main} will be given in Section~\ref{sec:values}; in particular $\hat c$ will be defined in Lemma~\ref{lem:const-correct}.

\begin{remark}
The constant $C_0$ is the one given by the ``na\"\i ve'' Wick ordering and can be explicitly expressed  as 
\begin{equ} [e:def-C0]
C_0=\int_{\R^4} P'(z)P'(\bar z) \,\kappa_2(z-\bar z) \,dz\, d\bar z\;,
\end{equ}
 where $P$ denotes the heat kernel, $P'$ is its space derivative,
 and $\kappa_2(z-\bar z) = \E (\zeta(z)\zeta(\bar z))$.
(For comparison we should identify $\kappa_2$ here with $\rho * \rho$ where $\rho$ is the mollifier function in \cite{KPZJeremy}.)
\end{remark}

\begin{remark}
At first sight, the result may appear somewhat trivial since, by the classical central limit
theorem, $\tilde\zeta_\eps \to \xi$ weakly so that it may not be surprising that solutions to
\eqref{e:rescaledKPZ} converge to solutions to \eqref{e:KPZ}. 
The problem of course is that the solution map to \eqref{e:KPZ}
is not continuous in any of the usual topologies in 
which this weak convergence takes place. This is apparent in the fact that, in order to obtain the right
limit, we need to perform
a non-trivial change of reference frame with divergent velocities that depend on higher
order cumulants of $\zeta$. 
\end{remark}

\begin{proof}[of Theorem~\ref{theo:main}]
Let $ h_\eps(t,x) \eqdef \tilde h_\eps (t, x - v_h t)-v_v^{(\eps)} t$. Then it is immediate to check that  $h_\eps$
satisfies the equation
\begin{equ} [e:renorm-equa]
\partial_t  h_\eps=\partial_x^2   h_\eps +\lambda (\partial_x   h_\eps)^2 
	-v_h \,\partial_x  h_\eps  - v_v^{(\eps)} + \zeta_\eps \;,
\end{equ}
where  $\zeta_\eps$ is defined in \eqref{e:def-zeta-eps} below.
To prove the convergence of $ h_\eps$ as stated in the theorem, we apply the theory of regularity structures.
The theory is briefly reviewed in Section~\ref{sec:framework},
where the regularity structure $\TT$ suitable for the study of equation \eref{e:KPZ}
is defined.
The proof of Theorem~\ref{theo:main} then follows in the following way.

We consider the (local) solution map to the abstract fixed point problem
for functions $H$ taking values in a suitable regularity structure 
(see \eqref{e:abs-KPZ} below)
\begin{equ}[e:absFP]
H = \CP \one_{t>0} \bigl(\lambda (\DD H)^2 + \Xi\bigr) + Ph_0\;,
\end{equ}
where $\CP$ is a suitable operator representing space-time convolution with the heat kernel and
$\Xi$ is an element in the regularity structure representing the noise term. 
It is then known from \cite{KPZ,Regularity,FrizHairer,WongZakai,KPZJeremy} that
there exists a natural model $\hat Z = (\hat \Pi,\hat \Gamma)$ for $\TT$ 
such that if $H$ solves \eqref{e:absFP} for $\hat Z$, then $\CR H$ coincides
with the Hopf-Cole solution to the KPZ equation. Furthermore, the model in question is obtained
as a limit of renormalised lifts of smooth Gaussian approximations to space-time white noise.

As a next step, we define a family of renormalised models $\hat Z_\eps$ (see Section~\ref{sec:renorm}) 
and show that they obey the following two crucial properties:
\begin{claim}
\item By appropriately choosing the values of the renormalisation constants,
we can ensure that $\hat Z_\eps$ converges in distribution as $\eps \to 0$
to $\hat Z$. This is the content of Theorem~\ref{theo:diagonal} below.
\item By appropriately choosing $v_h$ and $v_v^{(\eps)}$, we can ensure that 
if $H_\eps$ solves \eqref{e:absFP} for $\hat Z_\eps$, then $\CR H_\eps$ coincides
with the solution $h_\eps$ to \eqref{e:renorm-equa}.
See \eref{e:renom-eq}, combined with the fact that the correct choice of renormalisation constants
for the first step is given by \eref{e:choiceell}. In particular $v_h$ and $v_v^{(\eps)}$
indeed have the forms claimed in \eref{e:vxvy}.
\end{claim}
As a consequence of \cite[Thm~7.8]{Regularity}, these two properties immediately yield the existence of a 
random time $\tau$ such that $h_\eps$, when restricted to the interval $[0,\tau]$,
converges in law to the solution to the Hopf-Cole solution to the KPZ equation in 
$\CC^\eta([0,T]\times S^1)$,
for every $\eta\in(0,{1\over 2}\wedge \beta)$, see also the discussion at the end of 
Section~\ref{sec:framework}.

The extension of the convergence result to any finite deterministic time interval
follows from \cite[Prop.~7.11]{Regularity} and the a priori knowledge 
that Hopf-Cole solutions to the KPZ equation are
global and $\alpha$-H\"older continuous for every $\alpha < 1/2$.
\end{proof}

\subsection*{Structure of the article}

The article is organised as follows.
We start with the basic setup of our problem in Section~\ref{sec:Setting},
which includes some more discussions about the assumptions on the noise,
followed by 
a brief  introduction to the theory of regularity structures as the framework for our proof.
In Section~\ref{sec:renorm}, we then define a collection of renormalisation
constants which, through the action of the renormalisation group, determines
a family of renormalised models in the sense of \cite{Regularity}. 
In Section~\ref{sec:tightness}, we  show some key 
technical results (Proposition~\ref{cor:Wick-field}, Lemma~\ref{lem:collapse}, Proposition~\ref{prop:main}),
then obtain bounds on 
arbitrary moments of these renormalised models,
uniformly in the small parameter $\eps$.
This is the most technical step: because our models are built from 
a non-Gaussian random field, their moments are not automatically bounded
by their variance, unlike in the Gaussian case where equivalence of moments holds.
The proof of these moment bounds depends on some general technical tools developed in Section~\ref{sec:bounds}. 
The main result of that section is Corollary~\ref{cor:ulti-bound}, which provides conditions that can be straightforwardly checked and yield the desired moment estimates.
In the last section, Section~\ref{sec:identify},
we show Theorem~\ref{theo:diagonal}
which identifies the limiting model as the ``KPZ model", and therefore implies that
in a suitable reference frame the limiting solution does indeed coincide with 
the Hopf-Cole solution to the KPZ equation driven by Gaussian white noise.

\subsection*{Acknowledgements}

{\small
We are indebted to Hendrik Weber and Jeremy Quastel for several interesting discussions
on this topic.
MH gratefully acknowledges financial support from the Philip Leverhulme Trust and the 
European Research Council. We also thank Henri Elad-Altman and the two anonymous referees for their very careful proofreading.
}

\section{Framework} \label{sec:Setting}

\subsection{Assumption on noise} \label{sec:periodisation}

We assume that $\zeta$ satisfies Assumption~\ref{assump-0}
and \ref{assump-mix}.
 Let then $\{\zeta^{(\eps)}\}_{\eps \in(0,1]}$
be a family of random fields with the following properties.

\begin{assumption}\label{ass:approxField}
The fields $\zeta^{(\eps)}$ are stationary, almost surely periodic in 
space with period $1/\eps$,
and satisfy $\sup_{\eps \in (0,1]}\E |\zeta^{(\eps)}(0)|^p < \infty$ for every $p \ge 1$.

Furthermore, for any two sets $K_1$, $K_2 \subset \R^2$ that are periodic in space with period $1/\eps$
and such that $\inf_{z_i \in K_i}|z_1 - z_2| \ge 1$, $\CF_{K_1}^{(\eps)}$ and $\CF_{K_2}^{(\eps)}$ are independent,
where $\CF_K^{(\eps)}$ is the $\sigma$-algebra generated by $\{\zeta^{(\eps)}(z): z \in K\}$.

Finally, for every $\eps > 0$, there is a coupling of
$\zeta$ and $\zeta^{(\eps)}$ such that, for every $T > 0$ and every $\delta > 0$,
\begin{equ}[e:wanted]
\sup_{|t| \le T\eps^{-2}}\sup_{|x| \le (1 - \delta)/(2\eps)} \lim_{\eps \to 0} \eps^{-3} \E |\zeta(t,x) - \zeta^{(\eps)}(t,x)|^2 = 0\;.
\end{equ}
\end{assumption}

\begin{remark}
We could have replaced \eqref{e:wanted} by a slightly weaker condition, but 
the one given here has the advantage of being easy to state. It is also easy to verify in the 
examples we have in mind, see Example~\ref{ex:Poisson} below.
\end{remark}

We then set 
\begin{equ} [e:def-zeta-eps]
\tilde \zeta_\eps(t,x) = \eps^{-3/2}\zeta^{(\eps)}(t/\eps^2, x/\eps) \;,\qquad 
\zeta_\eps(t,x) = \eps^{-3/2}\zeta^{(\eps)}(t/\eps^2, (x-v_h t)/\eps) \;,
\end{equ}
where $v_h$ will be specified in Subsection~\ref{sec:values} (see Lemma~\ref{lem:const-correct} and recall from \eqref{e:vxvy} that $v_h = 4\lambda^2 \hat c$).
Note that both $\zeta_\eps$ and $\tilde \zeta_\eps$ are periodic in space with period $1$.

\begin{example}\label{ex:Poisson}
To show that Assumption~\ref{ass:approxField} is not unreasonable, consider the following example.
Let $\mu$ be a Poisson point process on $\R^2 \times [0,1]$ with uniform intensity measure,
let $\phi(t,x,a)$ be a smooth compactly supported function (say with support in the ball 
of radius $1$), and set
\begin{equ}[e:defzeta]
\zeta(t,x) = \int_{\R^2 \times [0,1]} \phi(t-s, x-y, a) \mu(ds,dy,da)\;.
\end{equ}
Let $\mu^{(\eps)}$ be the periodic extension to $\R^2 \times [0,1]$ of 
a Poisson point process on $\R \times [-1/(2\eps),1/(2\eps)] \times [0,1]$ 
with uniform intensity measure and let $\zeta^{(\eps)}$ be as 
in \eqref{e:defzeta}, with $\mu$ replaced by $\mu^{(\eps)}$.
Then it is immediate to verify that \eref{e:wanted} is satisfied, since 
the natural coupling between $\zeta$ and $\zeta^{(\eps)}$ is such that
$\zeta = \zeta^{(\eps)}$ on $\R \times [K-1/(2\eps), 1/(2\eps)-K]$, for some fixed $K$.
\end{example}

\subsection{Regularity structures setup} \label{sec:framework}

In order to prove the convergence result stated in Theorem~\ref{theo:main},
we will make use of the theory of regularity structures developed in \cite{Regularity}. 
Before delving into the technical details,
 we give a short summary of the main concepts of the theory, as applied to the problem at hand, in this subsection.
One can find a more concise exposition of the theory in the lecture notes \cite{IntroRegularity}.

Recall that a {\it regularity structure} consists of two pieces of data. First, a graded vector 
space $\CT=\bigoplus_{\alpha\in A}\CT_\alpha$, where the index set $A$,
called the set of homogeneities, is a subset of $\R$
which is locally finite and bounded from below. In our case, each $\CT_\alpha$ is furthermore
finite-dimensional. Second, 
a ``structure group'' $\CG$ of continuous linear transformations of $\CT$ such that,
for every $\Gamma\in \CG$,  every $\alpha\in A$, and every $\tau\in\CT_\alpha$, one has
\begin{equ} [e:strugroup]
\Gamma \tau -\tau \in  \CT_{<\alpha}
\qquad
\mbox{with} \;\;\;
\CT_{<\alpha} \eqdef \bigoplus_{\beta<\alpha} \CT_\beta \;.
\end{equ}
A canonical example is the space
of polynomials in $d$
indeterminates, with index set $A=\N$, $\CT_n$ consisting of homogeneous 
polynomials of $\s$-scaled degree $n$, and $\CG$ consisting of translations.
We recall here a scaling $\s \in\N^d$ is simply a vector $(\s_1,\ldots,\s_d)$ of 
positive relatively prime integers, and the $\s$-scaled degree of 
a monomial $X^k \eqdef \prod_{i=1}^d X_i^{k_i}$,
where $k=(k_1,\ldots,k_d)$ is a multi-index,
is defined as $|k|\eqdef \sum_{i=1}^d \s_i k_i$.

In general, \cite[Sec.~8]{IntroRegularity} gives a recipe
to build a regularity structure from any given family of (subcritical in an appropriate sense) 
semi-linear stochastic PDEs. In our case, space and time are both one dimensional so
$d=2$, and the natural scaling to choose is 
the parabolic scaling $\s=(2,1)$, and thus the scaling dimension of space-time is 
\[
|\s|=3 
\]
which is fixed throughout the article.
This scaling defines a distance
$\left\Vert x-y\right\Vert _{\s}$ on $\mathbf{R}^{2}$ by
\begin{equ}
  \left\Vert x \right\Vert_{\s}^4 \eqdef |x_0|^2 + |x_1|^4  \;,
\end{equ}
where $x_0$ and $x_1$ are time and space coordinates respectively.
To simplify the notation, in the sequel we will often just write $|x|=\|x\|_\s$
for a space-time point $x$.

The regularity structure relevant for the analysis
of \eref{e:KPZ} is then built in the following way.
We write $\CU$ for a collection of  formal expressions
that will be useful to describe the solution $h$,
 $\CU'$ for a collection of formal expressions useful to describe its spatial distributional derivative 
$\partial_x h$, and $\CV$ for a collection of formal expressions
 useful to describe the right hand side of the KPZ equation \eref{e:KPZ}. 
We decree that $\CU$ and $\CU'$ contain at least the polynomials
of two indeterminates $X_0,X_1$, denoting the time and space directions
respectively.
We then introduce three additional symbols, $\Xi$, $\CI$ and $\CI'$,
where $\Xi$ will be interpreted as an abstract representation of the 
driving noise $\zeta$ (or rather $\zeta_\eps$), and $\CI$ and $\CI'$ will be interpreted as
 the operation of convolving with a truncation of the heat kernel  and its spatial derivative respectively.
In view of the structure of the equation \eref{e:KPZ}, to have a regularity
structure that is rich enough to describe the fixed point problem
we should also decree that
\begin{equs} [e:buildUUV]
\tau,\bar\tau \in \CU' &\Rightarrow \tau\bar\tau\in\CV\;, \\
\tau\in\CV & \Rightarrow \CI(\tau)\in\CU, \quad \CI'(\tau)\in \CU' \;.
\end{equs}
The sets  $\CU$, $\CU'$ and $\CV$ are then defined as the smallest 
collection of formal expressions such that $\Xi\in\CV$, $X^k\in\CU$,
$X^k\in\CU'$, and \eref{e:buildUUV} holds.
Following \cite{Regularity},
we furthermore decree that $\tau\bar\tau=\bar\tau \tau$ and $\CI(X^k)=\CI'(X^k)=0$.
We then define 
\begin{equ}
\CW\eqdef \CU\cup\CU'\cup\CV \;.
\end{equ}

For each formal expression $\tau$, the homogeneity $|\tau| \in\R$
 is assigned in the following way.
First of all, for a multi-index $k=(k_0,k_1)$, we set $|X^k| =|k|=2k_0 + k_1$,
which is consistent with the parabolic scaling $\s$ chosen above.
We then set 
\begin{equ}
|\Xi| = -{3\over 2}-\bar\kappa
\end{equ}
where $\bar\kappa>0$ is a fixed small number,
and we define the homogeneity for every formal expression by
\begin{equ}
|\tau\bar\tau| = |\tau|+|\bar\tau|\;,
\qquad
|\CI(\tau)| =|\tau|+2 \;,
\qquad
|\CI'(\tau)| =|\tau|+1 \;.
\end{equ}
The linear space $\CT$ is then defined as the linear span
of $\CW$,
and $\CT_\alpha$ 
is the subspace spanned by $\{\tau\,:\,|\tau|= \alpha\}$.
By a simple power-counting argument, one finds
that as long as $\bar\kappa<1/2$, 
the sets $\{\tau\in\CW:|\tau|<\gamma\}$ are finite for every $\gamma \in \R$,
reflecting the fact that the KPZ equation \eref{e:KPZ}
in one space dimension is subcritical.

As in \cite{KPZ},
we use a graphical shorthand for the formal expressions in $\CW$.
We use dots to represent the expression $\Xi$,
and lines to represent the operator $\CI'$.
Joining of formal expressions by their roots
denotes their product. For example,
one has $\<Kxi> = \CI'(\Xi)$.
Writing $\Psi=\CI'(\Xi)$ as a shorthand, one also has
$\<2>=\Psi^2$ and $\<21>=\Psi\CI'(\Psi^2)$, etc. 
With this notation, the formal expressions in $\CW$ with negative 
homogeneities other than $\Xi$ are given by
\begin{equs} [e:list-symbols]
|\<2>|=-1-2 \bar\kappa \;, \qquad  &
|\<21>| =-{1\over 2}-3 \bar\kappa \;,\qquad
|\<Kxi>| =-{1\over 2}- \bar\kappa \;,\\
|\<4>|  =|\<211>|= & -4\bar\kappa \;, \qquad
 |\<21a>| = |\<K2>|=-2\bar\kappa \;,
\end{equs}
provided that $\bar\kappa  >0$ is sufficiently small. 
We will denote by $\CW_-$ the above formal expressions with negative 
homogeneities.
In fact, we will never need to consider the full space $\CT$,
but it will be sufficient to consider the subspace generated by all elements of homogeneity less
than some  large enough number $\sigma$. In practice, it actually turns out to be 
sufficient to choose any $\sigma > {3\over 2} +\bar \kappa$.

The last ingredient of the regularity structure, namely the structure group $\CG$,
can also be described explicitly. However, the precise description of $\CG$ 
does not matter in the present paper, so we refrain from giving it.
The interested reader is referred to \cite[Sec.~8]{Regularity} 
as well as \cite[Sec.~3]{KPZJeremy}
for the more general case in which the nonlinearity of the equation
also contains powers of higher order. The regularity structure $(\CT,\CG)$ 
defined here is the same as in \cite[Chapter~15]{FrizHairer}.

Given the regularity structure, a crucial concept is that of a {\it model} for it, 
which associates to each formal expression in the regularity structure
a ``concrete" function or distribution on $\R^d$ (which will be the space-time $\R^2$ in our case). It consists of a pair $(\Pi,\Gamma)$
of functions
\begin{equs}[2]
\Pi \colon \R^d &\to \CL(\CT,\CS') \quad & \quad \Gamma \colon \R^d\times \R^d  &\to \CG \\
x &\mapsto \Pi_x & (x,y) &\mapsto \Gamma_{xy} 
\end{equs}
where $\CL(\CT,\CS') $ is the space of linear maps from $\CT$
to the space of distributions (over space-time  $\R^2$ in our case),
such that $\Gamma_{xx}$ is the identity and
\begin{equ}
\Pi_y= \Pi_x\circ\Gamma_{xy} \;, 
\qquad
\Gamma_{xy}\Gamma_{yz} 
=\Gamma_{xz}\;,
\end{equ}
 for every $x,y,z\in\R^d$.
Furthermore, for every $\gamma>0$, 
we require that
\begin{equ} [e:analy-bnd]
|(\Pi_x \tau) (\varphi^\lambda_x)| \lesssim \lambda^{|\tau|} \;,
\qquad
\|\Gamma_{xy} \tau\|_m \lesssim |x-y|^{|\tau|-m}\;,
\end{equ}
where ``$\lesssim $" stands for ``less or equal" up to a constant which is
uniform over $x,y$ in any compact set of $\R^d$, and over all $\lambda\in(0,1]$,
all smooth test functions $\varphi \in \CB$,
all $m<|\tau|$,
and  all homogeneous elements $\tau\in\CT_{<\gamma}$ with $\|\tau\| \le 1$.   
Here, 
 $\|\cdot\|_m$ denotes the norm of the component in $\CT_m$ (which could be any norm since $\CT_m$ is finite dimensional), 
\begin{equ}[e:defBB]
	\CB = \{\varphi: \|\varphi\|_{\CC^2}\le 1,\mbox{and $\varphi$ is supported in the unit ball} \}\;,
\end{equ}
and $\varphi^\lambda_x (z)$ denotes the following recentred and rescaled
\footnote{We note that the notation $\lambda$ which stands for the rescaling parameter
here has also been used in the equation \eqref{e:KPZ} as the coefficient of the nonlinearity.
This should not cause any confusion because it will be clear what $\lambda$ means
from the context.}
 version of $\phi$
\begin{equ}
\varphi^\lambda_x (z)\eqdef \lambda^{-3} \varphi\Big( {z_0 -x_0 \over \lambda^2}, {z_1 -x_1 \over \lambda}\Big) \;,
\end{equ}
where the scaling exponents are chosen according to our scaling $\s=(2,1)$.
The reason why $\CC^2$ appears in \eqref{e:defBB} is that this is the smallest integer $r$ such that 
$r + \min_{\tau \in \CW}|\tau| > 0$.

Besides these requirements, we also impose that our models are {\it admissible}.
To define this notion, as in \cite{Regularity}, we first fix a 
kernel $K:\R^2\to\R$ such that $\mathop{\textrm{supp}}K\subset \{|(t,x)|\le 1, t> 0\}$, $K(t,-x)=K(t,x)$ and $K$ coincides with the heat kernel in $|(t,x)|<1/2$. We also require that for every polynomial $Q$ on $\R^2$ of parabolic degree $2$ (i.e. $Q(t,x)=a_1 t+a_2 x^2+a_3 x+ a_4$ for $a_i\in\R$)
one has $\int_{\R^2} K(t,x)Q(t,x) \,dx\,dt =0$.

Given this truncated heat kernel $K$, the set $\MM$ of \textit{admissible models} 
consists of those models such that, for every multiindex $k$,
\begin{equ}  \label{e:admissible1}
\bigl(\Pi_z X^k\bigr)(\bar z) = (\bar z- z)^k\;,
\end{equ}
and such that
\begin{equs} 
\bigl(\Pi_z \CI \tau\bigr)(\bar z)
	 =   \int_{\R^2} \! & K(\bar z - x)    \bigl(\Pi_{z} \tau\bigr)(d x) \label{e:admissible3}\\
	  & - \sum_{|k| < |\CI\tau|} {(\bar z - z)^k \over k!} 
	\int_{\R^2} \! D^k K(z - x)\bigl(\Pi_{z} \tau\bigr)(d x) \;.
\end{equs}
We refer to \cite[Sec.~5]{Regularity} for a discussion on the meaning of these expressions in general.
As an example to illustrate \eref{e:admissible3}, for $\CI\tau = \<KKxi>$
with $|\<KKxi>| = {1\over 2}-\bar\kappa$, one has
\begin{equ} 
\bigl(\Pi_z \<KKxi>)(\bar z)
	 =  \int_{\R^2}  \big(K(\bar z - x) - K(z - x) \big) \bigl(\Pi_{z} \<Kxi>  \bigr)(d x) \;.
\end{equ}
When analysing  the objects $\<21a>$ and $\<211>$ in Section~\ref{sec:tightness}
we will see how the increment of $K$ appearing in this expression is exploited.

\begin{remark}
It actually turns out that if $\Pi $ 
satisfies the {\em first} analytical bound in \eref{e:analy-bnd},
 and
the structure group $\CG$ acts on $\CT$ in a way that is compatible 
with the definition of the admissible model (which is true in our case), 
 then the second analytical bound in \eref{e:analy-bnd}
is {\em automatically} satisfied. This is a consequence of \cite[Thm.~5.14]{Regularity}.\end{remark}

Given a continuous space-time function $\zeta_\eps$, there is a canonical way of building an admissible model
$ (\Pi^{(\eps)}, \Gamma^{(\eps)})$,
as in  \cite[Sec.~8]{Regularity}.
 First, we set $\Pi^{(\eps)}_z \Xi = \zeta_\eps$ independently
of $z$, and we define it on $X^k$ as in \eref{e:admissible1}. 
Then, we define $\Pi^{(\eps)}_z$ recursively by \eref{e:admissible3}, as well as the identity
\begin{equ}[e:canonical]
\bigl(\Pi_z^{(\eps)} \tau \bar \tau\bigr)(\bar z) = \bigl(\Pi_z^{(\eps)} \tau\bigr)(\bar z)
\bigl(\Pi_z^{(\eps)} \bar \tau\bigr)(\bar z)\;.
\end{equ}
Note that this is only guaranteed to make sense if $\zeta_\eps$ is a sufficiently regular function.
 It was shown in
\cite[Prop.~8.27]{Regularity} that 
this does indeed define an admissible model 
for every continuous function $\zeta_\eps$, and we will call 
this admissible model  the canonical lift of $\zeta_\eps$.
However, we emphasize that, unlike in the case of rough paths \cite{MR2257130}, 
not every admissible model is obtained in this way, 
or even as a limit of such models.
In particular, while the renormalisation procedure used in 
Section~\ref{sec:renorm} preserves the admissibility of models, it does not preserve the property of being a limit
of canonical lifts of regular functions.

It was then shown in \cite{Regularity} how to
associate an {\it abstract fixed point problem} to the equation \eref{e:KPZ}
 in a certain space $\CD^{\gamma,\eta}$ of $\CT_{<\gamma}$-valued functions.  
This space is the analogue in this context of a H\"{o}lder space of order $\gamma$,
with the exponent $\eta$ allowing for a possible singular behavior at $t=0$.
In our particular case, the abstract fixed point problem can be formulated in
$\CD^{\gamma,\beta}$ for $\gamma>{3\over 2}+\bar\kappa$,
and for $\beta > 0$ 
 such that  the initial data $h_0\in\CC^\beta$.
One of the main results in \cite{Regularity} is the {\it reconstruction theorem}
which states that for every $U \in \CD^{\gamma,\beta}$ with $\gamma > 0$, there
exists a unique distribution $\CR U$ such that,
near every point $z$, $\CR U$ ``looks like'' $\Pi_z U(z)$ up to an error of order $\gamma$.
The operator $U \mapsto \CR U$ is called the reconstruction operator associated to the
model $(\Pi,\Gamma)$.


The idea to formulate this abstract fixed point is to define multiplication, differentiation, 
and integration against the heat kernel on elements in $\CD^{\gamma,\eta}$, so one can write
\begin{equ} [e:abs-KPZ]
H=\CP \one_{t>0}\big(\lambda (\DD H)^2 + \Xi \big)  +  Ph_0
\end{equ}
where $Ph_0$ is the heat kernel acting on the initial data $h_0$ 
and can be  interpreted
as an element of $\CD^{\gamma,\beta}$ by \cite[Lemma~7.5]{Regularity},
 $\mathscr D$ is given by $\mathscr D \CI(\tau) = \CI'(\tau)$, 
 and the product is simply given by pointwise multiplication in $\CT$.
The symbol $\one_{t>0}$ represents the (scalar) indicator function of the set $\{(t,x)\,:\, t>0\}$
and the linear operator $\CP$ represents space-time convolution by the heat kernel
in the space $\CD^{\gamma,\beta}$, in the sense that 
\begin{equ}
\CR\CP H =P *\CR H\;,
\end{equ}
for every $H \in \CD^{\gamma,\eta}$ for $\gamma > 0$ and $\eta > -2$.
The explicit expression for the operator $\CP$ does not matter for our purpose but
can be found in \cite[Sec.~5]{Regularity}, 
let us just mention that $\CP H$ and $ \CI H$  only differ by elements taking values in the 
linear span of the $X^k$ and that it satisfies a Schauder estimate.

Theorem~\cite[Theorem~7.8]{Regularity} shows that, for every admissible model, there exists 
a unique $T>0$ such that the above fixed point problem  has
a unique solution in 
$\CD^{\gamma,\beta} ([0,T]\times S^1)$. If the model is given by the canonical lift
$(\Pi^{(\eps)},\Gamma^{(\eps)})$ of a continuous function $\zeta_\eps$, 
then $\CR H$ coincides with the classical solution to \eref{e:KPZ}. 

Unfortunately, these canonical lifts do not converge to a limit in $\mathscr M$ as $\eps \to 0$.
However, we will show that one can build a natural finite-dimensional family of continuous transformations  
$\hat M_\eps$ of $\mathscr M$ such that
the ``renormalised models" $(\hat\Pi^{(\eps)},\hat \Gamma^{(\eps)}) \eqdef \hat M_\eps(\Pi^{(\eps)},\Gamma^{(\eps)})$ 
do converge to an admissible limiting model.
These transformations are parametrised by elements
of a {\it renormalisation group} $\mathfrak R$ associated to our regularity structure $(\CT,\CG)$.
Since the precise definition of $\mathfrak R$ requires a bit more understanding
of the algebraic properties of our regularity structure and is not quite relevant 
to the present article, we simply refer to \cite[Section~8]{Regularity}.

Once we obtain the convergence of renormalised models,
the function $h=\CR H$ with $\CR$ the reconstruction map
associated with the limiting model is then the limiting solution 
stated in Theorem~\ref{theo:main}.
The H\"older  regularity of the solution $h$
is given by the minimum of the regularity of $Ph_0$ and the  lowest homogeneity of elements of $\CU$  beside the Taylor polynomials which 
is $\CI(\Xi)$ with homogeneity $1/2-\bar \kappa$
where $\bar \kappa>0$ is a sufficiently small fixed number,
so \cite[Prop.~3.28]{Regularity} yields  
$h\in \CC^\eta ([0,T]\times S^1)$,
with $\eta\in({1\over 4},{1\over 2}\wedge \beta)$.

\section{Renormalisation} \label{sec:renorm}

We now give an explicit description of the renormalisation maps $\hat M_\eps$ described
above. These are parametrised by linear maps $M_\eps \colon \CT \to \CT$
belonging to the ``renormalisation group''
$\mathfrak R$ which was introduced in this context in \cite[Sec.~8.3]{Regularity}.
As a matter of fact, we only need to consider a certain $5$-parameter subgroup of $\mathfrak{R}$.
This subgroup consists of elements $M\in\mathfrak{R}$ of the form
$M=\exp(-\sum_{i=1}^{5} \ell_i L_i)$ where the $\ell_i$
are real-valued constants and
 the generators $L_i \colon \CT \to \CT$ are determined by the following contraction rules:
\begin{equ}[e:defLi]
L_1 \colon \<2> \mapsto \one\;,
\quad
L_2 \colon \<21a> \mapsto \one\;,
\quad
L_3 \colon \<21> \mapsto \one\;,
\quad 
L_4 \colon \<211> \mapsto \one\;,
\quad 
L_5 \colon \<4> \mapsto \one\;.
\end{equ}
For $L_i$ with $i \neq 2$, these rules are extended to the whole regularity structure 
by setting $L_i \tau = 0$ if $\tau$ is not the element given above.
For $L_2$ however, the above contraction rule
should be understood in the sense that for an arbitrary formal expression $\tau$,
$L_2\tau $ equals the sum of all expressions obtained by performing
a substitution of the type $\<21a> \mapsto \one$. For instance, one has
\begin{equ}
L_2 \<21> = 2 \; \<Kxi>\; \;, \qquad
L_2 \<211>=2\; \<21a> + \<K2> \;.
\end{equ}
Given an admissible model $(\Pi,\Gamma)$ and an element $M$ of the above type for some
choice of constants $\ell_i$,
we can then build a new model $(\hat \Pi,\hat \Gamma) = \hat M (\Pi,\Gamma)$ by postulating 
that the identity
\begin{equ}[e:defRenorm]
\hat \Pi_x \tau = \Pi_x M\tau\;,
\end{equ}
holds for every $\tau \in \CT$. This is sufficient to determine
$(\hat \Pi,\hat \Gamma)$ since in our case $\hat \Gamma$ is uniquely determined
by $\hat \Pi$ and the knowledge that the new model is again admissible,
see \cite{Regularity}. 

\begin{remark}
The fact that
$M\in\mathfrak{R}$ can be checked essentially in the same way as in \cite{KPZJeremy}, with the simplification that since we are only interested in a quadratic nonlinearity, the symbol $\CE$ never appears.
The only minor difference is the appearance of the generators $L_2$
and $L_3$ here. In fact, using the notations of \cite[Sec.~8.3]{Regularity}, 
one can verify that in our case one has the identities
$\hDeltaM \CJ_k(\tau) = \CJ_k(M \tau) \otimes \one$
and $\DeltaM \tau= M\tau \otimes \one$. The second identity, combined with \cite[Eq.~8.34]{Regularity}, implies 
\eqref{e:defRenorm}, while the combination of both identities and the fact that
$M$ is upper triangular imply that the ``upper triangular condition" 
in \cite[Def.~8.41]{Regularity} holds. 
\end{remark}

Let now $(\Pi,\Gamma)$ be the model obtained by the canonical lift of an arbitrary 
smooth function $\zeta$,
and let $(\hat \Pi,\hat \Gamma) = \hat M (\Pi,\Gamma)$ be as above, 
again with $M=\exp(-\sum_{i=1}^{5} \ell_i L_i)$. 
It is then straightforward to verify that if 
$\CR^M$ is the reconstruction operator associated to $(\hat \Pi,\hat \Gamma)$ 
and $H$ solves the corresponding abstract fixed point problem \eqref{e:abs-KPZ}, then
$h=\CR^M H$ satisfies 
the renormalised equation
\begin{equ} [e:renom-eq]
\partial_t  h=\partial_x^2  h + \lambda (\partial_x  h)^2  
	- 4 \lambda^2 \ell_2 \,\partial_x h +\zeta 
	-(\lambda \ell_1 +2\lambda^2 \ell_3+ 4 \lambda^3 \ell_4+\lambda^3\ell_5)  + 4\lambda^3 \ell_2^2\;.
\end{equ}
This can be shown in the same way as \cite[Prop.~15.12]{FrizHairer},
except that in our case the term $2\lambda^2\<21>$ appearing in the 
expression of $(\partial H)^2 +\Xi$
produces a constant $2\lambda^2 \ell_3$ when being acted upon 
by $L_3$ in the definition of $M$. Note that the last term $4\lambda^3 \ell_2^2$ comes from the action of $\frac12 (\ell_2 L_2)^2$ in the exponential that defines $M$ on the term $4\lambda^3 \<211>$.

\subsection{Joint cumulants}

Before giving the definition of the actual values for the $\ell_i$ relevant to our analysis, 
we review the definition and basic properties of joint cumulants. 
Given a collection of random variables $\CX = \{X_\alpha\}_{\alpha \in A}$ for some index set $A$,
and a subset $B \subset A$, we write $X_B \subset \CX$ and $X^B$ as shorthands for
\begin{equ}
X_B = \{X_\alpha\,:\, \alpha \in B\} \;,\qquad X^B = \prod_{\alpha\in B}X_\alpha\;.
\end{equ}
Given a finite set $B$, we furthermore write $\CP(B)$ for the collection of all partitions of $B$,
i.e.\ all sets $\pi \subset \powerset(B)$ (the power set of $B$) such that $\bigcup \pi = B$
and such that any two distinct elements of $\pi$ are disjoint. 

\begin{definition} \label{def:cumu}
Let $\CX$ be a collection of random variables as above 
with finite moments of all orders.
For any finite set 
$B \subset A$, we define the cumulant 
$\Cum(X_B)$ inductively over $|B|$ by
\begin{equ} [e:mome2cumu]
	\E \big( X^B\big)
	= \sum_{\pi \in \CP(B)} \prod_{\bar B\in\pi} \Cum \big(X_{\bar B}\big)\;.
\end{equ}
\end{definition}

The expression \eqref{e:mome2cumu} does indeed determine the cumulants
uniquely by induction over $|B|$. This is  because the right hand 
side only involves $\Cum(X_B)$, which is what we want to define,
as well as $\Cum(X_{\bar B})$ for some $\bar B$ with $|\bar B| < |B|$, which is already defined 
by the inductive hypothesis. 
 If all the random variables are centred and jointly Gaussian, then 
 it follows from Wick's theorem that $\Cum(X_B)$ always vanishes unless $|B|=2$.
Henceforth, we will use the notation $\kappa_n$ for the $n$th joint 
cumulant function  of the field $\zeta^{(\eps)}$: 
\begin{equ} [e:def-kappa]
\kappa_n (z_1,\ldots,z_n) \eqdef \Cum\bigl(\{\zeta^{(\eps)}(z_1),\ldots, \zeta^{(\eps)}(z_n)\}\bigr) \;.
\end{equ}
Note that $\kappa_1 = 0$ since $\zeta^{(\eps)}$ is assumed to be centred and $\kappa_2$ is its
covariance function.

\begin{remark}
A reader who is expertised in cumulants may know that one can also define cumulants 
for a collection of random variables in which some variables may appear more than once.
In order to keep notations simple, in Definition~\ref{def:cumu} we only considered subsets $B$ in the 
above definition, so that each random variable $X_\alpha$ is only allowed
to appear at most once in the collection $X_B$. This is sufficient for the purpose
of this article. As a matter of fact, one can easily reduce oneself to
this case by considering the augmented collection $\bar \CX = \{\bar X_{\bar \alpha}\}_{\bar \alpha \in \bar A}$
with $\bar A = A \times \N$ and random variables $\bar X_{\bar \alpha}$ 
such that $\bar X_{(\alpha,k)} = X_\alpha$ almost surely.
\end{remark}

\begin{remark}
There is a slight abuse of notation here since $\kappa_n$ does of course depend on
$\eps$ in general, but this dependence is very weak.
\end{remark}

We refer to \cite{PeccatiTaqqu} and \cite[Sec.~13.5]{MR887102} for 
the properties of joint cumulants; see also the recent article \cite{LukkarinenMarcozzi}.
A  property that will be useful in our problem is that the cumulant is zero when some of the variables are independent from the others. We formulate this in terms of our random field $\zeta^{(\eps)}$.
Given a collection $\{z_i\}_{i=1}^p$ of space-time points, we define $z_i\sim z_j$  if $|z_i-z_j| \le 1$,
and extend this into an equivalence relation on $\{z_i\}_{i=1}^p$. 

\begin{lemma} \label{lem:cumu-close}
The cumulants have the property 
that if $\kappa_p(z_1,\ldots,z_p) \neq 0$, then $z_1,\ldots,z_p$ all belong to the same equivalence class. \qed
\end{lemma}

%

\subsection{Values of the renormalisation constants} \label{sec:values}

We now have all the ingredients in place to determine the relevant values of the
renormalisation constants $\ell_i$.
We denote by $\kappa_p^{(\eps)}$ the $p$-th cumulant function of $\zeta_\eps$,
which is the properly rescaled cumulant function of $\zeta^{(\eps)}$ (rescaling according to \eref{e:def-zeta-eps} with a shift by $v_h$).
By Assumption~\ref{ass:approxField}, and the rescaling \eref{e:def-zeta-eps},
one has 
\[
|\kappa_p^{(\eps)}| \lesssim \eps^{-p|\s|/2} \;.
\]
By Lemma~\ref{lem:cumu-close}, one also has
$\kappa_p^{(\eps)} = 0$ unless all of its arguments are located within a parabolic ball of radius $p\eps$ (in fact, given a constant $v$, one can show that 
for $\eps>0$ sufficiently small, if $|(t,x)|>2\eps$, then $|(t,x-vt)|>\eps$).


To define the renormalisation constants,
we introduce some graphical notations which represent our integrations.
In our graphs, a dot $\tikz [baseline=-3] \node[dot] {};$ represents an integration variable,
and the special vertex $\tikz  [baseline=-3] \node[root] {};$ represents the origin $0$.
Each arrow $\tikz [baseline=-3] \draw[kernel] (0,0) to (1,0);$ represents 
the kernel $K'(y-x)$ with $x$ and $y$ being the starting and end points
of the arrow respectively.
A red polygon with $p$ dots inside, for instance 
\begin{tikzpicture}[baseline=-4]
\node		(mid)  	at (0,0) {};
	\node[cumu4]	(mid-cumu) 	at (mid) {};
\node[dot] at (mid.north west) {};
\node[dot] at (mid.south west) {};
\node[dot] at (mid.north east) {};
\node[dot] at (mid.south east) {};
\end{tikzpicture}
with $p=4$, 
represents the cumulant function $\kappa_p^{(\eps)}(z_1,\ldots,z_p)$, with $z_i$ given by the
$p$ integration variables represented by the $p$ dots. As an example,
\begin{equs}
\begin{tikzpicture}  [baseline=10]
\node[root]	(root) 	at (0,0) {};
\node[dot]		(left)  	at (-0.7,.8) {};
\node		(right)  	at (0.7,.8) {};		
	\node[cumu3]	(right-cumu) 	at (0.7, .85) {};
\draw[kernel] (left) to  (root);
\draw[kernel]   (right.east) node[dot]   {}  to (root) ;
\draw[kernel]  (right.north) node[dot]  {}  [bend right=50] to (left);
\draw[kernel]   (right.west) node[dot]  {} to (left);
\end{tikzpicture} 
&= \int_{\R^8} K'(-y) \,K'(-x_3)\,K'(y-x_1)\, K'(y-x_2) \\ 
&\qquad \times \kappa_3^{(\eps)}(x_1,x_2,x_3) \,dx_1\,dx_2\,dx_3\,dy \;.
\end{equs}

We then define a collection of renormalisation constants as follows. Here, we use the notations
$Q^{(\eps)} \approx \eps^{-\alpha} +\sigma$ (where $\sigma\in\{0,1\}$) as a shorthand for the statement
``there exist constants $Q_1$, $Q_2$ such that $\lim_{\eps \to 0} |Q^{(\eps)}-Q_1 \eps^{-\alpha} - Q_2\sigma| = 0$'' (and call $\sigma$ the $O(1)$ correction in the case $\sigma=1$),
and $Q^{(\eps)} \approx c\log \eps$ as a shorthand for
``there exists a constant $Q$ such that $\lim_{\eps \to 0} |Q^{(\eps)}-c\log \eps -Q| = 0$''. Define
\begin{equ}
C_0^{(\eps)} = \;
\begin{tikzpicture}  [baseline=10] 
\node[root]	(root) 	at (0,0) {};
\node[cumu2n]	(a)  		at (0,1) {};	
\draw[cumu2] (a) ellipse (8pt and 4pt);
\draw[kernel] (a.east) node[dot]   {} to [bend left = 60] (root);
\draw[kernel] (a.west) node[dot]   {} to [bend right = 60] (root);
\end{tikzpicture}
\; \approx \eps^{-1} +1 \;,
\qquad
C_1^{(\eps)} = \;
\begin{tikzpicture}  [baseline=10]
\node[root]	(root) 	at (0,0) {};
\node[dot]		(left)  	at (-0.7,.8) {};
\node		(right)  	at (0.7,.8) {};		
	\node[cumu3]	(right-cumu) 	at (0.7, .85) {};
\draw[kernel] (left) to  (root);
\draw[kernel]   (right.east) node[dot]   {}  to (root) ;
\draw[kernel]  (right.north) node[dot]  {}  [bend right=50] to (left);
\draw[kernel]   (right.west) node[dot]  {} to (left);
\end{tikzpicture} 
\; \approx \eps^{-{1\over 2}}  \;,
\end{equ}
where the $O(1)$ correction in $C_0^{(\eps)}$ is zero if $\kappa_2$
is even in the space variable (see Lemma~\ref{lem:const-correct} below),
and 
$C_2^{(\eps)} = 2C_{2,1}^{(\eps)} +C_{2,2}^{(\eps)} $
where
\begin{equ}
C_{2,1}^{(\eps)} = \;
\begin{tikzpicture}  [baseline=20]
\node[root]	(root) 	at (0,0) {};
\node[dot]		(mid)  	at (0,1) {};
\node[dot]		(top)  	at (0,2) {};
\node[cumu2n]	(left)  		at (-.8,1) {};	
	\draw[cumu2] (left) ellipse (4pt and 8pt);
\node[cumu2n]	(right)  		at (.8,1.5) {};	
	\draw[cumu2] (right) ellipse (4pt and 8pt);
\draw[kernel] (mid) to  (root);
\draw[kernel] (top) to  (mid);
\draw[kernel] (left.south) node[dot]   {} to  (root);
\draw[kernel] (left.north) node[dot]   {} to  (top);
\draw[kernel] (right.north) node[dot]   {} to  (top);
\draw[kernel] (right.south) node[dot]   {} to  (mid);
\end{tikzpicture}
\;-\; \frac{(\hat c^{(\eps)})^2}{2}
\; \approx -{\log \eps \over 32\pi \sqrt 3} 
\;,
\qquad
C_{2,2}^{(\eps)} = \;
\begin{tikzpicture}  [baseline=20]
\node[root]	(root) 	at (0,0) {};
\node[dot]		(mid)  	at (0,1) {};
\node[dot]		(top)  	at (0,2) {};
\node		(right)  	at (1,1.15) {};
	\node[cumu4]	(right-cumu) 	at (right) {};
\draw[->,kernel] (mid) to  (root);
\draw[->,kernel] (top) to  (mid);
\draw[->,kernel] (right.south west) node[dot] {} to  (mid);
\draw[->,kernel] (right.south east) node[dot] {} to  (root);
\draw[->,kernel] (right.north west) node[dot] {} to  (top);
\draw[->,kernel,bend right=40] (right.north east) node[dot] {} to  (top);
\end{tikzpicture}
\quad \approx 1
\;,
\end{equ}
and $C_3^{(\eps)} = 2C_{3,1}^{(\eps)} +C_{3,2}^{(\eps)} $
where
\begin{equ}
C_{3,1}^{(\eps)} = \;
\begin{tikzpicture}  [baseline=20]
\node[root]	(root) 	at (0,0) {};
\node[cumu2n]	(mid)  		at (0,.8) {};	
	\draw[cumu2] (mid) ellipse (8pt and 4pt);
\node[cumu2n]	(top)  		at (0,1.6) {};	
	\draw[cumu2] (top) ellipse (8pt and 4pt);
\node[dot]		(left)  	at (-1,.8) {};
\node[dot]		(right)  	at (1,.8) {};
\draw[kernel] (left) to  (root);
\draw[kernel] (right) to  (root);
\draw[kernel] (mid.west) node[dot]   {} to  (left);
\draw[kernel] (mid.east) node[dot]   {} to  (right);
\draw[kernel] (top.west) node[dot]   {} to  (left);
\draw[kernel] (top.east) node[dot]   {} to  (right);
\end{tikzpicture}
\; \approx {\log \eps \over 8\pi \sqrt 3}
\;,
\qquad
C_{3,2}^{(\eps)} = \;
\begin{tikzpicture}  [baseline=20]
\node[root]	(root) 	at (0,0) {};
\node		(mid)  	at (0,1.5) {};
	\node[cumu4]	(mid-cumu) 	at (mid) {};
\node[dot]		(left)  	at (-1,.8) {};
\node[dot]		(right)  	at (1,.8) {};
\draw[->,kernel] (left) to  (root);
\draw[->,kernel] (right) to  (root);
\draw[->,kernel,bend left=30] (mid.south west) node[dot] {} to  (left);
\draw[->,kernel,bend right=30] (mid.south east) node[dot] {} to  (right);
\draw[->,kernel,bend right=30] (mid.north west) node[dot] {} to  (left);
\draw[->,kernel,bend left=30] (mid.north east) node[dot] {} to  (right);
\end{tikzpicture}
\quad \approx 1
\;.
\end{equ}
When $v_h $ is zero,
our constants $C_0^{(\eps)}$, $C_{2,1}^{(\eps)}$, $C_{3,1}^{(\eps)}$
are the same as the renormalisation constants
$C_0^{(\eps)}$, $C_{2}^{(\eps)}$, $C_{3}^{(\eps)}$
 in \cite{KPZJeremy} where their divergence rates were obtained
(just identify $\kappa_2^{(\eps)}$ with the mollifier convolving with itself
 in \cite{KPZJeremy}).
However, since $v_h \neq 0$ in general,  
\begin{equ} [e:def-kappa-hat]
\kappa_2^{(\eps)} (t,x) \neq \hat\kappa^{(\eps)}(t,x) \eqdef \eps^{-3} \kappa_2(t/\eps^2,x/\eps) \;,
\end{equ}
so there may be lower order corrections to the asymptotic behaviors
of the renormalisation constants obtained in \cite{KPZJeremy}.
We have the following lemma.

\begin{lemma} \label{lem:const-correct}
For every sufficiently small $\eps>0$, there exists a choice of the constant $\hat c^{(\eps)}$
such that
\begin{equ} [e:equ-chat]
\hat c^{(\eps)} = 
\begin{tikzpicture}  [baseline=10]
\node[root]	(root) 	at (0,0) {};
\node[dot]		(left)  	at (0,1.2) {};
\node[cumu2n]	(right)  		at (.8,.6) {};	
	\draw[cumu2] (right) ellipse (4pt and 8pt);	
\draw[kernel] (left) to  (root);
\draw[kernel]   (right.south)  node[dot] {}  to (root) ;
\draw[kernel]   (right.north)  node[dot] {} to (left);
\end{tikzpicture} 
\end{equ}
and such that
as $\eps\to 0$, $\hat c^{(\eps)}$ converges to a finite limit $\hat c$.
Furthermore, as $\eps\to 0$ the right hand side of \eref{e:equ-chat}
with $\hat\kappa^{(\eps)}$ (which is defined in \eqref{e:def-kappa-hat}) in place of $\kappa_2^{(\eps)}$
also
converges to $\hat c$. 
Finally,
with this family of constants $\hat c^{(\eps)}$, one has
$C_0^{(\eps)} \approx \eps^{-1} +1 $ and
the $O(1)$ correction  is zero if $\kappa_2$
is even in the space variable.
\end{lemma}

We remark that \eref{e:equ-chat} is indeed an {\it equation} for $\hat c^{(\eps)} $
since on the right hand side $\kappa_2^{(\eps)}$ depends on $v_h = 4\lambda^2\hat c^{(\eps)}$ (see \eqref{e:vxvy} and \eqref{e:def-zeta-eps}).
The proof of this lemma relies crucially on the following result.

\begin{lemma}\label{lem:bddfcn}
The function $\Theta \colon \R^2  \to \R$ given by
\begin{equ}
\Theta(z) \eqdef
\int K'(x)K'(x-y)K'(y-z) \,dx\,dy
\end{equ}
is bounded.
\end{lemma}

\begin{proof}
Since $K'$ is an odd function in the space variable,
the function $K_2(x) \eqdef \int K'(x-y)K'(y) \,dy$ is even in the space variable.
We decompose the integral $\Theta(z)=\int K'(x)K_2(x-z)\,dx$
as integrals over three domains (similarly with the proof of \cite[Lemma~10.14]{Regularity}).
The first domain is $\Omega_1\eqdef \{|x| < 2|z|/3\}$. On this domain,
one has $|K_2 (x-z)| \lesssim |x-z|^{-1} \lesssim |z|^{-1}$,
and the integration of $|K'(x)|$ over this domain is bounded by $|z|$,
therefore the integration over $\Omega_1$ is bounded.
The second domain is $\Omega_2\eqdef\{|x-z| < 2|z|/3\}$, and the proof is analogous.

On the last domain $\Omega_3 = \R^2 \setminus (\Omega_1\cup \Omega_2)$,
if we replace $K_2(x-z)$ by $K_2(x)$, we have 
$|\int_{\Omega_3} K'(x)K_2(x) \,dx| \lesssim 1$. Indeed, since the integrand
is odd in the space variable and $\Omega_1$ is a  domain
that is symmetric under the space reflection, one has 
$\int_{\R^2\setminus \Omega_1} K'(x)K_2(x) \,dx =0$;
furthermore, one has $\int_{\Omega_2} |K'(x)K_2(x)| \,dx \lesssim 1$, so the claim follows. The difference
$|K_2(x-z) -K_2(x)|$ is bounded by $|z| |x|^{-2}$ using the gradient theorem,
so the error caused by the replacement
is bounded by $|z|\int_{\Omega_3} |K'(x)|  |x|^{-2} \,dx \lesssim |z||z|^{-1}= 1$.
Collecting all the bounds above completes the proof for 
the boundedness of $\Theta(z)$.
%
\end{proof}

\begin{proof}[of Lemma~\ref{lem:const-correct}]
We can rewrite \eref{e:equ-chat} as
\begin{equ}[e:rewrite]
\hat c^{(\eps)} = \int \Theta(t,x) \, \hat \kappa^{(\eps)}(t,x- 4\lambda^2 \hat c^{(\eps)} t) \,dt\,dx \eqdef F_\eps(\hat c^{(\eps)})\;.
\end{equ}
It follows immediately from Lemma~\ref{lem:bddfcn} and the properties of $\hat\kappa^{(\eps)}$ 
that the function $F_\eps$ is bounded, uniformly 
in $\eps$ and in $\hat c^{(\eps)}$.

By Assumptions~\ref{assump-0} and \ref{assump-mix}, for sufficiently small $\eps>0$, 
$\d_x \hat \kappa^{(\eps)}$ is bounded by $C\eps^{-4} \one_\eps$ for some constant $C$, 
where $\one_\eps$ is the characteristic function for
the parabolic ball of radius $\eps$. Combining this with the fact that $|t| \lesssim \eps^2$ for
$(t,x)$ in the support of $\hat\kappa^{(\eps)}$, we conclude that 
$F_\eps'$ is bounded by $C\eps$ for some $C$, again uniformly in $\eps < 1$ 
and in $\hat c^{(\eps)}$.
It follows that, provided that $\eps$ is sufficiently small, so that $|F_\eps'| < 1/2$,
 \eqref{e:rewrite} has indeed a unique solution (because the function 
 $x-F_\eps(x)$ is monotonically increasing in $x$ with derivative bounded below by a strictly positive number), and that this solution
converges to 
\begin{equ}
F_\star = \lim_{\eps \to 0} F_\eps(0)\;,
\end{equ}
as $\eps \to 0$.
In fact, the existence of this limit $F_\star$ is straightforward to show.
The right hand side of \eref{e:rewrite} can be rewritten as 
$F_\eps(0)$ plus an error which is bounded by 
$|\int_0^{\hat c^{(\eps)}} F'_\eps(x) d x| \lesssim \eps |\hat c^{(\eps)}|$.
Since $F_\eps$ is bounded uniformly in $\eps$, so is $\hat c^{(\eps)}$;
therefore this error vanishes, and taking limits on both sides of  \eref{e:rewrite}
shows that $\hat c^{(\eps)}$ converges to $F_\star$.

Regarding the statement for $C_0^{(\eps)}$,
by the scaling argument in \cite{KPZJeremy},
if we define $\hat C_0^{(\eps)}$ to be the same as $C_0^{(\eps)}$
except that $\kappa_2^{(\eps)}$ is replaced by $\hat\kappa^{(\eps)}$,
then $\hat C_0^{(\eps)} \approx \eps^{-1}$. Note that
\begin{equ} [e:O1correction]
C_0^{(\eps)} - \hat C_0^{(\eps)}
=  \int_0^{4\lambda^2 \hat c^{(\eps)} t} \!\!\! \int_{\R^4}
K'(x) K'(x-z) \, \partial_x \hat\kappa^{(\eps)}(t,x-\theta)\, d\theta \,dxdz
\end{equ}
and that $|\partial_x\hat\kappa^{(\eps)}| \lesssim \eps^{-4}\one_\eps$,
$|t|\lesssim \eps^{2}$ on the support of $\hat\kappa^{(\eps)} $,
and $\hat c^{(\eps)}$ is bounded,
$C_0^{(\eps)} - \hat C_0^{(\eps)}$
converges to a finite limit.
Therefore there exists a constant $Q$ such that 
\begin{equ}
\lim_{\eps\to 0} |C_0^{(\eps)}-C_0\eps^{-1} -Q|
=0 \;.
\end{equ}
If $\kappa_2$ is even in the space variable, then since $K'$ is odd
in the space variable one has $F_\eps(0)=0$,
so $\lim_{\eps\to 0} \hat c^{(\eps)} =0$
and the order $1$ correction $Q$ vanishes,
thus completing the proof.
\end{proof}

\begin{remark}
In fact, the right hand side of \eref{e:O1correction}
divided by $\lambda^2$ converges to a limit independent
of $\lambda$, which can be shown in an analogous way
as the fact that $\lim_{\eps\to 0}F_\eps(\theta)$
does not depend on $\theta$.
Since $C_0^{(\eps)}$ appears in the renormalised
equation \eref{e:renom-eq} in the form $\lambda C_0^{(\eps)}$ (see \eref{e:choiceell} below), it turns out that
the $O(1)$ correction to $C_0^{(\eps)}$  contributes to the term $\lambda^3 c$
in \eref{e:vxvy}.
\end{remark}

Regarding the other renormalisation constants, when $v_h=0$, it was also shown in \cite[Theorem~6.5]{KPZJeremy} that
there exists a constant $\bar c\in\R$  depending on both $\kappa_2$
 and the choice of the cutoff kernel $K$
such that $\lim_{\eps\to 0} (4C_{2,1}^{(\eps)} + C_{3,1}^{(\eps)}) =\bar c$.
The prefactors in front of the logarithmically divergent constants were
obtained in \cite{KPZ}.
Following  similar arguments as in \cite{WongZakai} and \cite{KPZJeremy},
one has  $C_1^{(\eps)} \approx \eps^{-1/2}$, with the limiting prefactor
$C_1$ equal to an integral which can be represented by the same graph
as the one for $C_1^{(\eps)}$, except that each arrow is understood as $P'$ and the red triangle is interpreted as $\kappa_3$ (i.e. setting $\eps=1$).
It is easy to see that this integral is finite since $\kappa_3$ is 
continuous and $P'$ is integrable,
in particular $C_1$ does not depend on the choice of $K$. 
In a similar way $C_{2,2}^{(\eps)}$ and $C_{3,2}^{(\eps)}$ converge to finite constants 
$C_{2,1}$, $C_{3,2}\in\R$
depending on $\kappa_4$, but independent of the choice of the kernel $K$.
These are all consequences of the scaling property of the heat kernel
and the fact that $K$ converges to heat kernel in suitable spaces
under parabolic scaling.
Following analogous arguments
as in the proof of Lemma~\ref{lem:const-correct},
one can show that  the next order corrections to these constants 
due to the nontrivial shift $v_h$
are $O(\eps^{1/2})$
and thus all vanish in the limit.


The specification of these renormalisation constants 
(in particular with $\hat c^{(\eps)}$ given by Lemma~\ref{lem:const-correct})
determines an element $M_\eps\in\mathfrak{R}$ 
as in \eqref{e:defLi} by setting
\begin{equ}[e:choiceell]
\ell_1 = C_0^{(\eps)}\;,\quad
\ell_2 = \hat c^{(\eps)} \;,\quad
\ell_3 = C_1^{(\eps)}\;,\quad
\ell_4 = C_2^{(\eps)}\;,\quad
\ell_5 = C_3^{(\eps)}\;.
\end{equ}
Given $\eps > 0$ and a realisation of $\zeta_\eps$, this yields a renormalised model 
$(\hat \Pi^{(\eps)}, \hat \Gamma^{(\eps)})$ by acting on
 the canonical model $( \Pi^{(\eps)},  \Gamma^{(\eps)})$ lifted from $\zeta_\eps$
by $M_\eps$.


\begin{remark}
One may wonder whether the constants $C_{2,2}^{(\eps)}$ and $C_{3,2}^{(\eps)}$ which involve the fourth cumulants 
{\it have to} be subtracted from 
the objects $\<4>$ and $\<211>$. Note that although their homogeneities are zero, they  
converge to finite constants rather than diverge logarithmically.
To obtain some limit, one does of course not have to subtract them, but
we do subtract them here in order to obtain a limit which is independent (in law) of
the choice of process $\zeta$.
\end{remark}

\section{A priori bounds on the renormalised models}  \label{sec:tightness}

In the previous section we have defined the renormalised models 
$\hat Z_\eps = (\hat \Pi^{(\eps)}, \hat \Gamma^{(\eps)})$. The goal of 
this section is to obtain
uniform bounds on the moments of this family of renormalised models. 
Actually, we will obtain a bit more, namely we obtain bounds also
on a smoothened version of this model
which will be useful in Section~\ref{sec:identify}
when we show convergence of the models $\hat Z_\eps$ to a limiting Gaussian model. 
More precisely, we define a second family of
renormalised random models $\hat Z_{\eps,\bar\eps} = (\hat \Pi^{(\eps,\bar\eps)}, \hat \Gamma^{(\eps,\bar\eps)})$
in the same way as $\hat Z_\eps$, except that 
the field  $\zeta_\eps$  used to build $\hat Z_\eps$ is replaced by
$ \zeta_{\eps,\bar\eps} \eqdef \zeta_\eps*\rho_{\bar\eps}$,
where $\rho$ is a compactly supported  smooth function on $\R^2$ integrating to $1$ that is 
even in the space variable, and
$\rho_{\bar\eps} = \bar\eps^{-3} \rho (\bar\eps^{-2} t, \bar\eps^{-1} x)$.

The renormalisation constants used to build the models
$\hat Z_{\eps,\bar\eps}$ are still defined as in Section~\ref{sec:renorm},
except that they are now dependent on cumulants of $\zeta_{\eps,\bar\eps}$,
so we denote them by $\hat c^{(\eps,\bar\eps)}$ 
and $C_i^{(\eps,\bar\eps)}$ with $i=0,\ldots,3$.
Note that while $\hat c^{(\eps)}$ was obtained implicitly by 
solving the equation \eqref{e:equ-chat}, we now consider $\zeta_\eps$ (and therefore $v_h$) 
as given and view \eqref{e:equ-chat} as the definition of $\hat c^{(\eps)}$ with 
given right hand side.

For instance, $C_{2,1}^{(\eps,\bar \eps)}$ and $\hat c^{(\eps,\bar \eps)}$ are given by
\begin{equ}
C_{2,1}^{(\eps,\bar\eps)} \eqdef \;
\begin{tikzpicture}  [baseline=20]
\node[root]	(root) 	at (0,0) {};
\node[dot]		(mid)  	at (0,1) {};
\node[dot]		(top)  	at (0,2) {};
\node[cumu2n]	(left)  		at (-.8,1) {};	
	\draw[cumu2b] (left) ellipse (4pt and 8pt);
\node[cumu2n]	(right)  		at (.8,1.5) {};	
	\draw[cumu2b] (right) ellipse (4pt and 8pt);
\draw[kernel] (mid) to  (root);
\draw[kernel] (top) to  (mid);
\draw[kernel] (left.south) node[dot]   {} to  (root);
\draw[kernel] (left.north) node[dot]   {} to  (top);
\draw[kernel] (right.north) node[dot]   {} to  (top);
\draw[kernel] (right.south) node[dot]   {} to  (mid);
\end{tikzpicture}
=
\begin{tikzpicture}  [baseline=20]
\node[root]	(root) 	at (0,0) {};
\node[dot]		(mid)  	at (0,1) {};
\node[dot]		(top)  	at (0,2) {};
\node[cumu2n]	(left)  		at (-.8,1) {};	
	\draw[cumu2] (left) ellipse (4pt and 8pt);
\node[cumu2n]	(right)  		at (.8,1.5) {};	
	\draw[cumu2] (right) ellipse (4pt and 8pt);
\draw[kernel] (mid) to  (root);
\draw[kernel] (top) to  (mid);
\draw[kernel,dashed] (left.south) node[dot]   {} to  (root);
\draw[kernel,dashed] (left.north) node[dot]   {} to  (top);
\draw[kernel,dashed] (right.north) node[dot]   {} to  (top);
\draw[kernel,dashed] (right.south) node[dot]   {} to  (mid);
\end{tikzpicture}
\qquad
\hat c^{(\eps,\bar\eps)} \eqdef
\begin{tikzpicture}  [baseline=10]
\node[root]	(root) 	at (0,0) {};
\node[dot]		(left)  	at (0,1.2) {};
\node[cumu2n]	(right)  		at (.8,.6) {};	
	\draw[cumu2b] (right) ellipse (4pt and 8pt);	
\draw[kernel] (left) to  (root);
\draw[kernel]   (right.south)  node[dot] {}  to (root) ;
\draw[kernel]   (right.north)  node[dot] {} to (left);
\end{tikzpicture} 
=
\begin{tikzpicture}  [baseline=10]
\node[root]	(root) 	at (0,0) {};
\node[dot]		(left)  	at (0,1.2) {};
\node[cumu2n]	(right)  		at (.8,.6) {};	
	\draw[cumu2] (right) ellipse (4pt and 8pt);	
\draw[kernel] (left) to  (root);
\draw[kernel,dashed]   (right.south)  node[dot] {}  to (root) ;
\draw[kernel,dashed]   (right.north)  node[dot] {} to (left);
\end{tikzpicture} 
\end{equ}
respectively, where the dashed arrows represent $K'_{\bar\eps}$, with $K_{\bar\eps} = K*\rho_{\bar\eps}$,
and \begin{tikzpicture}[baseline=-3.2]
\node		(mid)  	at (0,0) {};
\draw[cumu2b] (mid) ellipse (8pt and 4pt);
\node[dot] at (mid.west) {};
\node[dot] at (mid.east) {};
\end{tikzpicture} represents the covariance of $\zeta_{\eps,\bar \eps}$.
(This covariance depends on $v_h = \hat c^{(\eps)}$, but not
on $\hat c^{(\eps,\bar\eps)}$.)
It is straightforward to verify that the kernel $K_{\bar\eps}$
approximates the kernel $K$ in the sense that
\begin{equ}[e:approxKernel]
| K(z)-K_{\bar\eps}(z) | \lesssim \bar\eps^\eta |z|^{-1-\eta}\;,\qquad
| K'(z)-K'_{\bar\eps}(z) | \lesssim \bar\eps^\eta |z|^{-2-\eta}\;,
\end{equ}
for some sufficiently small $\eta>0$.
Our main result is the following a priori bound on  $\hat Z^{(\eps)}$ and 
on its difference with $\hat Z^{(\eps,\bar\eps)}$.

\begin{theorem} 
\label{thm:tightness}
Let $(\hat \Pi^{(\eps)}, \hat \Gamma^{(\eps)})$ and $ (\hat \Pi^{(\eps,\bar\eps)}, \hat \Gamma^{(\eps,\bar\eps)})$ be as above. 
There exist $\kappa,\eta>0$   such that 
for every $\tau\in\{\Xi,\<2>,\<21>,\<21a>,\<4>,\<211>\}$
 and every   $p>0$,
one has
\begin{equ}  [e:tightness]
\E | (\hat \Pi^{(\eps)}_x \tau) (\varphi_x^\lambda) |^p  
	\lesssim \lambda^{p (|\tau| + \eta)} \;,
\qquad
\E | (\hat\Pi^{(\eps)}_x \tau  -  \hat\Pi^{(\eps,\bar\eps)}_x \tau ) (\varphi_x^\lambda) |^p  
	\lesssim
	\bar\eps^\kappa
	 \lambda^{p (|\tau| + \eta)} 
\end{equ}
uniformly in all $\eps,\bar\eps\in(0,1]$, all $\lambda\in(0,1]$, all test functions $\varphi\in\CB$, and all $x\in\R^2$.
\end{theorem}

\begin{remark}
The first bound of \eref{e:tightness} actually implies tightness for the family of renormalised
models $\hat Z_\eps =(\hat \Pi^{(\eps)}, \hat \Gamma^{(\eps)})$.
The proof for this Kolmogorov type tightness criterion 
is similar to that of \cite[Theorem 10.7]{Regularity}.
Since we do not really need to know tightness here, we will not elaborate on this point.
\end{remark}

Note that although $\<Kxi>$ and $\<K2>$ both have negative homogeneities, 
they are of the form $\CI(\tau)$, so that the corresponding bounds follow 
from the extension theorem \cite[Theorem~5.14]{Regularity}, combined with the analogous bounds on $\tau$.
The remainder of this section is devoted to the proof of \eref{e:tightness}
for each $\tau\in\{\<2>,\<21>,\<21a>,\<4>,\<211>\}$.
In fact we will mainly focus on the first bound of \eref{e:tightness},
and the second bound will follow analogously.
We are going to use extensively the notion of Wick product for a collection of (not necessarily Gaussian) random variables which is defined as follows.

\begin{definition} \label{def:Wick-product}
Given a collection $\CX = \{X_\alpha\}_{\alpha \in \CA}$ of random variables as before, 
the Wick product 
$\Wick{X_A}$ for $A \subset \CA$ is defined recursively by setting $\E \Wick{X_{\emptyset}} = 1$ 
and postulating that
\begin{equ} [e:defWick]
X^A  = 
\sum_{B \subset A}  \Wick{X_B}  \sum_{\pi \in \CP(A \setminus B)} 
	  \prod_{\bar B \in \pi}
		 \Cum \big(X_{\bar B}\big)    \;.
\end{equ}
We will also sometimes write $\Wick{\prod_{i\in B} X_i} \eqdef \Wick{X_B}$,
keeping in mind that  $\Wick{\prod -}$ is really {\it one single} notation
rather than composition of two consecutive operations.
\end{definition}

Note that by \eref{e:mome2cumu}
we could also write 
$X^A  =\sum_{B \subset A}  \Wick{X_B} \E(X^{A\setminus B})$,
but \eref{e:defWick} is actually the identity
frequently being used in the paper.
Again, \eqref{e:defWick} is sufficient to define $\Wick{X_A}$
by recursion over the size of $A$. Indeed, the term with 
$B=A$ is precisely the quantity we want to define, and 
all other terms only involve Wick products of strictly smaller sets of random variables.
By the definition we can easily see that as soon as $A \neq \emptyset$, one
always has $\E \Wick{X_A}=0$.
Note also that if we take expectations on both sides of \eref{e:defWick},
then all the terms with $B \neq \emptyset$ vanish
and we obtain exactly the identity \eref{e:mome2cumu}.
For example, given centred random variables $X_i$, we have 
$\Wick{X_i}=X_i$, $\Wick{X_1 X_2}=X_1 X_2-\E(X_1 X_2)$,
\begin{equ} 
\Wick{X_1 X_2X_3}
	=X_1 X_2X_3
		-\sum_{i\neq j\neq k} X_i \,\E\big(X_j X_k\big)
		-\E(X_1 X_2X_3) \;.
\end{equ}

We refer to \cite{AvramTaqqu1987}, \cite[Appendix~B]{AvramTaqqu2005}
or \cite{LukkarinenMarcozzi}
for the properties of non-Gaussian Wick products.
The following result (Lemma~\ref{lem:WickExp}) is often called ``diagram formula" in the literature. 
As far as we know, the original proof of this particular statement 
was given by \cite{Surgailis1983} and \cite{GiraitisSurgailis1986}, but 
similar statements were known long before that.
Before stating the lemma, we define an important type of partition.

\begin{definition} \label{def:partition_M}
Let $M\times P$ be the Cartesian product of two sets $M$ and $P$, and $D\subset M\times P$.
We say that $\pi \in \CP_M(D)$ if $\pi \in \CP(D)$ and 
for every $B\in \pi$, there exist $(i,k),(i',k')\in B$ such that $k\neq k'$ (in particular $|B|>1$).
\end{definition}
In plain words, the requirement says that  for each $B\in\pi$,
the elements in $B$ can not all have the same ``$P$-index".
In most cases, we are interested in the situation $D=M\times P$.

\begin{lemma} \label{lem:WickExp}
Let $M\eqdef \{i: 1\le i\le m\}$ and $P\eqdef \{k: 1\le k\le p\}$.
Let $X_{(i,k)}$ where $1\le i\le m$, $1\le k\le p$ be centred random variables 
with bounded moments of all orders. One has
\begin{equ} [e:WickExp]
\E \Big( \prod_{k=1}^p \Wick{\prod_{i=1}^m X_{(i,k)}}  \Big)
	= 	 \sum_{\pi \in\CP_M(M\times P)}  \prod_{B\in\pi} \Cum \big( X_B \big) \;.
\end{equ}
\end{lemma}


Note that if the Wick products were replaced by usual products (say in the case $m=1$),
one would recover the definition of cumulants \eref{e:mome2cumu}.
In the case of centred Gaussians, only the partitions $\pi$ with $|B|=2$ for each $B\in \pi$ is allowed since joint cumulants of higher order vanish, and the statement recovers the fact that \eref{e:WickExp} can be computed by summing over all the ways of pairwise contracting $|M\times P|$ variables excluding ``self-contractions" (i.e. factors of the form $\Cum (X_{(i,k)},X_{(i',k)})$).

As an immediate consequence of Lemma~\ref{lem:WickExp} one has
\begin{corollary} \label{cor:Wick-field}
For every $p\ge 1$ and $m\ge 1$, one has
\begin{equ} [e:Wick-field]
\E \Big( \prod_{k=1}^p \Wick{ \zeta_\eps(x_k^{(1)}) \cdots \zeta_\eps(x_k^{(m)})}  \Big)
	= 	 \sum_{\pi \in \CP_M(M\times P)}  \prod_{B\in\pi} 
	\kappa^{(\eps)}_{|B|} \Big(\{x^{(i)}_k:(i,k)\in B\} \Big)
\end{equ}
where $M\eqdef \{i: 1\le i\le m\}$ and $P\eqdef \{k: 1\le k\le p\}$.
\end{corollary}

From now on we define a further simplified notation. Given a family of space-time points $\{x_\alpha : \alpha\in B\}$ parametrised by the index set $B$, we write
\begin{equ} 
\kappa^{(\eps)}_{|B|} (B) \eqdef
	\kappa^{(\eps)}_{|B|} \big(\{x_\alpha:\alpha\in B\} \big) \;.
\end{equ}
Here, $\alpha$ can be multi-indices, for instance $x_\alpha $ would be 
$x_k^{(i)}$ in the situation of \eref{e:Wick-field}.

According to Lemma~\ref{lem:cumu-close}, some factors $\kappa^{(\eps)}_{|B|}$  in \eref{e:Wick-field} may actually vanish.
In fact, a partition $\pi$ encodes the information that  the space-time points 
that are arguments of the function  $\kappa^{(\eps)}_{|B|}$ for each element $B\in\pi$
must be ``close to each other" in order to have non-zero contribution to \eref{e:Wick-field}.

\subsection{First object and some generalities}

We now have all the ingredients in place for the proof of Theorem~\ref{thm:tightness}.
For the moment, we only show how to obtain the first bound of \eqref{e:tightness}; we include
a short discussion on how to obtain the second bound at the end of this section. 
We start with the case $\tau = \Xi$, for which where we need to show that
$ \lambda^{|\s|p/2} \E |\zeta_\eps(\varphi^\lambda_0)|^p \lesssim 1$ 
uniformly in $\eps,\lambda>0$ for all $p>1$.
As a consequence of Corollary~\ref{cor:Wick-field} with $m=1$, one has
\begin{equ}
\E |\zeta_\eps(\varphi^\lambda_0)|^p
= \sum_{\pi \in\CP(\{1,\ldots,p\})} \prod_{B\in\pi} 
	\Big( \int \kappa_{|B|}^{(\eps)}(B) \prod_{i\in B} \varphi^\lambda_0 (x_i)
	\prod_{i\in B} dx_i \Big) \;.
\end{equ}
We bound the integral in the bracket for each $B\in\pi$, $|B|>1$. 
If $\eps>\lambda$, we bound 
$|\kappa_{|B|}^{(\eps)}| \lesssim \eps^{-|B||\s|/2} \lesssim \lambda^{-|B||\s|/2} $,
and the fact that the integral of the test functions $\varphi^\lambda_0$ are uniformly bounded.
If on the other hand $\eps\le \lambda$, then we use the fact that the integrand vanishes unless all
points $x_i$ are within a ball of radius $C\eps$ for some fixed $C>0$, and this ball is centred around
a point in the ball of radius $\lambda$ around the origin. The volume of this region is bounded by
$C\eps^{(|B|-1)|\s|}\lambda^{|\s|}$. 
Since the integrand on the other hand is bounded by $C \eps^{-|B||\s|/2}  \lambda^{-|B||\s|}$,
it follows that the integral is bounded by 
 $C\eps^{(|B|/2-1)|\s|}\lambda^{(1-|B|)|\s|}
\lesssim \lambda^{-|B||\s|/2}$. Since $\sum_{B\in \pi} |B|=p$, the desired bound follows.

We now turn to the simplest non-trivial object, namely $\tau=\<2>$. Using Definition~\ref{def:Wick-product} 
to rewrite the product $\zeta_\eps(x^{(1)}) \zeta_\eps(x^{(2)}) $
as $\Wick{\zeta_\eps(x^{(1)}) \zeta_\eps(x^{(2)}) } 
	+ \kappa^{(\eps)}_2(x^{(1)},x^{(2)})$, 
together with 
the definition of $C_0^{(\eps)}$, one has 
\begin{equ} [e:1stObj]
(\hat\Pi_{0}^{(\eps)} \<2>) (\varphi^\lambda_0)
	=
	\int \varphi^\lambda_0(z)\, K'( z-x^{(1)}) K'(z-x^{(2)}) 
		 \Wick{\zeta_\eps(x^{(1)}) \zeta_\eps(x^{(2)})} \,dx^{(1)} dx^{(2)} dz \;.
\end{equ}
The term containing $ \kappa^{(\eps)}_2$ is precisely canceled by $C_0^{(\eps)}$.
The $p$-th moment for any even number $p$ of the above quantity 
can be written as 
\begin{equ}
	\int \prod_{k=1}^p \Big(  
		\varphi^\lambda_0(z_k)\, K'( z_k-x^{(1)}_k) K'(z_k-x^{(2)}_k)  \Big) 
	 \E \prod_{k=1}^p \Big(
		\Wick{ \zeta_\eps(x^{(1)}_k) \zeta_\eps(x^{(2)}_k) } \Big)
	 \,dx\, dz \;,
\end{equ}
where the integration is over all $x_k^{(i)}$ and all $z_k$
with $i\in\{1,2\}$ and $1\le k\le p$.
Applying Corollary~\ref{cor:Wick-field}, one has the bound
\begin{equs}
\E &|(\hat\Pi_{0}^{(\eps)} \<2>) (\varphi^\lambda_0)|^p \label{e:bnd2-1}\\
	&\lesssim \sum_\pi
	\int \prod_{k=1}^p \Big| 
		\varphi^\lambda_0(z_k)\, K'( z_k-x^{(1)}_k) K'(z_k-x^{(2)}_k)  \Big|
	 \prod_{B\in\pi} |\kappa^{(\eps)}_{|B|}(B)|
	 \,dx\,dz\;, 
\end{equs}
where $\pi$ runs over partitions in $\CP_{\{1,2\}}(\{1,2\}\times \{1,\ldots,p\})$. 

As illustrations, the left picture below shows the situation for $p=2$ 
and $\pi $ being the particular partition consisting of only one element 
which is the whole set of cardinality $4$.
The right picture below
shows the situation for $p=4$ with a {\it prohibited} partition $\pi$ (namely, $\pi \in \CP(\{1,2\}\times \{1,\ldots,p\})$ but $\pi\notin \CP_{\{1,2\}}(\{1,2\}\times \{1,\ldots,p\})$),
i.e.\ one that does \textit{not} appear in the sum \eqref{e:bnd2-1}.
Graphically, the fact that the partition is prohibited is reflected in the presence 
of a subgraph identical to the one representing the diverging constant $C_0^{(\eps)}$. 
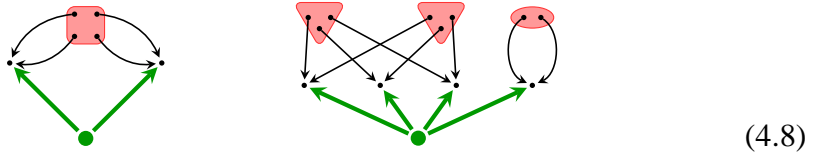
\begin{equ}[e:pictures]
\begin{tikzpicture} 
\node[root]	(root) 	at (0,0) {};
\node		(mid)  	at (0,1.5) {};
	\node[cumu4]	(mid-cumu) 	at (mid) {};
\node[dot]		(left)  	at (-1,1) {};
\node[dot]		(right)  	at (1,1) {};
\draw[testfcn] (left) to  (root);
\draw[testfcn] (right) to  (root);
\draw[->,kernel,bend left=30] (mid.south west) node[dot] {} to  (left);
\draw[->,kernel,bend right=30] (mid.south east) node[dot] {} to  (right);
\draw[->,kernel,bend right=30] (mid.north west) node[dot] {} to  (left);
\draw[->,kernel,bend left=30] (mid.north east) node[dot] {} to  (right);
\end{tikzpicture}
\qquad\qquad
\begin{tikzpicture} 
\node at (-.8, .6) (11) {};  
	\node[cumu3,rotate=180] at (11) {};
\node at (.8, .6) (12) {};
	\node[cumu3,rotate=180] at (12) {};
\node at (0,-.3) [dot] (mid) {}; 
\node at (-1,-.3) [dot] (left) {}; 
\node at (1,-.3) [dot] (right) {};
\node at (2,-.3) [dot] (rightmost) {};
\node[cumu2n]	(a)  at (2, .6) {};	\draw[cumu2] (a) ellipse (8pt and 4pt);
\draw[kernel] (a.east) node[dot]   {} to [bend left = 60] (rightmost);
\draw[kernel] (a.west) node[dot]   {} to [bend right = 60] (rightmost);
\draw[kernel] (11.south) node[dot] {} to (mid); 
\draw[kernel] (11.west) node[dot] {} to (left); 
\draw[kernel] (11.east) node[dot] {} to (right); 
\draw[kernel] (12.south) node[dot] {} to (mid); 
\draw[kernel] (12.west) node[dot] {} to (left); 
\draw[kernel] (12.east) node[dot] {} to (right); 
\node at (.5,-1) [root] (root) {};
\draw[testfcn] (left) to (root);
\draw[testfcn] (mid) to (root);
\draw[testfcn] (right) to (root);
\draw[testfcn] (rightmost) to (root);
\end{tikzpicture}
\end{equ}
Here, the interpretation of these graphs is as in Section~\ref{sec:renorm}, but with 
the green arrow $\tikz [baseline=-3] \draw[testfcn] (1,0) to (0,0);$ 
representing the test function $\varphi_0^\lambda$.

The right hand side of \eref{e:bnd2-1} involves an integration of functions 
of the following general form; 
we formulate it in a general way in order to deal with other objects $\tau\in\CW_-$ as well. 
We are given a finite set of space-time points
$B=\{x_1,\ldots,x_n\}$, and a cumulant function 
$\kappa^{(\eps)}_{|B|}(B) = \kappa^{(\eps)}_n (x_1,\ldots,x_n)$.
By observing the right hand side of \eref{e:bnd2-1}, 
 one realises that for every $1\le i\le n$, 
 there is a factor $|K'(y_i - x_i)|$ for some space-time point $y_i$
 in the integrand,
which is bounded by $|y_i - x_i |^\alpha$  for some $\alpha$.
   Here, it is allowed that some of the points $y_1,\ldots,y_n$ may coincide. 
In the case of \eref{e:bnd2-1} above, $y_i$ would be one of the points $z_k$
and $\alpha=-2$.
We then have the following result for integrating such a product of cumulants and kernels (see \ref{sec:framework} for definition of a scaling $\s$ of $\R^d$). 

\begin{lemma}  \label{lem:collapse}
Given a scaling $\s$ of $\R^d$, $n$ points $y_1,\ldots,y_n \in\R^d$, and real numbers $\alpha_1,\ldots,\alpha_n$, assuming  $-|\s|<\alpha_i<0$ for all $1\le i\le n$, one has
\begin{equs} 
\int
\prod_{i=1}^n |y_i-x_i|^{\alpha_i}
|\kappa^{(\eps)}_n (x_1, & \ldots,  x_n)|
  \,dx_1\cdots dx_n \label{e:reduce-JM}\\
&\lesssim \eps^{(n/2-1)|\s|}
\int_{\R^d}
 	\prod_{i=1}^n  
	\big(| y_i-x |+\eps\big)^{\alpha_i}
\, dx \;,
\end{equs}
uniformly in $y_1,\ldots,y_n \in\R^d$ and $\eps>0$.
Here, the integral on the left hand side is over $n$ variables,
and the integral on the right hand side is over only one  variable.
\end{lemma}

\begin{remark}
Lemma~\ref{lem:collapse} is one of the main technical results in this article and will be frequently applied.
The intuitive ``graphical'' meaning of this lemma is that for all practical purposes, we
can ``collapse'' the $p$ points in an expression of the type
\begin{tikzpicture}[baseline=-4]
\node		(mid)  	at (0,0) {};
	\node[cumu4]	(mid-cumu) 	at (mid) {};
\node[dot] at (mid.north west) {};
\node[dot] at (mid.south west) {};
\node[dot] at (mid.north east) {};
\node[dot] at (mid.south east) {};
\end{tikzpicture}
(here $p=4$) into one single point, and this operation then generates a 
factor $\eps^{(p/2-1)|\s|}$. It is because this exponent is positive as soon
as $p > 2$ that our bounds behave ``as if'' everything was Gaussian.
\end{remark}

\begin{remark}\label{rem:homogen}
If we consider $\kappa_n$ to have homogeneity $-n|\s|/2$, $|y-x|^\alpha$ to have
homogeneity $\alpha$, $\eps$ to have homogeneity $1$, and each integration 
variable to have homogeneity $|\s|$,
then the total homogeneity is identical on both sides of \eqref{e:reduce-JM}.
(It is equal to $n|\s|/2 + \sum\alpha_i$ in both cases.)
\end{remark}

\begin{proof}
First of all, 
we define $\one_a (x) = 1$ if $|x|<a$ and vanishes if $|x|\ge a$.
Recall then that one has the bound 
\begin{equ}  [e:bndones]
|\kappa^{(\eps)}_n (x_1,\ldots,x_n)|
 \lesssim \eps^{(n/2-1)|\s|} \sum_T \prod_{\{x_i,x_j\}\in\CE(T)}
\eps^{-|\s|} \one_\eps(x_i -x_j)\;,
\end{equ}
uniformly in $\eps>0$,
where $T$ runs over all spanning trees of the set $\{x_1,\ldots,x_n\}$,
and $\CE(T)$ denotes the edge set of $T$.
Indeed, for each $T$, since $|\CE(T)|=n-1$, the total power of $\eps$ on the right hand side is $(n/2-1)|\s| -(n-1)|\s| = -n|\s|/2$ which is consistent with the naive bound $|\kappa^{(\eps)}_n| \lesssim \eps^{-n|\s|/2}$.
By Lemma~\ref{lem:cumu-close}, $\kappa^{(\eps)}_n=0$ unless there exists a connected graph over $\{x_1,\ldots ,x_n\}$ such that $|x_i - x_j|<\eps$ whenever
 $\{x_i,x_j\}$ is an edge of the graph. The bound \eref{e:bndones} is then obtained by taking a spanning tree of this graph.

Let $T$ be a fixed tree in the above sum and $e=\{x_i,x_j\}$ be an arbitrary  fixed edge of $T$.
There exists a constant $c>1$ such that 
\begin{equ} [e:funstep]
\eps^{-|\s|} \one_\eps(x_i -x_j) \lesssim 
	\int_{\R^d} \eps^{-|\s|} \one_{c\eps}(x_i - x) \cdot \eps^{-|\s|} \one_{c\eps}(x-x_j) \,dx \;.
\end{equ}
uniformly in $\eps>0$.
This is obviously true for $\eps=1$, and the case for general $\eps>0$ follows by rescaling.
We can therefore add a new ``dummy variable'' $x$ representing the integration variable in
\eqref{e:funstep}, which suggests to define a new tree $T'$ by $T' = T\cup\{x\}$ and 
to replace the edge $e$ by the two edges $\{x_i,x\}$ and $\{x,x_j\}$. In other words, one
sets $\CE(T')=(\CE(T)\backslash \{e\}) \cup \{\{x_i,x\},\{x,x_j\}\}$.

Writing now $\one_\eps(e) \eqdef \one_\eps(x_i,x_j) $ as a shorthand 
if $e=\{x_i,x_j\}$, we note that 
$\prod_{e\in\CE(T')} \eps^{-|\s|} \one_{c\eps}(e)$
 is the indicator function of the event ``every edge of $T'$ has length at most $c\eps$''. Since the graph $T'$ has finite diameter, this 
 implies that every vertex of $T'$ is at distance at most $C\eps$ from the fixed vertex $x$, for
 some fixed $C>0$, so that
\begin{equs}
\prod_{e\in\CE(T)}
 \eps^{-|\s|} \one_\eps(e) 
& \le
\int_{\R^d} \prod_{e\in\CE(T')} \eps^{-|\s|} \one_{c\eps}(e) \,dx
\le
\int_{\R^d} \prod_{\bar x\in T}   \eps^{-|\s|}  \one_{C\eps} (x-\bar x)  \,dx\;,
\end{equs}
where the first inequality is given by \eref{e:funstep}.

It now remains to substitute this back into \eqref{e:bndones} and into the left hand side of \eqref{e:reduce-JM}.
Integrating over all the variables $x_i \in T' \setminus \{x\}$, the claim follows from the bound
\begin{equ}
\int_{\R^d}
 | y_i-\bar x |^{\alpha_i} \,
 \eps^{-|\s|} \one_{C\eps}(x-\bar x) \,d\bar x
 \lesssim  (| y_i-x | +\eps)^{\alpha_i}\;,
\end{equ}
which can easily be shown by considering separately the two cases $| y_i-x | \le 2C\eps$
and $| y_i-x | \ge 2C\eps$.
\end{proof}

\begin{remark}
In the above proof, if one chose $x$ to be a fixed point in $T$, which would seem more natural 
in principle than the first step \eref{e:funstep},
the resulting bound would have one of the factors on the right hand side of \eref{e:reduce-JM} equal to $| y_i-x |^{\alpha_i}$ (i.e. without the additional $\eps$ appearing in the statement).
These occurrences of $\eps$ in our bound will however be useful in the sequel, as
it will allow us to use the brutal bound $(| y_i-x |+\eps)^{\alpha_i} \le \eps^{\alpha_i}$.
\end{remark}

We introduce the following notation: for any space-time point $x$ and $\eps\ge 0$, let
\begin{equ}
|x|_\eps \eqdef |x|+\eps \;,
\end{equ}
where $|x|$ is as above the norm of $x$ with respect to the parabolic scaling.

We can now apply Lemma~\ref{lem:collapse} to each $\pi$ and each $B\in \pi$ of the right hand side of \eref{e:bnd2-1}, so that the entire right hand side of \eref{e:bnd2-1} is bounded by
\begin{equ} [e:bnd2-2]
\sum_\pi
	\int \Big( \prod_{k=1}^p 
		\varphi^\lambda_0(z_k) \Big)
	\prod_{B\in\pi}
		\Big( \eps^{(|B|/2-1)|\s|}
		\prod_{(i,k)\in B} 
				|z_{k}-x_B|_\eps^{-2}
	 \Big) 
	 \prod_{B\in\pi}dx_B \prod_{k=1}^p dz_k
\end{equ}
where $\pi \in \CP_{\{1,2\}}(\{1,2\}\times \{1,\ldots,p\})$ as above,
and for each $B\in\pi$ we have an integration variable $x_B \in \R^2$.
For example, in the case of the left picture in \eref{e:pictures}, we have
$\pi=\{B\}$ with $B=\{(1,1),(1,2),(2,1),(2,2)\}$,
and the corresponding term in \eqref{e:bnd2-2} can be depicted graphically as
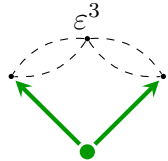
\begin{center}
\begin{tikzpicture}  
\node[root]	(root) 	at (0,0) {};
\node[dot]		(mid)  	at (0,1.5) {};
	\node at (0,1.8) {$\eps^3$};
\node[dot]		(left)  	at (-1,1) {};
\node[dot]		(right)  	at (1,1) {};
\draw[testfcn] (left) to  (root);
\draw[testfcn] (right) to  (root);
\draw[dashed,bend left=30] (mid) to  (left);
\draw[dashed,bend right=30] (mid) to  (right);
\draw[dashed,bend right=30] (mid) to  (left);
\draw[dashed,bend left=30] (mid) to  (right);
\end{tikzpicture}
\end{center}
with dashed lines representing the function $x \mapsto |x|_\eps^{-2}$.

\begin{remark}
Bounds of the form \eref{e:bnd2-2} are crucial to our analysis.
As already noted in Remark~\ref{rem:homogen},  although the number of integration 
variables appearing in \eref{e:bnd2-2} differs for each $\pi$,
a simple power counting shows  that the ``homogeneity" of the summand is 
always the same, namely
\begin{equ}
\sum_{B\in\pi} \Big( (|B|/2-1)|\s| -2|B| \Big)+\sum_{B \in\pi} |\s| 
= -\sum_{B\in\pi} |B|/2 =-p \;.
\end{equ}
In a certain sense, the powers of $\eps$ thus guarantee that the entire quantity 
is still of the correct ``homogeneity".
\end{remark}

In the rest of this subsection we explain some generalities 
for bounding expressions of the type \eref{e:bnd2-2},
and in the next several subsections we will follow the same routine
for the other objects $\tau\in\CW_-$.
Our idea is to apply the general tools developed in Section~\ref{sec:bounds},
which are summarised in Proposition~\ref{prop:main} below.
To apply this proposition to obtain our desired bound, we will associate to $\<2>$
a ``partial graph" $H$ (in general, several such graphs for other elements $\tau\in\CW_-$)  as defined in Definition~\ref{def:H-type} below.
Then, we represent the multi-integral for each partition $\pi$
in the expressions of the type \eref{e:bnd2-2},
which result from application of Corollary~\ref{cor:Wick-field} and Lemma~\ref{lem:collapse},
 a graph  $\CV$ obtained by
``Wick contracting" $p$ copies of $H$, in the sense of Definition~\ref{def:Big-graph}
below.
The set of all allowed partitions in \eref{e:bnd2-2} is then in one-to-one correspondence
with the set of possible ``Wick contractions".
Given an integer $p>1$, a rescaled test function $\varphi^\lambda_0$,
and a partial graph $H$, one then has a number 
$\Lambda^p_{\varphi^\lambda_0} (\CV)$ for every such graph $\CV$
which equals the value of the multi-integral,
and the sum of them over all Wick contractions then yields a number
$\Lambda^p_{\varphi^\lambda_0} (H)$
which provides a bound for the moment \eref{e:tightness}.
In a nutshell, the ``routine" we will always follow is that
\begin{equ} [e:routine]
 \E |(\hat \Pi^{(\eps)}_x \tau) (\varphi_x^\lambda) |^p 
 \stackrel{\eref{e:Wick-field} \eref{e:reduce-JM}}{\lesssim}
 \Lambda^p_{\varphi^\lambda_0} (H)
 \stackrel{{\rm Prop}\;\ref{prop:main}}{\lesssim}
 \lambda^{\bar\alpha} \;.
\end{equ}
In general (see the next subsections), the quantity in the middle may be 
replaced by a sum over several graphs,
and before applying \eref{e:reduce-JM} one may have to deal with the renormalisations.
Now we start to define the above mentioned terminologies precisely.

\begin{definition} \label{def:H-type}
A partial  graph $H$ is a connected graph with each edge $e$ labelled by a real number $m_e$, which satisfies the following requirements.
There exists  a unique distinguished vertex $0\in H$, 
as well as a unique distinguished edge $e=\{0,v_\star\}$ attached to $0$  with label $m_e=0$. (Other edges may also connect to $0$ but they are not called distinguished edges.)
We define 
$H_0 \eqdef H\setminus \{0\}$ and $H_\star \eqdef \{0,v_\star\}$.
The set $H_0$ can be decomposed as a union of two disjoint subsets of vertices
$H_0 = H_{ex} \cup H_{in}$, 
 such that
$\deg(v)=1$ for every $v\in H_{ex}$ and 
$\deg(v)\ge 2$ for every $v\in H_{in} = H_0\setminus H_{ex}$.
We call a vertex in $H_{ex}$ an ``external vertex" and the only edge attached to it an ``external edge". 
We call a vertex in $H_{in}$ an ``internal vertex", and the edges which are not distinguished or external are called ``internal edges". 
We require that $v_\star \in H_{in}$. 
\end{definition}

When we draw a partial  graph, we use a special dot  $\tikz \node [root] (0,0) {};$ to represent the distinguished vertex $0$, 
a black dot $\tikz [baseline=-0.1cm] \node [dot] (0,0) {};$ to represent  a generic vertex which is not the $0$,
a special thick line $\tikz [baseline=-0.1cm] \draw[testfunction] (1,0) to  (0,0);$
to represent the distinguished edge $\{0,v_\star\}$. The labels $m_e$ are often drawn on the edges $e$, such as $\tikz \draw (0,0) to node[labl] {\scriptsize 2+}  (1,0); $, understood as $m_e=2+\delta$ for sufficiently small $\delta>0$.
As some examples of partial  graphs, see for instance \eqref{e:def1stH}, \eqref{e:def2ndH} or \eqref{e:defH1iv1ev}.

\begin{definition} \label{def:Big-graph}
Suppose that we are given a partial  graph $H$ and an integer $p>1$, 
and for each $1\le i\le p$ let $H^{(i)}$ be a copy of the graph $H$.
We say that 
a graph $\CV$ is obtained from $H$ by {\it Wick contracting} $p$ copies of $H$, if there exists
an equivalence relation $\sim$ over the disjoint union
of the $p$ copies $\cup_{i=1}^p H^{(i)}$, 
such that  each equivalence class  $B$ is one of the following three forms:
\begin{itemize}
\item $B=\{0^{(1)},\ldots,0^{(p)}\}$ where $0^{(i)}$ is the distinguished vertex of $H^{(i)}$; or
\item $B=\{v\}$ where $v\in \cup_{i=1}^p H_{in}^{(i)} $; or
\item $B=\{v_1,\ldots,v_{|B|}\} \subset \cup_{i=1}^p H_{ex}^{(i)}$ with $|B|>1$, 
with requirement that there exist at least two elements  $v,v'\in B$ such that $v\in H_{ex}^{(i)}$, $v'\in H_{ex}^{(i')}$ and $i\neq i'$;
\end{itemize}
and such that  $\CV = \cup_{i=1}^p H^{(i)} / \sim$ (namely, we identify all the vertices in each equivalence class as one vertex). 
The set of edges $\CE'(\CV)$ of $\CV$ is the union of all the sets of edges 
of the $p$ copies of $H$,
and each edge $e\in\CE'(\CV)$ of $\CV$ naturally inherits a label $m_e$ from $H$.
We still call the equivalence class consisting of the distinguished vertices $0$, and call an equivalence class consisting of an internal vertex (resp. some external vertices) an in-vertex (resp. an ex-vertex), 
and call an edge $e\in\CE'(\CV)$ an in-edge (resp. ex-edge) if $e$ is  an internal (resp. external) edge of $H$.
We say that an in-vertex belongs to a certain copy $H^{(i)}$ if the only element in this equivalence class is an internal vertex of $H^{(i)}$.
We write $\CV = \CV_{ex} \cup \CV_{in} \cup \{0\}$
where $\CV_{ex}$ is the set of ex-vertices 
and $\CV_{in}$ is the set of in-vertices. 
Let $\CV_0 = \CV\setminus \{0\}$ and $\CV_\star=\{0,v_\star^{(1)},\ldots,v_\star^{(p)}\}$.
\end{definition}

\begin{remark}
Note that given $(H,p)$ there could be many possible graphs $\CV$ constructed as above.
The graph $(\CV,\CE')$ may have multiple edges between two vertices.
In the sequel, $\deg(v)$ for $v\in\CV$ stands for the degree of $v$
counting multiple edges.
\end{remark}

\begin{remark} \label{rem:in-ex-edge}
The graphical notation for a graph $\CV$ in Definition~\ref{def:Big-graph}
will be the same as that for a partial  graph $H$.
We remark that although we use the same solid line to 
represent all edges of  $\CV$,  the functions associated with them will turn out
to be slightly different: an in-edge will be associated with
a singular function $|x|^{-m_e}$ while an ex-edge will be associated with
a mollified function $|x|_\eps^{-m_e}$. We will not distinguish 
them by complicating the graphical notation since it is always
clear whether an edge is an in-edge or ex-edge.
\end{remark}

Given a graph $\CV$ as in Definition~\ref{def:Big-graph}
and a test function $\varphi$,
we can associate a number to it, as alluded above:
\begin{equ} [e:LambdaCV]
\Lambda^p_{\varphi} (\CV)
\eqdef
\int 
\prod_{i=1}^p \varphi(x_{v_\star^{(i)}})
\prod_{v\in \CV_{ex}} \eps^{(\deg(v)/2-1)|\s|} \!\!\!\!
\prod_{e=\{v,\bar v\}\in \CE'(\CV)} \!\!\!\!\! |x_v - x_{\bar v}|_{\eps(e)}^{-m_e}
\prod_{v\in \CV_0}dx_v \;,
\end{equ}
where the norm $|\cdot|_{\eps(e)} = |\cdot|_{\eps}$
if $e$ is an ex-edge and $|\cdot|_{\eps(e)} = |\cdot|$ if $e$ is an in-edge.
Now by observing the expressions \eref{e:bnd2-2} and \eref{e:LambdaCV}
one realises that  we can define a partial graph 
$H$ with one internal vertex $u$ and two external vertices $v_1,v_2$:
\begin{equ} [e:def1stH]
H\eqdef  \;\;
\begin{tikzpicture}   [baseline=-0.1cm]
\node[root]	(root) 	at (0,0) {};
\node[dot]		(mid)  	at (1,0) {};
	\node at (1,.3)  {$u$};
\node[dot]		(a1)  	at (2,.7) {};
	\node at (2.3, .7)  {$v_1$};
\node[dot]		(a2)  	at (2,-.7) {};
	\node at (2.3, -.7)  {$v_2$};
\draw[testfunction] (mid) to node[labl] {\scriptsize 0}  (root);
\draw (mid) to node[labl] {\scriptsize 2+}  (a1);
\draw (mid) to node[labl] {\scriptsize 2+} (a2);
\end{tikzpicture}
\end{equ}
so that for every $\pi$ in \eref{e:bnd2-2}
the integral is equal to $\Lambda^p_{\varphi_0^\lambda} (\CV)$
for  a graph $\CV$
obtained from Wick contracting $p$ copies of $H$. 
Indeed, every in-vertex (resp. ex-vertex) of $\CV$ corresponds to an integration variable $z_k$ (resp. $x_B$) in \eref{e:bnd2-2}.
Define
\begin{equ}  [e:LambdaH]
\Lambda^p_{\varphi^\lambda_0} (H)
\eqdef
\sum_{\CV} \Lambda_{\varphi^\lambda_0}^p (\CV)
\end{equ}
where the sum is over all graphs $\CV$ obtained from Wick 
contracting $p$ copies of $H$.
Then by Definition~\ref{def:partition_M} (for the allowed partitions) and Definition~\ref{def:Big-graph} (for the allowed Wick contractions),
\eref{e:bnd2-2} is equal to $ \Lambda^p_{\varphi^\lambda_0} (H)$.

Proposition~\ref{prop:main} below will then yield the desired bound on 
$\Lambda^p_{\varphi^\lambda_0} (H)$. Before stating the proposition we need the following definitions.

\begin{definition} \label{def:set-edges}
Given a graph $(\CV,\CE)$ we define for any subgraph $\bar \CV \subset \CV$ the following subsets of $\CE$:
\begin{equ}
\CE_0(\bar \CV) = \{e  \in \CE\,:\, e\cap \bar \CV = e\}\;,\qquad
\CE(\bar \CV) = \{e  \in \CE\,:\, e\cap \bar \CV \neq \emptyset\}\;.
\end{equ}
\end{definition}

\begin{definition} \label{def:be-KPZ1}
Given $\bar H \subset H_0$
such that  $ \bar H_{ex} \eqdef \bar H\cap H_{ex}  \neq \emptyset$, 
define number $c_e(\bar H)$ for  edges $e$ of $H$ as follows.
If $| \bar  H_{ex}| =1$, then $c_e(\bar H) =0$ for all $e$.
If $| \bar  H_{ex}| =2$, then, assuming that $\bar  H_{ex} =\{v,\bar v\}$,
there are two cases. The first case is that $v$ and $\bar v$ are connected with the same vertex of $H$, and we define $c_e(\bar H) =3/4$ for every $e\in\CE(\bar H_{ex})$;
the second case is that $v$ and $\bar v$ are connected with different vertices of $H$, and we define $c_e(\bar H) =1/2$ for every $e\in\CE(\bar H_{ex})$.
Finally, if $| \bar H_{ex}| >2$, then 
\begin{equ} [e:be-KPZ1]
c_e(\bar H) = {3\over 2}  -{3\over |\bar  H_{ex}| +1} \;,
\end{equ}
for every $e\in\CE(\bar H_{ex})$. Given any $\bar H \subset H$ such that $0\in\bar H$, we define $c_e(\bar H) =3/2$ for every $e\in\CE(\bar H_{ex})$.
For any subset $\bar H\subset H$ we set $c_e(\bar H) =0$ for every $e\notin\CE(\bar H_{ex})$. 
\end{definition}

\begin{proposition} \label{prop:main}
Given a partial  graph $H$,
suppose that the following conditions hold.
\begin{itemize}
\item[(1)]
For every subset 
$\bar H \subset H$ with $|\bar H| \ge 2$ and $|\bar H_{in}| \ge 1$ one has 
\begin{equ} [e:cond-A]
\sum_{e \in \CE_0(\bar H)} (m_e - c_{e}(\bar H)) < 3 \,\bigl(|\bar H_{in}| - \one_{\bar H \subset H_{in}} \bigr) \;.
\end{equ}
\item[(2)]
For every non-empty $\bar H\subset H\setminus H_\star$ one has 
\begin{equ}[e:cond-B]
\sum_{e \in \CE(\bar H)} m_e 
	> 3\, \Big( |\bar H \cap H_{in}| +{1\over 2} |\bar H \cap H_{ex}| \Big)  \;.
\end{equ}
\end{itemize}
Then, for every $p>1$, one has
\begin{equ} [e:main-bound]
\sup_{\varphi\in\CB} \Lambda^p_{\varphi_0^\lambda}(H) 
	\lesssim \lambda^{\bar\alpha \,p}\;,
\end{equ}
where 
\begin{equ}
\bar\alpha = 3\, \Big( |H_{ex}|/2 + |H_{in} \setminus H_\star| \Big)
		- \sum_{e \in \CE(H)} m_e  \;.
\end{equ}
\end{proposition}

\begin{proof} 
This follows from Corollary~\ref{cor:ulti-bound} below, together 
with Remark~\ref{rem:be-consistent} that the number $c_e(\bar H)$  
in Definition~\ref{def:be-KPZ1} yields an admissible $\eps$-allocation rule in the sense
of Definition~\ref{def:eps-alloc},
and the fact that $\Lambda^p_{\varphi^\lambda_0} (H) \lesssim \Lambda^p_{\varphi^\lambda_0} (\CV)$ with proportionality constant depending only on the cardinality of $H$ and on $p$ (see \eref{e:LambdaH}).
\end{proof}

\begin{remark}
By Remark~\ref{rem:origin}, if the partial graph $H$ is such that $\deg(0)=1$, namely the distinguished edge  $e_\star=\{0,v_\star\}$ with label $m_{e_\star}=0$ is the only edge  connected to $0$, then one does not need to check \eref{e:cond-A}
for the subsets $\bar H$ such that $0\in\bar H$.
\end{remark}


\begin{lemma} \label{lem:check-1st}
The conditions of 
Proposition~\ref{prop:main}
are satisfied for $H$ defined in \eref{e:def1stH}. 
\end{lemma}
\begin{proof}
%
Since $\deg(0)=1$, we only check \eref{e:cond-A} for $\bar H \subset H_0$.
If $\bar H=\{u,v_1\}$, then since $c_{e}(\bar H) = 0$ condition \eref{e:cond-A} reads 
$2+\delta <3$
where  $\delta>0$ is sufficiently small.
The case $\bar H=\{u,v_2\}$ follows in the same way by symmetry.

Still considering condition \eref{e:cond-A}, 
we now look at subgraph $\bar H=\{u,v_1,v_2\}$.
On the left hand side of \eref{e:cond-A},  the set $\CE_0(\bar H)$ consists of two edges. 
Since $ \bar H\cap H_{ex} $ consists of two points which are connected with
the same point (i.e. $u$) in $H$, by Definition~\ref{def:be-KPZ1},
$c_{e}(\bar H) =3/4$ for each $e\in\CE_0(\bar H)$. So condition \eref{e:cond-A}
reads $5/2 +2\delta <3$ for sufficiently small $\delta>0$.

Regarding condition \eref{e:cond-B}, 
if $\bar H=\{v_1\}$ then condition \eref{e:cond-B}
reads $2+\delta > {3\over 2}$,
and the case $\bar H=\{v_2\}$ follows in the same way.
If $\bar H=\{v_1,v_2\}$, then condition \eref{e:cond-B}
reads $4+2\delta > 3$. So \eref{e:cond-B} is verified.
\end{proof}

Proposition~\ref{prop:main}
then yields the desired bound (with $\bar\alpha=-1-\kappa$ for some arbitrarily small $\kappa>0$) on $ \Lambda^p_{\varphi_0^\lambda}(H)$, 
and therefore Theorem~\ref{thm:tightness} is proved for the case $\tau=\<2>$.

\begin{remark}
While for the simplest object $\<2>$ we have directly applied Lemma~\ref{lem:collapse} to the bound \eref{e:bnd2-1}, for the more complicated objects, we will often have to deal with renormalisations {\it before} applying the absolute values as in \eref{e:bnd2-1}.
\end{remark}

\subsection{Bounds on the other objects} \label{sec:bound-other}

\subsubsection*{Bounds on \<21>}

With the generalities discussed in the previous subsection we now proceed to consider the object $\tau=\<21>$.
Recall that we have defined some graphical notations from Section~\ref{sec:renorm} to represent integrations of kernels. Besides these graphical notations, we will need another type of special vertices $\tikz \node [var] {};$ in our graphs. Each instance of $\tikz \node [var] {};$ stands for an integration variable $x$, as well as a factor
$\zeta_\eps(x)$.
Furthermore, if more than one such vertex appear, then the corresponding product of 
$\zeta_\eps$ is always a Wick product 
 $\Wick{ \zeta_\eps(x_1) \cdots \zeta_\eps(x_n)}$, where the $x_i$ are the integration variables represented by 
{\it all} of the special vertices $\tikz \node [var] {};$ appearing in the graph.
This is consistent with the graphical notation used in \cite{WongZakai,KPZJeremy}: our graphs
represent integrals of product of kernels and a (non-Gaussian) Wick product
of the random fields. In the particular case where $\zeta_\eps$ is a Gaussian random field,
this yields an element of the $n$th homogeneous Wiener chaos as in  \cite{WongZakai,KPZJeremy}, and
the two notations do coincide.

With these notations, 
by Definition~\ref{def:Wick-product} and 
the definition of $C_1^{(\eps)}$, combined with the fact that $K$ annihilates constants,
we have the identity
\begin{equ}
(\hat\Pi_{0}^{(\eps)} \<21>) (\varphi^\lambda_0)
=
\begin{tikzpicture}[baseline=0.3cm]
	\node at (0,-.5)  [root] (root) {};
	\node at (-.8,1)  [dot] (left) {};
	\node at (-.8,0)  [dot] (mid) {};
	\node at (0,1.2) [var] (topvar1) {};
	\node at (0,.8) [var] (topvar2) {};
	\node at (0,0) [var] (rightvar) {};
	
	\draw[testfcn] (mid) to  (root);
	
	\draw[kernel] (left) to (mid);
	\draw[kernel] (topvar1) to (left); \draw[kernel] (topvar2) to (left);
	\draw[kernel] (rightvar) to (mid); 
\end{tikzpicture}
\;+\;2\; \left(
\begin{tikzpicture}[baseline=0.3cm]
	\node at (0,-.5)  [root] (root) {};
	\node at (-.8,1)  [dot] (left) {};
	\node at (-.8,0)  [dot] (mid) {};
	\node at (0,1.2) [var] (topvar) {};
	\node at (0,.5) [cumu2n] (rightvar) {};
		\draw[cumu2] (rightvar) ellipse (4pt and 8pt);
	
	\draw[testfcn] (mid) to  (root);
	
	\draw[kernel] (left) to (mid);
	\draw[kernel] (topvar) to (left); 
	\draw[kernel] (rightvar.south) node[dot] {}  to (mid); 
	\draw[kernel] (rightvar.north) node[dot] {}  to (left); 
\end{tikzpicture}
\;-\;\hat c^{(\eps)}\;
\begin{tikzpicture}[baseline=0.3cm]
	\node at (0,-.2)  [root] (root) {};
	\node at (0,.9) [var] (topvar) {};
	\node at (0,.3)  [dot] (mid) {};
	\draw[kernel] (topvar) to (mid); 
	\draw[testfcn] (mid) to  (root);
\end{tikzpicture} \right)
=: \CI_1 + 2\CI_2\;.
\end{equ}
For instance, 
\begin{equ}
\CI_1 = \int  \varphi^\lambda_0(z)\,	 K'( w-x^{(1)}) K'(w-x^{(2)}) K'(z-x^{(3)}) K'( z-w) \Wick{\prod_{i=1}^3 \zeta_\eps(x^{(i)}) } dxdwdz
\end{equ}
where $x=(x^{(1)},x^{(2)},x^{(3)})$. 
Note that in principle one would expect the appearance of a  term 
involving $\kappa_3^{(\eps)}$ in the expression for $\hat\Pi_{0}^{(\eps)} \<21>$, 
but this term is precisely cancelled by $C_1^{(\eps)}$.


We bound the $p$-th moments of $\CI_1$ and $\CI_2$ separately. 
For the moment $\E|\CI_1|^p$ of the first integral $ \CI_1 $, we can apply Corollary~\ref{cor:Wick-field}
and Lemma~\ref{lem:collapse}.
We are then again in the situation that there is a natural  partial graph 
$H$ with two internal vertices and three external vertices:
\begin{equ} [e:def2ndH]
H\eqdef  \;\;
\begin{tikzpicture}   [baseline=-0.1cm]
\node[root]	(root) 	at (0,0) {};
\node[dot]		(u)  	at (1,0) {};
\node[dot]		(w)  	at (2.5,0) {};
\node[dot]		(a1)  	at (1,1) {};
\node[dot]		(a2)  	at (2,1) {};
\node[dot]		(a3)  	at (3,1) {};
\draw[testfunction] (u) to node[labl] {\scriptsize 0}  (root);
\draw (u) to node[labl] {\scriptsize 2+}  (w);
\draw (u) to node[labl] {\scriptsize 2+}  (a1);
\draw (w) to node[labl] {\scriptsize 2+} (a2);
\draw (w) to node[labl] {\scriptsize 2+} (a3);
\end{tikzpicture}
\end{equ}
so that $\E|\CI_1|^p \lesssim \Lambda^p_{\varphi_0^\lambda}(H)$.
As in the case $\tau=\<2>$, it is straightforward to verify 
 the conditions of 
 Proposition~\ref{prop:main}.
We only check the condition \eref{e:cond-A} 
with $\bar H$ consisting of all the vertices of $H$ except $0$ (so $|\bar H|=5$).
Since $|\bar H\cap H_{ex}| =3$, by \eref{e:be-KPZ1} 
in Definition~\ref{def:be-KPZ1} we have 
$c_{e}(\bar H) =3/4$ for each of the three external edges,
so  \eref{e:cond-A} holds since
the left hand side is  $8 + 4\delta -{3\over 4}\cdot 3$ while the right hand side
is equal to $6$.
We remark that the fact that the labels in the graph are actually $2+\delta$ for some {\it strictly positive} $\delta$
 guarantees the strict inequality in the condition \eref{e:cond-B}.
We therefore have the desired bound on
the moments of $ \CI_1 $ by Proposition~\ref{prop:main} 
(with $\bar\alpha=-1/2 -\kappa$ for some arbitrarily small $\kappa>0$).
 
Regarding the second integral $ \CI_2 $, 
we define a function
\begin{equ} [e:defQ]
 Q_\eps(w)\eqdef \int K'(w-x) K'(-y) K'(-w)\,\kappa_2^{(\eps)}(x,y) \,dxdy 
 	- \hat c^{(\eps)} \delta(w)\;.
\end{equ}
We use the notation $\tikz[baseline=-3] \draw[kernelBig] (0,0) to (1,0);$
for the function $Q_\eps$. One then has
\begin{equ} [e:diff-c]
\CI_2 \;= \; 
\begin{tikzpicture}[baseline=0.3cm]
	\node at (0,-.5)  [root] (root) {};
	\node at (-.8,1)  [dot] (left) {};
	\node at (-.8,0)  [dot] (mid) {};
	\node at (0,1.2) [var] (topvar) {};
	\draw[testfcn] (mid) to  (root);
	\draw[kernelBig] (left) to (mid);
	\draw[kernel] (topvar) to (left); 
\end{tikzpicture}
\end{equ}

By \cite[Lemma~10.16]{KPZ} one has
\begin{equ} [e:convolve]
\Big| \int 
	K'(w-x ) Q_\eps(z-w)   \,dw \Big|
\lesssim 
	|z-x|^{-2}\;,
\end{equ}
uniformly over $\eps\in (0,1]$.
Note that the renormalised distribution $\mathscr R Q_\eps$
appearing in \cite[Lemma~10.16]{KPZ} 
is precisely the same as $Q_\eps$, because $Q_\eps$ integrates to $0$ by the choice of $\hat c^{(\eps)}$ in Lemma~\ref{lem:const-correct}.

As an immediate consequence of \eref{e:convolve},
one has for all even $p$
\begin{equ} [e:pmom-I2]
\E |\CI_2|^p \lesssim
 \int \prod_{i=1}^p 
 	\Big(
	\phi_0^\lambda(z_i) \,
	|z_i-x_i|^{-2} \Big)
	\Big| \E\Big(\prod_{i=1}^p \zeta_\eps(x_i)\Big) \Big| \, dx\,dz 
\end{equ}
where the integration is over $x_1,\ldots,x_p$ and $z_1,\ldots,z_p$.
Again, after 
applying Corollary~\ref{cor:Wick-field} and Lemma~\ref{lem:collapse}, we can define a partial graph
$H$ with one internal vertex and one external vertex as
\begin{equ}  [e:defH1iv1ev]
H\eqdef  \;\;
\begin{tikzpicture}   [baseline=-0.1cm]
\node[root]	(root) 	at (0,0) {};
\node[dot]		(mid)  	at (1,0) {};
\node[dot]		(a1)  	at (2.5, 0) {};
\draw[testfunction] (mid) to node[labl] {\scriptsize 0}  (root);
\draw (mid) to node[labl] {\scriptsize 2+}  (a1);
\end{tikzpicture}
\end{equ}
so that \eref{e:pmom-I2} is bounded by $\Lambda^p_{\varphi_0^\lambda}(H)$.
It is then very straightforward to check that 
the conditions of Proposition~\ref{prop:main}  
are all satisfied.
Therefore we have the desired bound on
the moments of $ \CI_2 $.

This completes the proof of Theorem~\ref{thm:tightness} for the case $\tau=\<21>$.
Note that in the bound of $\E |\CI_2|^p$, 
we actually integrated out {\it some of} the integration variables
before applying the first step in the ``routine" \eref{e:routine} 
and representing the rest of the multi-integral by graphs obtained by Wick contracting
an $H$. This procedure will be useful sometimes for other objects.

\subsubsection*{Bounds on \<21a>}

To study the case $\tau=\<21a>$,  we introduce an additional graphical notation
 $\tikz \draw[kernel1]  (0,0) to (1.2,0);$, which is the same as in \cite{WongZakai}
 and \cite{KPZJeremy}.
This barred arrow  represents $K'( z-w)-K'(-w)$ where
$w$ and $z$ are the coordinates of the starting and end point respectively.
With this notation at hand, we have 
\begin{equ} [e:graphs-21a]
(\hat\Pi_{0}^{(\eps)} \<21a>) (\varphi^\lambda_0)
= \;
\begin{tikzpicture}[baseline=0.2cm]
	\node at (0,-.8)  [root] (root) {};
	\node at (-.8,0)  [dot] (left) {};
	\node at (0,0)  [dot] (right) {};
	\node at (-.8,1) [var] (leftvar) {};
	\node at (0,1) [var] (rightvar) {};
	
	\draw[testfcn] (right) to  (root);
	
	\draw[kernel1] (left) to (right);
	\draw[kernel] (leftvar) to (left);
	\draw[kernel] (rightvar) to (right); 
\end{tikzpicture}
\;+\;
\begin{tikzpicture}[baseline=0.2cm]
	\node at (0,-.8)  [root] (root) {};
	\node at (-.8,0)  [dot] (left) {};
	\node at (0,0)  [dot] (right) {};
	\node[cumu2n]	(top)  at (-0.4,1) {};	
	\draw[cumu2] (top) ellipse (8pt and 4pt);
	
	\draw[testfcn] (right) to  (root);
	
	\draw[kernel1] (left) to (right);
	\draw[kernel] (top.west) node[dot] {}  to (left);
	\draw[kernel] (top.east) node[dot] {}  to (right); 
\end{tikzpicture}
\;-\;\hat c^{(\eps)}\;
\begin{tikzpicture}[scale=0.35,baseline=-0.2cm]
	\node at (0,-1)  [root] (root) {};
	\node at (0,1)  [dot] (only) {};
	\draw[testfcn] (only) to  (root);
\end{tikzpicture}
\;\;=\;\;
\begin{tikzpicture}[baseline=0.2cm]
	\node at (0,-0.8)  [root] (root) {};
	\node at (-.8,0)  [dot] (left) {};
	\node at (0,0)  [dot] (right) {};
	\node at (-.8,1) [var] (leftvar) {};
	\node at (0,1) [var] (rightvar) {};
	
	\draw[testfcn] (right) to  (root);
	
	\draw[kernel1] (left) to (right);
	\draw[kernel] (leftvar) to (left);
	\draw[kernel] (rightvar) to (right); 
\end{tikzpicture}
\;-\;
\begin{tikzpicture}[baseline=0.2cm]
	\node at (0,-0.8)  [root] (root) {};
	\node at (-.4,0)  [dot] (left) {};
	\node at (0.4,0)  [dot] (right) {};
	\node[cumu2n]	(top)  at (0,1) {};	
	\draw[cumu2] (top) ellipse (8pt and 4pt);
	
	\draw[testfcn] (right) to  (root); \draw[kernel] (left) to (root);
	
	\draw[kernel] (top.west) node[dot] {}  to (left);
	\draw[kernel] (top.east) node[dot] {}  to (right); 
\end{tikzpicture}
\end{equ}
where we used the definition of $\hat c^{(\eps)}$ (i.e. Lemma~\ref{lem:const-correct}) in the second equality.


The second term on the right hand side of \eref{e:graphs-21a} is deterministic and it is bounded by $\lambda^{-\delta}$
for every $\delta>0$ (for instance, using Lemma~\ref{lem:collapse} and \cite[Lemma~10.14]{Regularity}).  So it remains to bound the first term, which is equal to
\begin{equ}
\int_{\R^8} \varphi^\lambda_0(z)
K'( w-x^{(1)}) K'(z-x^{(2)}) 
	\big( K'( z-w)-K'(-w) \big)
  \Wick{\zeta_\eps(x^{(1)}) \zeta_\eps(x^{(2)})} dx dw dz
\end{equ}
where the $x$-integration is over $x^{(1)}, x^{(2)} $.
This integral can be written as a sum of two terms $ \CI_1' + \CI_2'$,
where $ \CI_i'$  denotes the integration over $\Omega_i$
for $i\in\{1,2\}$, with
\begin{equ}
\Omega_1 \eqdef \{(x^{(1)},x^{(2)},w,z)\in \R^{8}: |w|\ge |z|/2 \} \;,
\end{equ}
and  $\Omega_2 \eqdef \R^{8}\setminus \Omega_1$.
We then bound the $p$th moment of each of these two terms separately.

By the generalised Taylor theorem in \cite[Proposition~11.1]{Regularity},
\begin{equ} [e:pos-ren]
	| K'( z-w)-K'(-w) | \lesssim 
	\begin{cases}
	|z|^{1/2} \, |z-w|^{-5/2}  
			& \mbox{on } \Omega_1  \\
	|z|^{{1/2}-\eta} \, |w|^{-5/2+\eta} 
			& \mbox{on } \Omega_2
	\end{cases}
\end{equ}
for any $\eta \in [0,1/2]$. To see \eqref{e:pos-ren},
by Taylor's theorem one has 
\[
| K'( z-w)-K'(-w) | \lesssim  |z|^{\frac12-\eta} (|z-w|^{-\frac52+\eta} +|w|^{-\frac52+\eta} )
\]
for any $\eta \in [0,1/2]$. If $|w|\ge |z|/2$, one has $|z-w|\lesssim |z|+|w| \lesssim |w|$
and thus the first bound in \eqref{e:pos-ren} by choosing $\eta=0$. 
If $|w| < |z|/2$, then $|w| \lesssim |z-w|$ and one obtains the second bound in \eqref{e:pos-ren}.

On the support of $\phi_0^\lambda(z)$ one has the additional bound $|z| \le \lambda$.
 One then has
\begin{equs}
\,&\E  |\CI_1'|^p
\lesssim
		\lambda^{ {p\over 2}}
  \int_{\Omega_1 \times \cdots \times \Omega_1} \prod_{i=1}^p |\phi_0^\lambda(z_i)| \,
	| K'( w_i-x_i^{(1)}) \,K'(z_i-x_i^{(2)}) |  \label{e:boundI1}\\
	 &\qquad\qquad \times |z_i-w_i|^{-5/2} 
	\Big| \E\Big(\prod_{i=1}^p \Wick{\zeta_\eps(x_i^{(1)}) \zeta_\eps(x_i^{(2)})} \Big) \Big| \, dx\,dw\,dz \\
& \lesssim
	\lambda^{{p\over 2}}
  \!\!   \int_{\R^{6p}} \prod_{i=1}^p |\phi_0^\lambda(z_i)| \,
	 |z_i-x_i^{(2)} |^{-2}  
	 |z_i-x_i^{(1)}|^{-{3\over 2}} 
	\Big| \E\Big(\prod_{i=1}^p 
		\Wick{\zeta_\eps(x_i^{(1)}) \zeta_\eps(x_i^{(2)})} 
	\!\!\Big) \Big| dxdz \;.
\end{equs}
Note that after applying \eref{e:pos-ren}, we bounded the integral 
over the Cartesian product of $p$ copies of $\Omega_1$
by the integration over all of $(\R^2)^{4p}$.
In the last step we then integrated out $w_i$ for all $1\le i\le p$.

\begin{remark}
In the bound \eref{e:pos-ren}, just like in \eref{e:convolve},
we deal with renormalisations ``by hand", no matter whether they are 
``negative renormalisation" (meaning the situations in which we have kernels with too 
small homogeneity such as $Q$ in \eref{e:defQ}) or ``positive renormalisation" (meaning the 
situations in which we have to subtract the kernel at the origin as in \eref{e:pos-ren}), 
rather than directly relying on the general bounds in \cite{KPZJeremy}. We follow this approach 
because, when we apply the bounds on the cumulants (e.g. Lemma~\ref{lem:collapse}), each factor 
in the integrand is bounded by its absolute value, but the renormalisations encode
cancellations that would then be lost. Another reason for following this approach is that Assumption~\ref{ass:graph} below will be simpler to verify than the original assumption in \cite{KPZJeremy}.
\end{remark}

Now we apply Corollary~\ref{cor:Wick-field}
and Lemma~\ref{lem:collapse},
which for $\E  |\CI_1'|^p$  gives the same bound as \eref{e:bnd2-2},
times $\lambda^{p/2}$,
except that some factors $|z_k-x_B |^{-2}$ are replaced by $|z_k-x_B |^{-3/2}$.
This bound is again represented for each $\pi$ 
by a graph
obtained from Wick contracting $p$ copies of $H$, where 
$H$ is given by
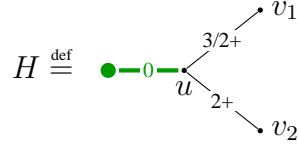
\begin{equ}
H\eqdef  \;\;
\begin{tikzpicture}   [baseline=-0.1cm]
\node[root]	(root) 	at (0,0) {};
\node[dot]		(mid)  	at (1,0) {};
	\node[below] at (mid)  {$u$};
\node[dot]		(a1)  	at (2,.8) {};
	\node[right] at (a1)  {$v_1$};
\node[dot]		(a2)  	at (2,-.8) {};
	\node[right] at (a2)  {$v_2$};
\draw[testfunction] (mid) to node[labl] {\scriptsize 0}  (root);
\draw (mid) to node[labl] {\scriptsize 3/2+}  (a1);
\draw (mid) to node[labl] {\scriptsize 2+} (a2);
\end{tikzpicture}
\end{equ}
One can directly check that the conditions of Proposition~\ref{prop:main} 
are all satisfied.
In fact, the choice of exponent $-5/2$ in \eref{e:pos-ren}
was designed so that, on one hand the function $|z_i-w_i|^{-5/2} $ is still integrable
when we integrate out $w_i$ in \eqref{e:boundI1}, and on the other hand
condition \eref{e:cond-B} is satisfied for the subgraph $\bar H = \{v_1\}$. 
This is because the left hand side of \eref{e:cond-B}
is then equal to $3/2+\delta$ for some small $\delta>0$, while the right hand side
is equal to $3/2$.
We conclude that we have the desired bound on $\E  |\CI_1'|^p$ by
Proposition~\ref{prop:main} with $\bar\alpha=-1/2-\kappa$ for arbitrarily small $\kappa>0$.

Regarding the term $\CI_2'$, we note that on the domain $\Omega_2$
and the support of the function $\varphi^\lambda_0$,
one has $|z|^{1/2-\eta} \lesssim \lambda^{1/2} |w-z|^{-\eta}$.
Therefore one has
\begin{equs}
\,&\E  |\CI_2'|^p
\lesssim
		\lambda^{ {p\over 2}}
  \int_{\Omega_2 \times \cdots \times \Omega_2} 
  \prod_{i=1}^p \phi_0^\lambda(z_i) \,
	| K'( w_i-x_i^{(1)}) \,K'(z_i-x_i^{(2)}) |  \label{e:boundI2} \\
	 &\qquad\qquad \times |z_i-w_i|^{-\eta} \,|w_i|^{-5/2+\eta}
	\Big| \E\Big(\prod_{i=1}^p \Wick{\zeta_\eps(x_i^{(1)}) \zeta_\eps(x_i^{(2)})} \Big) \Big| \, dx\,dw\,dz  \;.
\end{equs}
Again, we bound the integral 
by the corresponding integral over all of $(\R^2)^{4p}$.
After applying Corollary~\ref{cor:Wick-field}
and Lemma~\ref{lem:collapse},
we obtain integrals which are again represented by graphs obtained from Wick 
contracting the partial graph
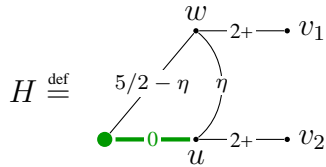
\begin{equ}
H\eqdef  \;\;
\begin{tikzpicture}   [baseline=14,scale=1.2]
\node[root]	(root) 	at (0,0) {};
\node[dot]		(mid)  	at (1,0) {};
	\node[below] at (mid)  {$u$};
\node[dot]		(w)  	at (1,1.2) {};
	\node[above] at (w)  {$w$};
\node[dot]		(a1)  	at (2,1.2) {};
	\node[right] at (a1)  {$v_1$};
\node[dot]		(a2)  	at (2,0) {};
	\node[right] at (a2)  {$v_2$};
\draw[testfunction] (mid) to node[labl] {\scriptsize 0}  (root);
\draw (root) to node[labl] {\scriptsize $5/2-\eta$}  (w);
\draw (w) to node[labl] {\scriptsize 2+}  (a1);
\draw (mid) to node[labl] {\scriptsize 2+} (a2);
\draw[bend right =50] (mid) to node[labl] {\scriptsize $\eta$} (w);
\end{tikzpicture}
\end{equ}
with $H_\star = \{0,u\}$.
It is again straightforward to
verify that the conditions of Proposition~\ref{prop:main} 
are all satisfied, provided that we choose $\eta$ and $\delta$ in a suitable way.
Note that, due to existence of the edge between the vertex $w$
and $0$, we really do need to 
check condition \eref{e:cond-A} for $\bar H \ni 0$.
In fact,
 for the subgraph $\bar H=\{0,w,v_1\}$,
the right hand side of \eref{e:cond-A} is equal to $3$, 
and since $c_{\{w,v_1\}}(\bar H) =3/2$ the left hand side is equal to $(5/2-\eta)+ (2+\delta) -3/2 <3$  
for sufficiently small $\delta>0$ such that $\delta<\eta$.
The fact that the other subgraphs all satisfy condition \eref{e:cond-A}  can be verified straightforwardly, so we omit the details.

Regarding the condition \eref{e:cond-B},
if the subgraph  is given by $\bar H=\{w,v_1\}$, then
the right hand side of \eref{e:cond-B} equals $9/2$
and the left hand side equals $(5/2-\eta )+ (2+\delta)+\eta >9/2$.
For the other subgraphs the condition follows similarly, and we conclude
that the desired bound holds for $\E  |\CI_2'|^p$ by
Proposition~\ref{prop:main} with again $\bar\alpha=-1/2-\kappa$ for arbitrarily small $\kappa>0$.
The proof to the first bound of \eqref{e:tightness} is therefore completed 
for the case $\tau=\<21a>$, noting that although one has
$|\<21a>| = -2\bar \kappa$ by \eqref{e:list-symbols}, the presence of the factor $\lambda^{ {p\over 2}}$
in front of the expressions \eqref{e:boundI1} and \eqref{e:boundI2}
yields the required bound.

\subsubsection*{Bounds on \<4>}

For the object $\<4>$,  it is straightforward to check using \eref{e:defWick} 
and definition of the renormalised model that
\begin{equ}
(\hat\Pi_{0}^{(\eps)} \<4>) (\varphi^\lambda_0)
=
\begin{tikzpicture}[baseline=-0.1cm,scale=0.8]
	\node at (0,-1.2)  [root] (root) {};
	\node at (0,-0.5)  [dot] (mid) {};
	\node at (-1,0)  [dot] (left) {};
	\node at (1,0)  [dot] (right) {};
	\node at (-1.4,1) [var] (var1) {};
	\node at (-0.6,1) [var] (var2) {};
	\node at (0.6,1) [var] (var3) {};
	\node at (1.4,1) [var] (var4) {};
	
	\draw[testfcn] (mid) to  (root);
	
	\draw[kernel] (left) to (mid);\draw[kernel] (right) to (mid);
	\draw[kernel] (var1) to (left);\draw[kernel] (var2) to (left);
	\draw[kernel] (var3) to (right); \draw[kernel] (var4) to (right); 
\end{tikzpicture}
\;+\;4\;
\begin{tikzpicture}[baseline=-0.1cm,scale=0.8]
	\node at (0,-1.2)  [root] (root) {};
	\node at (0,-0.5)  [dot] (mid) {};
	\node at (-1,0)  [dot] (left) {};
	\node at (1,0)  [dot] (right) {};
	\node at (-1.4,1) [var] (var1) {};
	\node at (1.4,1) [var] (var4) {};
	\node[cumu2n]	(top)  		at (0,0.5) {};	
	\draw[cumu2] (top) ellipse (8pt and 4pt);
	
	\draw[testfcn] (mid) to  (root);
	
	\draw[kernel] (left) to (mid);\draw[kernel] (right) to (mid);
	\draw[kernel] (var1) to (left);\draw[kernel] (top.west) node[dot] {} to (left);
	\draw[kernel] (top.east) node[dot] {} to (right); \draw[kernel] (var4) to (right); 
\end{tikzpicture}
\;+\;4\;\;
\begin{tikzpicture}[baseline=-0.1cm,scale=0.8]
	\node at (0,-1.2)  [root] (root) {};
	\node at (0,-0.5)  [dot] (mid) {};
	\node at (-1,0)  [dot] (left) {};
	\node at (1,0)  [dot] (right) {};
\node (top) at (-0, 1) {};
 \node[cumu3,rotate=180] (topcumu) at (-0, 0.9) {};
	\node at (1.4,1) [var] (var4) {};
	
	\draw[testfcn] (mid) to  (root);
	
	\draw[kernel] (left) to (mid);\draw[kernel] (right) to (mid);
	\draw[kernel,bend left=40] (top.south) node[dot] {}  to (left);
	\draw[kernel,bend right=40] (top.west) node[dot] {}  to (left);
	\draw[kernel] (top.east) node[dot] {}  to (right); \draw[kernel] (var4) to (right); 
\end{tikzpicture}
\end{equ}
Note that the terms in which all the four leaves are contracted 
are  canceled by the renormalisation constants $C_{3,1}^{(\eps)}$ and $C_{3,2}^{(\eps)}$. Here, the first two graphs are essentially the same as in \cite{KPZJeremy} (except that our graphs represent non-Gaussian Wick products of random fields), but the last graph is new. It will be shown that the last term vanishes in the limit $\eps\to 0$.

Denote by $\CI_{i}^{\dagger}$ ($i=1,2,3$) the above terms respectively. 
For $ \CI^{\dagger}_1 $  and $ \CI^{\dagger}_2 $ we can apply Corollary~\ref{cor:Wick-field}
and Lemma~\ref{lem:collapse} to obtain 
the bound  $\E| \CI^{\dagger}_i|^p \lesssim \Lambda^p_{\varphi^\lambda_0} (H_i)$
where $i\in\{1,2\}$ and $H_1$, $H_2$ are defined below.
It is straightforward to check the conditions of 
Proposition~\ref{prop:main}. 
Note that for $ \CI^{\dagger}_2 $, we have integrated out the second cumulant $\kappa_2^{(\eps)}$,
which gives the factor $| x-y |^{-1} $ (corresponding to the edge with label ``$1+$" in $H_2$).
\begin{equ}
H_1  =
\begin{tikzpicture}[baseline=-0.1cm,scale=1]
	\node at (0,-1.2)  [root] (root) {};
	\node at (0,-0.4)  [dot] (mid) {};
	\node at (-1,0)  [dot] (left) {};
	\node at (1,0)  [dot] (right) {};
	\node at (-1.4,1) [dot] (var1) {};
	\node at (-0.6,1) [dot] (var2) {};
	\node at (0.6,1) [dot] (var3) {};
	\node at (1.4,1) [dot] (var4) {};
	
	\draw[testfunction] (mid) to node[labl] {\scriptsize 0} (root);
	
	\draw  (left) to node[labl] {\scriptsize 2+} (mid);
	\draw  (right) to node[labl] {\scriptsize 2+} (mid);
	\draw  (var1) to node[labl] {\scriptsize 2+} (left);
	\draw  (var2) to node[labl] {\scriptsize 2+} (left);
	\draw  (var3) to node[labl] {\scriptsize 2+} (right); 
	\draw  (var4) to node[labl] {\scriptsize 2+}  (right); 
\end{tikzpicture}
\qquad H_2 =
\begin{tikzpicture}[baseline=-0.1cm,scale=1]
	\node at (0,-1.2)  [root] (root) {};
	\node at (0,-0.4)  [dot] (mid) {};
	\node at (-1,0)  [dot] (left) {};
	\node at (1,0)  [dot] (right) {};
	\node at (-1.4,1) [dot] (var1) {};
	\node at (1.4,1) [dot] (var4) {};
	
	\draw[testfunction] (mid) to node[labl] {\scriptsize 0} (root);
	
	\draw  (left) to node[labl] {\scriptsize 2+} (mid);
	\draw  (right) to node[labl] {\scriptsize 2+} (mid);
	\draw  (var1) to node[labl] {\scriptsize 2+} (left);
	\draw[bend right=60] (right)  to node[labl] {\scriptsize 1+} (left);
	\draw  (var4) to node[labl] {\scriptsize 2+} (right); 
\end{tikzpicture}
\qquad H_3 =  \;\;
\begin{tikzpicture}[baseline=-0.1cm,scale=1]
	\node at (0,-1.2)  [root] (root) {};
	\node at (0,-0.4)  [dot] (mid) {};
	\node at (0,1) [dot] (var4) {};
	
	\draw[testfunction] (mid) to node[labl] {\scriptsize 0} (root);
	 \draw  (var4) to node[labl] {\scriptsize 2-} (mid); 
\end{tikzpicture}
\end{equ}

 
Regarding the term  $ \CI^{\dagger}_3 $,
when we apply Lemma~\ref{lem:collapse},
 we gain a factor $\eps^{3/2}$
from the third cumulant $\kappa_3^{(\eps)}$.
It is clear from our graphical representation that there exist ``double edges"
connecting to the third cumulant function,
which stand for a function $(|x-w|+\eps)^{-4}$.
 We can use a factor $\eps^{1+\delta}$ with $0<\delta<1/2$
  to improve the homogeneity of this function, namely, we have bound
  $\eps^{1+\delta} (|x-w|+\eps)^{-4} \lesssim (|x-w|+\eps)^{-3+\delta}$;
 and we still have a factor $\eps^{{1\over 2} -\delta}$ left.
Integrating out all the variables represented by dots $\tikz \node[dot] at (0,0) {};$
(which are simple convolutions) except for the ones where the test functions $\varphi^\lambda_0$ are evaluated at, we  again obtain 
a bound $\Lambda^p_{\varphi_0^\lambda}(H_3)$
where the partial graph  drawn as 
 $H_3$  above (with  the label $2-$ denoting $2-\delta$).
It is easy to verify that the
conditions  of Proposition~\ref{prop:main}
are again satisfied for $H_3$.
With the remaining factor  $\eps^{{1\over 2} -\delta}$, we see that as $\eps\to 0$,
the moments of 
this term converge to zero.
The proof of the first bound of \eqref{e:tightness} is therefore completed 
for the case $\tau=\<4>$.

\subsubsection*{Bounds on \<211>}

For the object $\<211>$, by Definition~\ref{def:Wick-product} for Wick products 
and the definition of the renormalised model $\hat\Pi^{(\eps)}$, one has
\begin{equs}
\bigl( \hat \Pi_0^{(\eps)} \<211>\bigr) & (\phi^\lambda_0) =  \;
\begin{tikzpicture}[scale=0.35,baseline=0.5cm]
	\node at (0,-1)  [root] (root) {};
	\node at (-2,1)  [dot] (left) {};
	\node at (-2,3)  [dot] (left1) {};
	\node at (-2,5)  [dot] (left2) {};
	\node at (0,1) [var] (variable1) {};
	\node at (0,3) [var] (variable2) {};
	\node at (0,4.3) [var] (variable3) {};\node at (0,5.7) [var] (variable4) {};
	
	\draw[testfcn] (left) to  (root);
	
	\draw[kernel1] (left1) to (left);
	\draw[kernel] (left2) to (left1);
	\draw[kernel] (variable3) to (left2); \draw[kernel] (variable4) to (left2); 
	\draw[kernel] (variable2) to (left1); 
	\draw[kernel] (variable1) to (left); 
\end{tikzpicture}
\;+\; \left(
\begin{tikzpicture}[scale=0.35,baseline=0.5cm]
	\node at (0,-1)  [root] (root) {};
	\node at (-2,1)  [dot] (left) {};
	\node at (-2,3)  [dot] (left1) {};
	\node at (-2,5)  [dot] (left2) {};
	\node[cumu2n]	(variable12)  		at (0,2) {};	
	\draw[cumu2] (variable12) ellipse (10pt and 20pt);
	\node at (0,4.3) [var] (variable3) {};\node at (0,5.7) [var] (variable4) {};
	
	\draw[testfcn] (left) to  (root);
	
	\draw[kernel1] (left1) to (left);
	\draw[kernel] (left2) to (left1);
	\draw[kernel] (variable3) to (left2); \draw[kernel] (variable4) to (left2); 
	\draw[kernel] (variable12.north) node[dot]   {} to (left1); 
	\draw[kernel] (variable12.south) node[dot]   {}  to (left); 
\end{tikzpicture}
\;-\; \hat c^{(\eps)} \;
\begin{tikzpicture}[scale=0.35,baseline=0.3cm]
	\node at (-.2,-.8)  [root] (root) {};
	\node at (-2,1)  [dot] (left) {};
	\node at (-2,3)  [dot] (left1) {};
	\node at (0,2.3) [var] (variable3) {};\node at (0,3.7) [var] (variable4) {};
	
	\draw[testfcn] (left) to  (root);	
	\draw[kernel] (left1) to (left);
	\draw[kernel] (variable3) to (left1); \draw[kernel] (variable4) to (left1); 
\end{tikzpicture}
\right)
\;+\;2\; \left(
\begin{tikzpicture}[scale=0.35,baseline=0.5cm]
	\node at (0,-1)  [root] (root) {};
	\node at (-2,1)  [dot] (left) {};
	\node at (-2,3)  [dot] (left1) {};
	\node at (-2,5)  [dot] (left2) {};
	\node at (0,1) [var] (variable1) {};
	\node[cumu2n]	(variable23)  		at (0,4) {};	
	\draw[cumu2] (variable23) ellipse (10pt and 20pt);
	\node at (0,5.7) [var] (variable4) {};
	
	\draw[testfcn] (left) to  (root);
	
	\draw[kernel1] (left1) to (left);
	\draw[kernel] (left2) to (left1);
	\draw[kernel] (variable23.north) node[dot]   {} to (left2); 
	\draw[kernel] (variable4) to (left2); 
	\draw[kernel] (variable23.south)  node[dot]   {} to (left1); 
	\draw[kernel] (variable1) to (left); 
\end{tikzpicture}
\;-\; \hat c^{(\eps)} \;
\begin{tikzpicture}[scale=0.35,baseline=0.3cm]
	\node at (0,-1)  [root] (root) {};
	\node at (-2,1)  [dot] (left) {};
	\node at (-2,3)  [dot] (left1) {};
	\node at (0,1) [var] (variable1) {};
	\node at (0,3.7) [var] (variable4) {};
	
	\draw[testfcn] (left) to  (root);	
	\draw[kernel1] (left1) to (left);
	\draw[kernel] (variable4) to (left1); 
	\draw[kernel] (variable1) to (left); 
\end{tikzpicture} 
\right)
\\&
\;+\;2\;
\begin{tikzpicture}[scale=0.35,baseline=0.5cm]
	\node at (0,-1)  [root] (root) {};
	\node at (-2,1)  [dot] (left) {};
	\node[cumu2n]	(left1)  		at (-2,3) {};	
	\draw[cumu2] (left1) ellipse (10pt and 20pt);
	\node at (-2,5)  [dot] (left2) {};
	\node at (0,3) [dot] (right) {};
	\node at (0,4.6) [var] (variable3) {};
	\node at (-.5, 5.7) [var] (variable4) {};
	
	\draw[testfcn] (left) to  (root);	
	\draw[kernel] (left1.south)  node[dot] {} to (left);
	\draw[kernel] (left1.north) node[dot]   {} to (left2); 
	\draw[kernel] (left2) to (right); 
	\draw[kernel] (variable4) to (left2); \draw[kernel] (variable3) to (right); 
	\draw[kernel1] (right) to (left); 
\end{tikzpicture}
\;+\;2\;
\left(
\begin{tikzpicture}[scale=0.35,baseline=0.5cm]
	\node at (0,-1)  [root] (root) {};
	\node at (-2,1)  [dot] (left) {};
	\node[cumu2n]	(left1)  		at (-2,3) {};	
	\draw[cumu2] (left1) ellipse (10pt and 20pt);
	\node at (-2,5)  [dot] (left2) {};
	\node at (0,3) [dot] (right) {};
	\node[cumu2n]	(top)  		at (0,5.5) {};	
	\draw[cumu2] (top) ellipse (20pt and 10pt);

	\draw[testfcn] (left) to  (root);
	
	\draw[kernel] (left1.south)  node[dot] {} to (left);
	\draw[kernel] (left1.north) node[dot]   {} to (left2); 
	\draw[kernel] (left2) to (right); 
	\draw[kernel] (top.west) node[dot] {} to (left2); \draw[kernel] (top.east) node[dot] {} to (right); 
	\draw[kernel1] (right) to (left); 
\end{tikzpicture}
\;-\hat c^{(\eps)} \;
\begin{tikzpicture}[scale=0.35,baseline=0.2cm]
	\node at (0,-1)  [root] (root) {};
	\node at (0,1)  [dot] (left) {};
	\node[cumu2n]	(c)  		at (2,3) {};	
	\draw[cumu2] (c) ellipse (10pt and 20pt);

	\node at (0,4) [dot] (top) {};

	\draw[testfcn] (left) to  (root);
	
	\draw[kernel] (c.south)  node[dot] {} to (left);
	\draw[kernel] (c.north) node[dot]   {} to (top); 
	\draw[kernel1] (top) to (left); 
\end{tikzpicture}
\;-\;C_{2,1}^{(\eps)}\;
\begin{tikzpicture}[scale=0.35,baseline=-0.2cm]
	\node at (0,-1)  [root] (root) {};
	\node at (0,1)  [dot] (only) {};
	\draw[testfcn] (only) to  (root);
\end{tikzpicture}
\;+\;
\frac{(\hat c^{(\eps)})^2}{2} \;
\begin{tikzpicture}[scale=0.35,baseline=-0.2cm]
	\node at (0,-1)  [root] (root) {};
	\node at (0,1)  [dot] (only) {};
	\draw[testfcn] (only) to  (root);
\end{tikzpicture}
\right)
\label{e:211graphs} \\&
\;+\;
\begin{tikzpicture}[scale=0.35,baseline=0.5cm]
	\node at (0,-1)  [root] (root) {};
	\node at (-2,1)  [dot] (left) {};
	\node	(left1)  		at (-2,3) {};	
	\node[cumu3,rotate=90] (left1-cumu) at (-2.2,3) {};
	\node at (-2,5)  [dot] (left2) {};
	\node at (0,3) [dot] (right) {};
	\node at (0,5) [var] (variable3) {};
	
	\draw[testfcn] (left) to  (root);
	
	\draw[kernel] (left1.south)  node[dot] {} to (left);
	\draw[kernel] (left1.north) node[dot]   {} to (left2); 
	\draw[kernel,bend left=60] (left1.west) node[dot]   {} to (left2); 
	\draw[kernel] (left2) to (right); 
	\draw[kernel] (variable3) to (right); 
	\draw[kernel1] (right) to (left); 
\end{tikzpicture}
\;+\; 2\;\;
\begin{tikzpicture}[scale=0.35,baseline=0.5cm]
	\node at (0,-1)  [root] (root) {};
	\node at (-2,1)  [dot] (left) {};
	\node at (-2,3)  [dot] (left1) {};
	\node at (-2,5)  [dot] (left2) {};
	\node	(right)  		at (0,3) {};	
	\node[cumu3,rotate=90] (right-cumu) at (-0.2,3) {};
	\node at (0,5) [var] (variable4) {};
	
	\draw[testfcn] (left) to  (root);
	
	\draw[kernel1] (left1) to (left);
	\draw[kernel] (left2) to (left1);
	\draw[kernel] (right.north) node[dot] {} to (left2); 
	\draw[kernel] (variable4) to (left2); 
	\draw[kernel] (right.west) node[dot] {} to (left1); 
	\draw[kernel] (right.south) node[dot] {} to (left); 
\end{tikzpicture}
\;+\; \left(
\begin{tikzpicture}  [scale=0.35,baseline=0.5cm]
\node[root]	(root) 	at (2,-1) {};
\node[dot]		(bottom)	at (0,1) {};
\node[dot]		(mid)  	at (0,3) {};
\node[dot]		(top)  	at (0,5) {};
\node		(right)  	at (2,3.4) {};
	\node[cumu4]	(right-cumu) 	at (right) {};
\draw[testfcn] (bottom) to  (root);
\draw[kernel1] (mid) to  (bottom);
\draw[kernel] (top) to  (mid);
\draw[kernel] (right.south west) node[dot] {} to  (mid);
\draw[kernel] (right.south east) node[dot] {} to  (bottom);
\draw[kernel] (right.north west) node[dot] {} to  (top);
\draw[kernel,bend right=40] (right.north east) node[dot] {} to  (top);
\end{tikzpicture}
\;-\;C_{2,2}^{(\eps)}\;
\begin{tikzpicture}[scale=0.35,baseline=-0.2cm]
	\node at (0,-1)  [root] (root) {};
	\node at (0,1)  [dot] (only) {};
	\draw[testfcn] (only) to  (root);
\end{tikzpicture}
\right)
\end{equs}
Here, the graphs in the first two lines are essentially the same 
as those in \cite{KPZJeremy} (except for the renormalisation constant $\hat c^{(\eps)}$), but the graphs in the last line
are new due to the nontrivial higher cumulants of the random field $\zeta_\eps$.
Note that the term involving the joint cumulant of $\zeta_\eps$ at three points
represented by the top three $\tikz \node[var] (0,0) {};$ shaped vertices
vanishes because the kernel $K$ annihilates constants.

Using the notation $\tikz[baseline=-3] \draw[kernelBig] (0,0) to (1,0);$
for the kernel $Q_\eps$ defined in \eref{e:defQ},
and proceeding in the same way as \cite{KPZJeremy} with the definition of constant   $C_{2,1}^{(\eps)}$, we see that
the sum of all the graphs appearing in {\it the first two lines} of \eref{e:211graphs} is equal to
\begin{equ}   [e:graphs-211x]
\begin{tikzpicture}[scale=0.35,baseline=0.5cm]
	\node at (0,-1)  [root] (root) {};
	\node at (-2,1)  [dot] (left) {};
	\node at (-2,3)  [dot] (left1) {};
	\node at (-2,5)  [dot] (left2) {};
	\node at (0,1) [var] (variable1) {};
	\node at (0,3) [var] (variable2) {};
	\node at (0,4.3) [var] (variable3) {};\node at (0,5.7) [var] (variable4) {};
	
	\draw[testfcn] (left) to  (root);
	
	\draw[kernel1] (left1) to (left);
	\draw[kernel] (left2) to (left1);
	\draw[kernel] (variable3) to (left2); \draw[kernel] (variable4) to (left2); 
	\draw[kernel] (variable2) to (left1); 
	\draw[kernel] (variable1) to (left); 
\end{tikzpicture}
\;+\;
\begin{tikzpicture}[scale=0.35,baseline=0.5cm]
	\node at (0,-1)  [root] (root) {};
	\node at (-2,1)  [dot] (left) {};
	\node at (-2,3)  [dot] (left1) {};
	\node at (-2,5)  [dot] (left2) {};
	\node at (0,4.3) [var] (variable3) {};\node at (0,5.7) [var] (variable4) {};
	
	\draw[testfcn] (left) to  (root);
	
	\draw[kernelBig] (left1) to (left);
	\draw[kernel] (left2) to (left1);
	\draw[kernel] (variable3) to (left2); \draw[kernel] (variable4) to (left2); 
\end{tikzpicture}
\;-\;
\begin{tikzpicture}[scale=0.35,baseline=0.5cm]
	\node at (0,-1)  [root] (root) {};
	\node at (-2,1)  [dot] (left) {};
	\node at (0,3)  [dot] (left1) {};
	\node at (0,5)  [dot] (left2) {};
	\node[cumu2n]	(variable12)  		at (-2,2.7) {};	
	\draw[cumu2] (variable12) ellipse (10pt and 20pt);
	\node at (-2,4.3) [var] (variable3) {};\node at (-2,5.7) [var] (variable4) {};
	
	\draw[testfcn] (left) to  (root);
	
	\draw[kernel] (left1) to (root);
	\draw[kernel] (left2) to (left1);
	\draw[kernel] (variable3) to (left2); \draw[kernel] (variable4) to (left2); 
	\draw[kernel] (variable12.north) node[dot]   {} to (left1); 
	\draw[kernel] (variable12.south) node[dot]   {}  to (left); 
\end{tikzpicture}
\;+\;2\;
\begin{tikzpicture}[scale=0.35,baseline=0.5cm]
	\node at (0,-1)  [root] (root) {};
	\node at (-2,1)  [dot] (left) {};
	\node at (-2,3)  [dot] (left1) {};
	\node at (-2,5)  [dot] (left2) {};
	\node at (0,1) [var] (variable1) {};
	\node at (0,5.7) [var] (variable4) {};
	
	\draw[testfcn] (left) to  (root);
	
	\draw[kernel1] (left1) to (left);
	\draw[kernelBig] (left2) to (left1);
	\draw[kernel] (variable4) to (left2); 
	\draw[kernel] (variable1) to (left); 
\end{tikzpicture}
\;+\;2\;
\begin{tikzpicture}[scale=0.35,baseline=0.5cm]
	\node at (0,-1)  [root] (root) {};
	\node at (-2,1)  [dot] (left) {};
	\node[cumu2n]	(left1)  		at (-2,3) {};	
	\draw[cumu2] (left1) ellipse (10pt and 20pt);
	\node at (-2,5)  [dot] (left2) {};
	\node at (0,3) [dot] (right) {};
	\node at (0,4.6) [var] (variable3) {};
	\node at (-.5, 5.7) [var] (variable4) {};
	
	\draw[testfcn] (left) to  (root);
	
	\draw[kernel] (left1.south)  node[dot] {} to (left);
	\draw[kernel] (left1.north) node[dot]   {} to (left2); 
	\draw[kernel] (left2) to (right); 
	\draw[kernel] (variable4) to (left2); \draw[kernel] (variable3) to (right); 
	\draw[kernel1] (right) to (left); 
\end{tikzpicture}
\;-\;2\;
\begin{tikzpicture}[scale=0.35,baseline=0.5cm]
	\node at (0,-1)  [root] (root) {};
	\node at (-2,1)  [dot] (left) {};
	\node[cumu2n]	(left1)  		at (-2,3) {};	
	\draw[cumu2] (left1) ellipse (10pt and 20pt);
	\node at (-2,5)  [dot] (left2) {};
	\node at (0,3.5) [dot] (right) {};

	\draw[testfcn] (left) to  (root);
	
	\draw[kernel] (left1.south)  node[dot] {} to (left);
	\draw[kernel] (left1.north) node[dot]   {} to (left2); 
	\draw[kernelBig] (left2) to (right); 
	\draw[kernel] (right) to (root); 
\end{tikzpicture}
\end{equ}
%
%

We denote by $\CI_k^{\star}$ the $k$-th term in \eref{e:graphs-211x} with $k=1,\ldots,6$. 
Note that $\CI_6^{\star}$ is deterministic
and easily bounded by $\lambda^{-\eta}$ for any $\eta>0$.

For the term $\CI_2^{\star}$, we apply the bound \eref{e:convolve},
then it can be bounded using Corollary~\ref{cor:Wick-field} and
Lemma~\ref{lem:collapse} followed by 
Proposition~\ref{prop:main}, with the partial graph shown
in the left picture below. 
The term $\CI_3^{\star}$ is bounded in the same way, 
with  the partial  graph shown in the right picture:
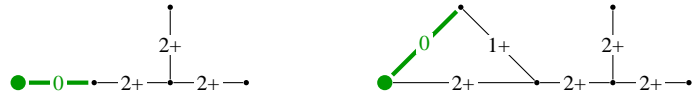
\begin{equ}
\begin{tikzpicture}   [baseline=0.4]
\node[root]	(root) 	at (0,0) {};
\node[dot]		(mid)  	at (1,0) {};
\node[dot]		(right)  	at (2,0) {};
\node[dot]		(a1)  	at (2,1) {};
\node[dot]		(a2)  	at (3,0) {};
\draw[testfunction] (mid) to node[labl] {\scriptsize 0}  (root);
\draw  (right) to node[labl] {\scriptsize 2+}  (mid);
\draw  (right) to node[labl] {\scriptsize 2+}  (a1);
\draw  (right) to node[labl] {\scriptsize 2+} (a2);
\end{tikzpicture}
\qquad\qquad
\begin{tikzpicture}[baseline=0.4]
	\node at (0,0)  [root] (root) {};
	\node at (1,1)  [dot] (left) {};
	\node at (2,0)  [dot] (left1) {};
	\node at (3,0)  [dot] (left2) {};
	\node at (4,0) [dot] (variable3) {};
	\node at (3,1) [dot] (variable4) {};
	
	\draw[testfunction] (left) to  node[labl] {\scriptsize 0}  (root);
	
	\draw  (left1) to node[labl] {\scriptsize 2+}  (root);
	\draw  (left2) to node[labl] {\scriptsize 2+} (left1);
	\draw  (variable3) to node[labl] {\scriptsize 2+} (left2); 
	\draw  (variable4) to node[labl] {\scriptsize 2+} (left2); 
	\draw    (left) to node[labl] {\scriptsize 1+}  (left1); 
\end{tikzpicture}
\end{equ}
It is again straightforward to check the conditions of 
Proposition~\ref{prop:main}  for these partial graphs.

For the term $\CI_4^{\star}$, we apply \eref{e:convolve},
then it is essentially the same as in the case $\tau = \<21a>$.
Regarding the term $\CI_5^{\star}$, we can bound 
the two kernels in $\tikz \draw[kernel1]  (0,0) -- (1,0); $ separately.
In other words, we write $\CI_5^{\star}$ as sum of two terms,
each containing one of the two kernels. The moments 
of each term  can then be bounded in the same way as above 
with partial graphs defined as follows respectively:
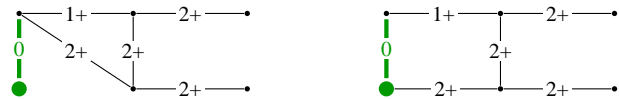
\begin{equ}
\begin{tikzpicture}[baseline=0.5cm]
	\node at (0,0)  [root] (root) {};
	\node at (0,1)  [dot] (left) {};
	\node at (1.5,1)  [dot] (left2) {};
	\node at (1.5,0) [dot] (right) {};
	\node at (3,0) [dot] (variable3) {};
	\node at (3,1) [dot] (variable4) {};
	
	\draw[testfunction] (left) to  node[labl] {\scriptsize 0} (root);
	
	\draw  (left2) to node[labl] {\scriptsize 1+} (left);
	\draw  (left2) to node[labl] {\scriptsize 2+} (right); 
	\draw  (variable4) to node[labl] {\scriptsize 2+} (left2); 
	\draw  (variable3) to node[labl] {\scriptsize 2+} (right); 
	\draw  (right) to node[labl] {\scriptsize 2+} (left); 
\end{tikzpicture}
\qquad\qquad
\begin{tikzpicture}[baseline=0.5cm]
	\node at (0,0)  [root] (root) {};
	\node at (0,1)  [dot] (left) {};
	\node at (1.5,1)  [dot] (left2) {};
	\node at (1.5,0) [dot] (right) {};
	\node at (3,0) [dot] (variable3) {};
	\node at (3,1) [dot] (variable4) {};
	
	\draw[testfunction] (left) to  node[labl] {\scriptsize 0} (root);
	
	\draw  (left2) to node[labl] {\scriptsize 1+} (left);
	\draw  (left2) to node[labl] {\scriptsize 2+} (right); 
	\draw  (variable4) to node[labl] {\scriptsize 2+} (left2); 
	\draw  (variable3) to node[labl] {\scriptsize 2+} (right); 
	\draw  (right) to node[labl] {\scriptsize 2+} (root); 
\end{tikzpicture}
\end{equ}
The conditions of 
Proposition~\ref{prop:main}  are again verified.

We now bound moments of $\CI_1^{\star}$.
Using the bound \eref{e:pos-ren} 
and proceeding 
in the same way as in the case $\<21a>$,
one can write $\CI_1^{\star}$ as a sum of two terms,
so that the moments of these two terms can be bounded separately.
The partial graphs associated with these two terms are
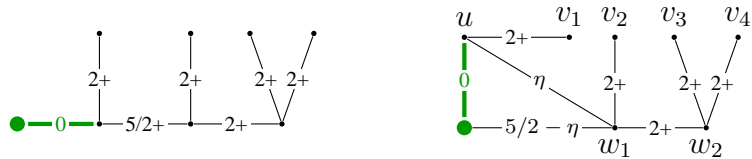
\begin{equ}
\begin{tikzpicture}[scale=0.6,baseline=0cm]
	\node at (-.8,-2)  [root] (root) {};
	\node at (1,-2)  [dot] (left) {};
	\node at (3,-2)  [dot] (left1) {};
	\node at (5,-2)  [dot] (left2) {};
	\node at (1,0) [dot] (variable1) {};
	\node at (3,0) [dot] (variable2) {};
	\node at (4.3,0) [dot] (variable3) {};
	\node at (5.7,0) [dot] (variable4) {};
	
	\draw[testfunction] (left) to  node[labl] {\scriptsize 0} (root);
	
	\draw  (left1) to node[labl] {\scriptsize 5/2+} (left);
	\draw  (left2) to node[labl] {\scriptsize 2+} (left1);
	\draw  (variable3) to node[labl] {\scriptsize 2+} (left2); 
	\draw  (variable4) to node[labl] {\scriptsize 2+} (left2); 
	\draw  (variable2) to node[labl] {\scriptsize 2+} (left1); 
	\draw  (variable1) to node[labl] {\scriptsize 2+} (left); 
\end{tikzpicture}
\qquad\qquad
\begin{tikzpicture}[scale=0.6,baseline=0.05cm]
	\node at (-.3, -2)  [root] (root) {};
	\node at (-.3,0)  [dot] (u) {}; \node[above] at (u) {$u$};
	\node at (3,-2)  [dot] (w1) {}; \node[below] at (w1) {$w_1$};
	\node at (5,-2)  [dot] (w2) {}; \node[below] at (w2) {$w_2$};
	\node at (2,0) [dot] (variable1) {}; \node[above] at (variable1) {$v_1$};
	\node at (3,0) [dot] (variable2) {}; \node[above] at (variable2) {$v_2$};
	\node at (4.3,0) [dot] (variable3) {}; \node[above] at (variable3) {$v_3$};
	\node at (5.7,0) [dot] (variable4) {}; \node[above] at (variable4) {$v_4$};
	
	\draw[testfunction] (u) to  node[labl] {\scriptsize 0} (root);
	
	\draw (w1) to node[labl] {\scriptsize $5/2-\eta$} (root);
	\draw  (w2) to node[labl] {\scriptsize 2+} (w1);
	\draw  (variable3) to node[labl] {\scriptsize 2+} (w2); 
	\draw  (variable4) to node[labl] {\scriptsize 2+} (w2); 
	\draw  (variable2) to node[labl] {\scriptsize 2+} (w1); 
	\draw  (variable1) to node[labl] {\scriptsize 2+} (u); 
	\draw  (u) to node[labl] {\scriptsize $\eta$} (w1); 
\end{tikzpicture}
\end{equ}
These graphs still satisfy the conditions of  Proposition~\ref{prop:main}. 
We only check here the conditions for some subgraphs of the second graph.
For condition \eref{e:cond-A},
 if the subgraph $\bar H=\{0,w_1,w_2,v_2,v_3,v_4\}$,
 then the right hand side of \eref{e:cond-A} is equal to $6$, 
and the left hand side is equal to $(5/2-\eta )+ 4(2+\delta) -3\cdot 3/2 < 6$
for sufficiently small $\delta>0$ such that $4 \delta<\eta$.
Regarding condition \eref{e:cond-B},
if the subgraph $\bar H=\{w_1,w_2,v_2,v_3,v_4\}$, then
the condition reads $(5/2-\eta )+ 4(2+\delta)+\eta > 21/2$ for any $\delta>0$.
Proposition~\ref{prop:main} 
 (with $\bar\alpha=-1/2 -\kappa$ for some arbitrarily small $\kappa>0$) then yields the desired bound on the $p$-th moment, noting the overall factor $\lambda^{ {p\over 2}}$ showing up as we proceed
 as in the case $\<21a>$.

We proceed to consider the first graph in the last line of \eref{e:211graphs}.
Applying Lemma~\ref{lem:collapse} to the third cumulant $\kappa_3^{(\eps)}$
yields a factor $\eps^{3/2}$. As before we combine a factor $\eps^{1+\eta}$ 
together with the ``double-edge" to obtain an edge with label $3-\eta$.
Regarding the two terms in the barred edge, we deal with them separately.
After integrating out some variables which are simple convolutions,
the bounds boil down to straightforward verification of the conditions of 
Proposition~\ref{prop:main} 
for the following two partial graphs:
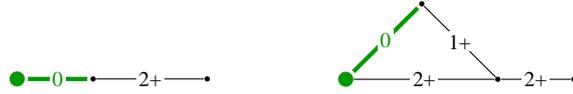
\begin{equ}
\begin{tikzpicture}   
\node[root]	(root) 	at (0,0) {};
\node[dot]		(mid)  	at (1,0) {};
\node[dot]		(a1)  	at (2.5, 0) {};
\draw[testfunction] (mid) to node[labl] {\scriptsize 0}  (root);
\draw (mid) to node[labl] {\scriptsize 2+}  (a1);
\end{tikzpicture}
\qquad\qquad
\begin{tikzpicture} 
	\node at (0,0)  [root] (root) {};
	\node at (1,1)  [dot] (left) {};
	\node at (2,0)  [dot] (left1) {};
	\node at (3,0)  [dot] (var) {};
	
	\draw[testfunction] (left) to  node[labl] {\scriptsize 0}  (root);
	
	\draw  (left1) to node[labl] {\scriptsize 2+}  (root);
	\draw  (var) to node[labl] {\scriptsize 2+} (left1);
	\draw    (left) to node[labl] {\scriptsize 1+}  (left1); 
\end{tikzpicture}
\end{equ}
Since there is the factor $\eps^{1/2-\eta}$ remaining, the moments of
this graph converge to zero.

With the same argument the second graph in the last line of \eref{e:211graphs} also converges to zero. 
For the quantity in the last bracket in \eref{e:211graphs}, after cancellation with the renormalisation constant it is equal to the same graph with the barred arrow replaced by an arrow pointing to the origin, and this graph, which represents a deterministic number, is bounded by $\eps^{1-\delta}\lambda^{\delta-1}$ for sufficiently small $\delta>0$ using \cite[Lemma~10.14]{Regularity} and therefore converges to zero.
We therefore obtain the desired bounds for all the terms.

\begin{proof}[of Theorem~\ref{thm:tightness}]
Collecting all the results of this section, we obtain the first bound  
of \eref{e:tightness}, so it remains to show the second bound.
Just as in the verification of the second bound 
in \cite[Theorem~10.7]{Regularity}, this follows in essentially the same way as 
the first bound. 
Indeed, as we consider the difference between 
$\hat\Pi_{0}^{(\eps)} \tau$
and $ \hat\Pi_{0}^{(\eps,\bar\eps)} \tau$
for any $\tau \neq \Xi$, we obtain a sum 
of expressions of the type \eref{e:1stObj},
but in each term some of the instances of $K'$
are replaced by $K'_{\bar\eps}$
and exactly one instance is replaced by $K'-K'_{\bar\eps}$.

This is because $\zeta_{\eps,\bar \eps}$ always appears as
part of an expression of the type $K' * \zeta_{\eps,\bar \eps}$,
which can be rewritten as $K'_{\bar \eps} * \zeta_\eps$.
We then use the fact that $K'_{\bar\eps}$
satisfies the same bound as $K'$, while $ K'-K'_{\bar\eps}$
satisfies the improved bound
\eref{e:approxKernel}.

The case $\tau = \Xi$ is slightly different, since in this case 
$\zeta_{\eps,\bar \eps}$  appears on its own. However, the required bound
there follows immediately from the fact that, on any fixed bounded domain,
and for any $\alpha < 0$ and sufficiently small $\delta > 0$, and for
any distribution $\eta$ with $\|\eta\|_{\CC^{\alpha + \delta}} < \infty$, one has
\begin{equ}
\|\eta - \rho_{\bar \eps} * \eta\|_{\CC^\alpha} \lesssim \bar \eps^{\delta}\|\eta\|_{\CC^{\alpha + \delta}}\;.
\end{equ}
The proof of Theorem~\ref{thm:tightness} is therefore complete.
\end{proof}

\section{General bounds} \label{sec:bounds}

In this section we derive some general results on bounding the type of integrals appearing
in the previous section, for instance \eref{e:bnd2-2}. 
These integrals are represented by labelled graphs which, after certain operations, fall into the scope of the general bounds in \cite{KPZJeremy} (see also \cite{WongZakai}). We view these labelled graphs as being obtained from ``Wick contractions" of partial graphs, and the purpose of this section
is to prove criteria on these partial graphs (i.e. Proposition~\ref{prop:main})
which are easy to verify and yield the desired bounds on the integrals.

We start with recalling the settings and results 
of the general a priori bound 
from \cite{KPZJeremy}, which should really be viewed as some version of the BPHZ theorem
\cite{BP,Hepp,Zimmermann} in this context. 
The paper \cite{KPZJeremy} considers labelled graphs $(\CV,\CE)$, in which
each vertex $v\in\CV$  represents an integration variable $x_v$, except for a
distinguished vertex $0\in\CV$ which represents the origin. 
Each edge $e=\{e_-,e_+\} \in\CE$ is labelled by a number $a_e$
and represents a kernel $J_e(x_{e_+} - x_{e_-})$ with homogeneity $-a_e$.
There are $p$ special vertices $v_{\star,1},\ldots,v_{\star,p}$ and
 $p$ distinguished
edges of the type $e = \{v_{\star, i},0\}$ with label $a_e=0$, which represent  factors 
$  J_e(x_v- x_0) = \phi^\lambda_0(x_v-0)$. 
 
\begin{remark} 
 In \cite{KPZJeremy} the edges $e = (v,\bar v)\in\CE$ are oriented and decorated 
with labels $(a_e,r_e) \in \R \times \Z$ where $r_e$ represents certain renormalisation procedure, and the orientation of an edge matters only if $r > 0$. 
In our case, since we always treat the renormalisations ``by hand"
before applying these general bounds, we ignore the label $r_e$ (in other words $r_e$ is always $0$) and the orientation, so we only have a label $a_e$.
(The orientations of the distinguished edges do not matter since we always  just assume that they are associated with function $\varphi^\lambda(x_v-0)$ rather than $\varphi^\lambda(0-x_v)$.)
\end{remark}

For every $a \in \R$, we define a (semi)norm on the space of 
compactly supported functions that are smooth everywhere, except at the origin:
\begin{equ}
\|J\|_{a} = \sup_{0 < |x|_\s \le 1} |x|_\s^{a} |J(x)|\;.
\end{equ}
We use the notation $\CV_0 = \CV \setminus \{0\}$ and
$
\CV_\star = \{0,v_{\star,1},\ldots,v_{\star,p}\}
$.
With all of these notations at hand, a labelled graph as above, together with
the corresponding collection of kernels 
determines a number
\begin{equ}[e:bigsum]
\CI_\lambda(J) \eqdef \int_{(\R^{d})^{\CV_0}} \prod_{e \in \CE}   J_e(x_{e_+}-x_{e_-})\,dx\;,
\end{equ}
where $d$ is the space-time dimension, and we implicitly set $x_0 = 0$.

To bound the quantity $\CI_\lambda(J)$, we will impose the following assumption.
Recall that for a subgraph $\bar\CV$, the sets $\CE_0(\bar\CV)$
and $\CE(\bar\CV)$ are defined in Definition~\ref{def:set-edges}.

\begin{assumption}\label{ass:graph}
The labelled graph $(\CV,\CE)$ satisfies the following properties.
\begin{itemize}\itemsep0em
\item[1.] For every subset
$\bar \CV \subset \CV$ of cardinality at least $2$, one has 
\begin{equ}[e:assEdges]
\sum_{e \in \CE_0(\bar \CV)} a_e < |\s|\,(|\bar \CV| - 1) \;.
\end{equ}
\item[2.] For every non-empty subset $\bar \CV \subset \CV\setminus\CV_\star$,
one has the bounds
\begin{equ}[e:assEdges2]
\sum_{e \in \CE(\bar \CV) 
} a_e  
> |\s|\,|\bar \CV| \;.
\end{equ}
\end{itemize}
\end{assumption}

Note that in \cite{KPZJeremy} there are four assumptions, but they reduce to the two given 
here since  we assume that the labels $r_e$  appearing in \cite{KPZJeremy} are always $0$.
We then have the following result taken from \cite{KPZJeremy}:
\begin{theorem} \label{theo:ultimate}
Provided that Assumption~\ref{ass:graph} holds, there exists a constant 
$C$ depending only on the number of  vertices in $\CV$ 
such that
\begin{equ}
|\CI_\lambda(J)| \le C \lambda^\alpha \prod_{e\in \CE} \|J_e\|_{a_e}\;,\qquad \lambda \in (0,1]\;,
\end{equ}
where $\alpha =|\s|\, |\CV\setminus\CV_\star| - \sum_{e \in \CE} a_e$.
\end{theorem}

%
%
%
%

The rest of this section is devoted to verify  that Assumption~\ref{ass:graph} indeed holds for the labelled graphs which represent our integrals,
but after some necessary modifications which we explain now.

First of all, we are actually concerned with labelled graphs $(\CV,\CE')$ 
which are of the same type as the labelled graphs $(\CV,\CE)$ defined above,
except that there may exist more than one edge in $\CE'$ between two vertices. 
More formally, $\CE'$ is a multiset (i.e. it allows 
multiple instances of elements) of unordered pairs of vertices.
(Such a graph $(\CV,\CE')$ is sometimes called a multigraph in the graph theory literature.) 
Indeed, for instance if we Wick contract two copies of the graph $H$
defined in \eref{e:def1stH} (resp. in \eref{e:def2ndH}) by identifying all the four (resp. six) external vertices as one equivalence class, 
then we obtain (multi)graphs on the left hand sides of the arrows below,
 which clearly contain multi-edges:
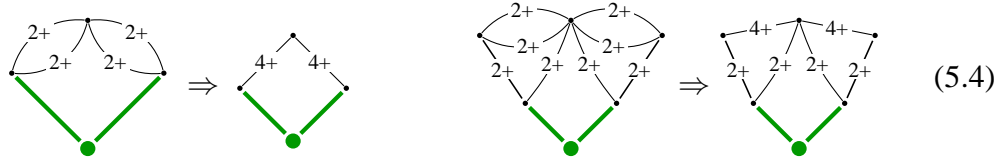
\begin{equ} [e:multi-edge]
\begin{tikzpicture}[baseline=-5]
\node[dot]	(root1) 	at (0,0) {};
\node[dot]	(root2) 	at (2,0) {};
\node[dot]		(mid)  	at (1,.7) {};
\node[root]	(root) 	at (1,-1) {};
\draw[bend left=40] (mid) to node[labl] {\scriptsize 2+}  (root1);
\draw[ bend right=40] (mid) to node[labl] {\scriptsize 2+}  (root1);
\draw[bend left=40] (mid) to node[labl] {\scriptsize 2+}  (root2);
\draw[bend right=40] (mid) to node[labl] {\scriptsize 2+}  (root2);
\draw[testfunction] (root1) to  (root);  
\draw[testfunction] (root2) to  (root); 
\node at (2.4,-0.2) {$\;\; \Rightarrow$};
\node[dot]	(root1) 	at (3,-0.2) {};
\node[dot]	(root2) 	at (4.4,-0.2) {};
\node[dot]		(mid)  	at (3.7,0.5) {};
\node[root]	(root) 	at (3.7,-0.9) {};
\draw (mid) to node[labl] {\scriptsize 4+}  (root1);
\draw (mid) to node[labl] {\scriptsize 4+}  (root2);
\draw[testfunction] (root1) to  (root);  
\draw[testfunction] (root2) to  (root);
\end{tikzpicture}
		\qquad \qquad
\begin{tikzpicture}[baseline=-5]
\node[dot]	(left1) 	at (-0.2,0.5) {};
\node[dot]	(left2) 	at (0.4,-0.4) {};
\node[dot]	(right1) 	at (2.2,0.5) {};
\node[dot]	(right2) 	at (1.6,-0.4) {};
\draw[generic]	(left1) to node[labl] {\scriptsize 2+} (left2) {};
\draw[generic]	(right1) to node[labl] {\scriptsize 2+} (right2) {};
\node[dot]		(mid)  	at (1,.7) {};
\node[root]	(root) 	at (1,-1) {};
\draw[bend left=40] (mid) to node[labl] {\scriptsize 2+}  (left1);
\draw[bend right=40] (mid) to node[labl] {\scriptsize 2+}  (left1);
\draw[bend left=10] (mid) to node[labl] {\scriptsize 2+}  (left2);
\draw[ bend left=40] (mid) to node[labl] {\scriptsize 2+} (right1);
\draw[bend right=40] (mid) to node[labl] {\scriptsize 2+}  (right1);
\draw[bend right=10] (mid) to node[labl] {\scriptsize 2+}  (right2);
\draw[testfunction] (left2) to  (root);  
\draw[testfunction] (right2) to  (root);
\node at (2.5,-0.2) {$\;\; \Rightarrow$};
\node[dot]	(left1) 	at (3,0.5) {};
\node[dot]	(left2) 	at (3.4,-0.4) {};
\node[dot]	(right1) 	at (5,0.5) {};
\node[dot]	(right2) 	at (4.6,-0.4) {};
\draw[generic]	(left1) to node[labl] {\scriptsize 2+} (left2) {};
\draw[generic]	(right1) to node[labl] {\scriptsize 2+} (right2) {};
\node[dot]		(mid)  	at (4,.7) {};
\node[root]	(root) 	at (4,-1) {};
\draw (mid) to node[labl] {\scriptsize 4+}  (left1);
\draw[bend left=10] (mid) to node[labl] {\scriptsize 2+}  (left2);
\draw (mid) to node[labl] {\scriptsize 4+} (right1);
\draw[bend right=10] (mid) to node[labl] {\scriptsize 2+}  (right2);
\draw[testfunction] (left2) to  (root);  
\draw[testfunction] (right2) to  (root);
\end{tikzpicture}
\end{equ}
However, this issue can be easily resolved because given a  labelled (multi)graph $(\CV,\CE')$,
one can naturally define a  labelled graph $(\CV,\CE)$ by identifying all the multi-edges in $\CE'$ between every pair of two vertices; the label $a_e$ for $e=\{v,\bar v\}\in\CE$ is simply defined as the sum of the labels $a_{e'}$ of all the edges $e' \in\CE'$  between $v$ and $\bar v$, see the right hand sides of the arrows in \eref{e:multi-edge}.
In the sequel the prime in the notation $\CE'$ always indicates that the set $\CE'$ allows multi-edges. 
For an edge $e\in\CE'$, we will sometimes still write $e=\{v,\bar v\}$ or $v,\bar v\in e$,
which simply means that $e$ is an edge between the two vertices $v$ and $\bar v$.

Another necessary modification is due to the following fact. 
One can verify that Assumption~\ref{ass:graph} fails to hold
in the situations illustrated in the above two graphs.
In fact, both examples contain subgraphs 
$\bar\CV$ with $|\bar\CV|=2$ and $\sum_{e \in \CE_0(\bar \CV)} a_e=4$,
and the right example also contains subgraphs $\bar\CV$
with $|\bar\CV|=3$ and $\sum_{e \in \CE_0(\bar \CV)} a_e=8$. In both cases, \eref{e:assEdges}
is violated, and one might worry that the corresponding integrals diverge.
However this will not cause any problem, since we will see that these graphs arise in situations
where one is allowed to modify 
the labels of $\CE'$ (and $\CE$ accordingly)
in such a way that Assumption~\ref{ass:graph} holds for the modified graph.
In fact, by observing the expressions \eref{e:bnd2-2}, 
it is clear that in the left example above, which arises in \eref{e:def1stH}, 
one has an additional factor $ \eps^{(|B|/2-1)|\s|} = \eps^3  $. Similarly, in the right 
example, which arises in \eref{e:def2ndH}, the factor is $\eps^6$. 
In general for each contracted vertex $v$, we have a factor $\eps^{3(\deg(v)/2-1)}$ 
arising from Lemma~\ref{lem:collapse}.
These factors associated with the ex-vertices together with the fact that their neighboring ex-edges correspond to mollified functions (see Remark~\ref{rem:in-ex-edge})
can be exploited to improve the homogeneities of the edges 
attached to $v$ and thus cure these ``fake divergences".
Here and in the sequel, $\deg(v)$ is always understood as the degree of $v$ as a vertex in $(\CV,\CE')$ (rather than $(\CV,\CE)$);
in other words $\deg(v)$ counts the multiple edges rather than regarding them as one edge.

Our main idea to verify that the  Assumption~\ref{ass:graph}
indeed holds for the objects we want to bound
is based on the observation that the graphs $(\CV,\CE')$ are actually
built by contracting simple ``atomic" graphs called partial graphs (see Definition~\ref{def:H-type} and Definition~\ref{def:Big-graph}),
and we have precise knowledge on the structures of these partial  graphs.
We will then describe the rules to ``allocate" the factors $\eps$
to the neighboring edges to cure the ``fake" divergences,
see Definition~\ref{def:eps-alloc} below. Finally we will identify multi-edges
and obtain graphs $(\CV,\CE)$ which will be shown to satisfy  Assumption~\ref{ass:graph}.

We start with the following remark about some obvious properties of our graphs $(\CV,\CE')$. 

\begin{remark} \label{rem:Vfeature}
We will frequently use the following facts. Let $(\CV,\CE')$ be a graph
constructed by Wick contracting $p$ copies of $H$ as in Definition~\ref{def:Big-graph}.
If $v$ is an in-vertex, then $\deg(v)$ is equal to the degree of the corresponding internal vertex in $H$.
If $v$ is an ex-vertex, then $\deg(v)$ is equal to the cardinality of $v$, viewed 
as an equivalence class in $H^p$.
By construction there does not exist any edge connecting two in-vertices belonging to two different copies of $H$. Also, there does not exist any edge connecting two ex-vertices of $\CV$. 
\end{remark}

\begin{definition}  \label{def:eps-alloc}
An {\it $\eps$-allocation  rule} is
a way to  assign, for
every graph $(\CV,\CE')$ constructed according to the procedure in Definition~\ref{def:Big-graph} and
 every $v\in\CV_{ex}$,
 a non-negative real number $b^{(v)}_e \ge 0$ 
to every $e\in\{e\in\CE'\,:\,v\in e\}$.
An $\eps$-allocation  rule is called {\it admissible}
 if the following holds for every such graph $(\CV,\CE')$.
\begin{claim}
\item  For every $v\in\CV_{ex}$, one has
\begin{equ} [e:adm-1]
\sum_{e\ni v} b^{(v)}_e = (\deg(v)/2-1)\,|\s|  \;.
\end{equ}
\item For every $v\in\CV_{ex}$ and every $A \subset \{e\,:\, v \in e\}$, one has
\begin{equ}[e:adm-11]
\sum_{e \in A} b^{(v)}_e \ge (|A|/2-1)\,|\s|  \;.
\end{equ}
\item For any $v_1, v_2 \in\CV_{ex}$, $v_1\neq v_2$,
define a new graph $(\hat \CV, \hat \CE')$ 
by identifying $v_1$ and $v_2$ as one vertex $w$
and apply the $\eps$-allocation rule to $\hat \CV$ to obtain numbers $\hat b_e^{(w)}$.
Then, the following monotonicity condition holds
\begin{equ} [e:adm-22]
\hat b^{(w)}_{e} \ge b^{(v_1)}_{e}  \;\;
	(\forall e\ni v_1 )
\quad \mbox{and} \quad
\hat b^{(w)}_{ e} \ge b^{(v_2)}_{ e} \;\;
	(\forall e\ni  v_2) \;.
\end{equ}
(With the obvious identification of $\CE'$ with $\hat \CE'$.)
\end{claim}
Note that by Remark~\ref{rem:Vfeature} one necessarily has $\{e:v_1\in e\} \cap \{e:v_2\in e\} =\emptyset $.
We set a convention that $b^{(v)}_e =0$ if $v\notin e$.
\end{definition}

\begin{remark}
In general the $\eps$-allocation rule could be allowed to depend on the underlying  partial graph $H$,
as long as it is the same rule for all graphs $(\CV,\CE')$ built from a given $H$.
However, in this article, we just fix one $\eps$-allocation rule for all 
the graphs.
\end{remark}

The following result will be used.
\begin{lemma} \label{lem:admis}
Suppose that we are given an admissible $\eps$-allocation  rule. Then
for any graph $\CV$ constructed as in Definition~\ref{def:Big-graph},
 for any $v_1,  v_2 \in\CV_{ex}$, $v_1\neq v_2$ and any
$Q_1 \subset \{e:v_1\in e\}$,
$Q_2 \subset \{e:v_2\in e\}$,
if we define a new graph $\hat\CV$ by identifying $v_1$ and $v_2$ as one vertex $w$ 
as in Definition~\ref{def:eps-alloc} 
(so that $Q_1 \cup Q_2 \subset\{e:w\in e\}$), 
then
\begin{equ} [e:adm-2]
\sum_{e\in Q_1} \hat b^{(w)}_{e} + \sum_{e\in Q_2} \hat b^{(w)}_{e}
\le \sum_{e\in Q_1} b^{(v_1)}_{e} + \sum_{e\in Q_2} b^{(v_2)}_{e} + |\s|  \;,
\end{equ}
where $b^{(v_1)}_{e} $  and $b^{(v_2)}_{e} $ are numbers for $\CV$
and $\hat b^{(w)}_{e}$ are numbers for $\hat \CV$
when the $\eps$-allocation  rule is applied.
\end{lemma}

\begin{proof}
Suppose that $Q_1=\{e_1,\ldots,e_{q_1}\}$ and $Q_2=\{f_1,\ldots,f_{q_2}\} $.
Using \eref{e:adm-1}, 
and the convention that  $b^{(v)}_e =0$ if $v\notin e$,
the left hand side  minus the first two terms on the right hand side of \eref{e:adm-2} is equal to
\begin{equ} [e:adm-2pf]
 \Big( D_w- \sum_{e\notin Q_1 \cup Q_2} \hat b_e^{(w)} \Big)
 - \Big(D_{v_1} - \sum_{e\notin Q_1} b_e^{(v_1)} \Big)
 - \Big(D_{v_2} - \sum_{e\notin Q_2} b_e^{(v_2)}\Big)\;,
\end{equ}
where
$D_u \eqdef (\deg(u)/2-1)\,|\s|$ for any vertex $u \in \{w,v_1,v_2\}$
and the degree $\deg(u)$ is understood to be with respect to the graph $\CV$ 
for $u\in \{v_1,v_2\}$ and with respect to the graph $\hat \CV$ for $u=w$.
Since the set $\{e: w\in e\}\setminus (Q_1\cup Q_2)$
is equal to the disjoint union of the set $\{e: v_1\in e\}\setminus Q_1$
and the set $\{e: v_2\in e\}\setminus Q_2$,
applying \eref{e:adm-22} yields 
\begin{equ}
  \sum_{e\notin Q_1} b_e^{(v_1)}
 + \sum_{e\notin Q_2} b_e^{(v_2)} 
 -\sum_{e\notin Q_1\cup Q_2} \hat b_e^{(w)} 
\le 0 \;.
\end{equ}
Therefore \eref{e:adm-2pf} is bounded by $D_w-D_{v_1}-D_{ v_2} =|\s|$.
\end{proof}

With restrictions \eref{e:adm-1} and  \eref{e:adm-22} there is not much freedom for 
an admissible $\eps$-allocation  rule. 
If $\deg(v)=2$, then one must have $b^{(v)}_{e} =0$ by \eref{e:adm-1}.
For the graphs arising in our analysis of the KPZ equation, we define 
an $\eps$-allocation  rule  as follows.
\begin{itemize}
\item If $\deg(v) > 3$, or $\deg(v) = 3$ and there are $3$ distinct
vertices connected to $v$, then $b^{(v)}_{e} = (\deg(v)/2-1)\,|\s|/\deg(v)$
for every $e$ adjacent to $v$ (i.e. the ``even allocation" rule).
\item If $\deg(v) = 3$ and there are only $2$ distinct
vertices connected to $v$ then, calling $e_1,\bar e_1$ the two edges connecting $v$ to the same vertex 
and $e_2$ the remaining edge, we set $b^{(v)}_{e_1} =b^{(v)}_{\bar e_1} =|\s|/4$
and $b^{(v)}_{e_2} =0$ (i.e. the ``divergence priority" rule).
\end{itemize}
For the KPZ equation, it turns out that if $\deg(v) = 3$, then it is never the case 
that all three edges connect $v$ to the same vertex.

\begin{lemma} \label{lem:is-admiss}
The above $\eps$-allocation  rule is admissible. 
\end{lemma}
\begin{proof}
The equality in \eref{e:adm-1} always holds by definition.
Regarding \eref{e:adm-22}, let $V=\deg (v)$, $\bar V=\deg (\bar v)$,
and $W=\deg (w)=V+\bar V$ 
(no edge between $v$ and $\bar v$ by Remark~\ref{rem:Vfeature}).
Since $W\ge 4$ the ``even allocation" rule is necessarily applied to $w$.
If both $v$ and $\bar v$ are such that the ``even allocation" rules are applied,
 it is easy to check that monotonicity condition  \eref{e:adm-22} reduces to
\begin{equ}
(V/2-1)|\s|/V \le (W/2-1)|\s|/W\;,
\end{equ}
which holds since $V\le W$ and $x \mapsto {x-2\over x}$ is increasing on $\R_+$.

Now suppose that the ``even allocation" rule is applied to $\bar v$ 
while the ``divergence priority" rule
is applied to $v$ (one then necessarily has $V=3$ and $W=\bar V+3$).
Then
\begin{equ}
{(\bar V/2 -1)|\s|\over \bar V} \vee {|\s|\over 4} 
	\le {((\bar V+3) /2 -1)|\s| \over \bar V+3 } \;,
\end{equ}
namely, the condition  \eref{e:adm-22} holds.
Here $|\s|/4$ is the largest possible value of $b_e^{(v)}$.
Finally if the ``divergence priority" rules are applied to both $v$ and $\bar v$,
so that one necessarily has $V=\bar V=3$,
then one can verify \eref{e:adm-22}  in an analogous way.
\end{proof}

\begin{remark} \label{rem:alloc-subcri}
The above divergence priority  rule is such that when $\CV$ contains a subgraph 
\begin{tikzpicture}[scale=0.7,baseline=-3]
\node[dot] at (1,0) (left) {};
\node[dot] at (2,0) (mid) {};
\node[dot] at (2.8,0) (right) {};
\draw[dashed, bend left =50] (left) to (mid);
\draw[dashed, bend right =50] (left) to (mid);
\draw[dashed] (right) to (mid);
\end{tikzpicture}
where each dashed line represents a function  $|x|_\eps^{1-|\s|}$ (think of it as $|K'_\eps|$) and a factor $\eps^{|\s|/2}$ is associated with the middle vertex,
the function $\eps^{|\s|/2}|x|_\eps^{2(1-|\s|)} \le |x|_\eps^{2-3|\s|/2} $  is integrable as long as $|\s|<4$,
which turns out to be the subcriticality condition for the KPZ equation.
For the dynamical $\Phi^4$ equation, consider the subgraph 
\begin{tikzpicture}[scale=0.7,baseline=-3]
\node[dot] at (1,0) (left) {};
\node[dot] at (2,0) (mid) {};
\node[dot] at (2.8,0) (right) {};
\draw[dashed, bend left =50] (left) to (mid);
\draw[dashed, bend right =50] (left) to (mid);
\draw[dashed] (left) to (mid);
\draw[dashed] (right) to (mid);
\end{tikzpicture}
with a factor $\eps^{|\s|}$. One should again associate all the powers of $\eps$
to the triple-edge, so that
the function $\eps^{|\s|}|x|_\eps^{3(2-|\s|)} \le |x|_\eps^{6-2|\s|}$ is integrable as long as $|\s|<6$, which is again the subcriticality condition.
(In fact for the situation 
\begin{tikzpicture}[scale=0.7,baseline=-3]
\node[dot] at (1,0) (left) {};
\node[dot] at (2,0) (mid) {};
\node[dot] at (2.8,0) (right) {};
\draw[dashed, bend left =50] (left) to (mid);
\draw[dashed, bend right =50] (left) to (mid);
\draw[dashed] (left) to (mid);
\draw[dashed, bend left =40] (right) to (mid);
\draw[dashed, bend right =40] (right) to (mid);
\end{tikzpicture}
one also has to associate more powers of $\eps$ to the triple-edge than that to the double-edge when we are very close to the criticality in a certain sense.) 
See Remark~\ref{rem:Phi4} for discussion on the dynamical $\Phi^4$ equation
in three space dimensions.
\end{remark}

As mentioned earlier, the $\eps$-allocation procedure is aimed to improve 
the homogeneities associated to the edges and thus cure the ``fake" divergences arising
from the contractions.
Given a graph $(\CV,\CE')$ constructed from Wick contracting
$p$ copies of an partial graph $H$, together with an admissible $\eps$-allocation rule, we define for every $e\in \CE'$
\begin{equ}  [e:me2ae]
a_e \eqdef m_e - b_e^{(v)}
\end{equ}
if there exists ex-vertex $v$ 
(which is necessarily unique by Remark~\ref{rem:Vfeature})
 such that $v\in e$,
 and $a_e \eqdef m_e $ otherwise.
In this way, given an $\eps$-allocation rule and a Wick contraction of $H$,
we obtain a labelled graph $(\CV,\CE', \{a_e\}_{e\in \CE'})$.

Propositions~\ref{prop:Cond-A} and \ref{prop:Cond-B} below
are the main results in this section which state that certain conditions on $H$
ensure that $(\CV, \CE', \{a_e\})$ satisfies the conditions required for
applying the results of \cite[Sec.~9]{KPZJeremy}.

Before being able to state Proposition~\ref{prop:Cond-A}, we need to introduce one more notation.
Let $H$ be a partial  graph. Given an $\eps$-allocation rule and
a subgraph $\bar H\subset H$,
we define numbers $b_{e}(\bar H)$ for edges $e$ of $ H$ as follows.
Consider a graph $\CV$ obtained from Wick contracting 
$p$ copies $H^{(1)},\ldots,H^{(p)}$ of $H$ for some $p\ge 2$, such that, 
considering $\bar H$ as a subgraph of $H^{(1)}$, 
{\it all} the elements in $\bar H\cap H_{ex}^{(1)} $ are identified in $\CV$ as one single vertex $v$.
Applying the $\eps$-allocation rule to $v$,
we obtain numbers $b_{e}^{(v)}$. 
We then define
\begin{equ}   [e:def-be-gen]
b_{e}(\bar H) \eqdef \inf_{\CV} \; b_{e}^{(v)}\;,
\end{equ}
where the infimum is over all choices of $\CV$ as above.



%
%

Given this definition, we furthermore define numbers $c_e(\bar H)$ by
\begin{equ}
c_e(\bar H) = 
\left\{\begin{array}{cl}
	b_e(\bar H) & \text{if $0 \not\in  \bar H$ and $e$ touches an ex-vertex,} \\
	|\s|/2 & \text{if $0 \in  \bar H$ and $e$ touches an ex-vertex,}\\
	0 & \text{otherwise.}
\end{array}\right.
\end{equ}
This definition in engineered in such a way that, as a consequence of \eqref{e:adm-11} 
and of \eqref{e:def-be-gen}, 
if $\CV$ is obtained by Wick contracting $p$
copies $H^{(1)}, \ldots, H^{(p)}$ of $H$ for some integer $p\ge 2$,
and $\bar H^{(i)}$ is the preimage of some $\bar \CV \subset \CV$ with $|\bar \CV_{ex}| \le 1$ under the map
$H^{(i)} \to \CV$, then
one has the bound
\begin{equ}[e:propce]
\sum_{i} \sum_{e \in \CE_0(\bar H^{(i)})} c_e(\bar H^{(i)}) \le |\s|\one_{0 \in \bar \CV}\one_{\bar \CV_{ex} \neq \emptyset} + \sum_{e \in \CE_0(\bar \CV)} b_e^{(v)} \;,
\end{equ}
where $v$ in the unique element in $\CV_{ex} $ if $|\bar \CV_{ex}| =1$ (if $|\bar \CV_{ex}| =0$, \eqref{e:propce} reads $0\le 0$.)
Here, we wrote $\bar \CV_{ex}$ as a shorthand for $\bar \CV\cap \CV_{ex}$. (We will similarly also use the notation
$\bar \CV_{in} = \bar \CV \cap \CV_{in}$.)

\begin{remark} \label{rem:be-consistent}
The number $c_{e}(\bar H)$ defined in Definition~\ref{def:be-KPZ1}
is a special case (and thus consistent with) the general definition 
of $c_{e}(\bar H)$ here with the $\eps$-allocation rule specified above Lemma~\ref{lem:is-admiss}. Indeed, with $|\s|=3$, let $n=|\bar H\cap H_{ex}|$. 
One then has $\deg(v)\ge n+1$ by the definition of a Wick contraction. 
For $n>2$, by \eref{e:adm-1} and the even allocation rule, one then has
$b_{e}(\bar H) = \inf_{k\ge 1} 3({n+k \over 2}-1) /(n+k) $,
which is equal to the right hand side of \eref{e:be-KPZ1}.
In the cases $n\in\{1,2\}$, it can be also easily seen that 
Definition~\ref{def:be-KPZ1} follows from \eref{e:def-be-gen}.
\end{remark}

We then have the following criterion. 

\begin{proposition} \label{prop:Cond-A}
Given a partial  graph $H$, fix an admissible $\eps$-allocation  rule.
Suppose that for every subset 
$\bar H \subset H$ with $|\bar H| \ge 2$ and $|\bar H_{in}| \ge 1$ one has 
\begin{equ} [e:cond-AAall]
\sum_{e \in \CE_0(\bar H)} (m_e - c_{e}(\bar H)) < |\s|\,\bigl(|\bar H_{in}| - \one_{\bar H \subset H_{in}} \bigr) \;.
\end{equ}
Then, for every $p>1$ and every graph $\CV$ obtained from Wick contracting $p$ copies of the graph $H$,
the condition 
\begin{equ}[e:wantedCondition]
\sum_{e \in \CE_0(\bar\CV)} a_e < |\s|\,(|\bar\CV| - 1) 
\end{equ}
holds for every subset 
$\bar\CV\subset \CV$ of cardinality at least $2$,
where $a_e$ is defined in \eref{e:me2ae}.
\end{proposition}

\begin{proof}
We first remark that in the case $|\bar H| = 1$,
\eqref{e:cond-AAall} always reads $0 < 0$, so that it still holds, 
but without the inequality being strict.
We first treat the case when $\bar\CV$ has at most one ex-vertex. 
As a consequence of \eqref{e:propce}, and with the same notations as above for $\bar H^{(i)}$, we have
\begin{equs}
\sum_{e \in \CE_0(\bar\CV)} a_e 
&\le |\s|\one_{0 \in \bar \CV}\one_{\bar \CV \cap \CV_{ex} \neq \emptyset}  + \sum_{i\,:\, \bar H^{(i)}_{in} \neq \emptyset} \sum_{e \in \CE_0(\bar H^{(i)})} (m_e - c_{e}(\bar H^{(i)})) \label{e:firstinequ}\\
&\le |\s|\,\Big(\one_{0 \in \bar \CV}\one_{\bar \CV \cap \CV_{ex} \neq \emptyset}  + \sum_{i\,:\, \bar H^{(i)}_{in} \neq \emptyset} \bigl(|\bar H^{(i)}_{in}| - \one_{\bar H^{(i)} \subset H^{(i)}_{in}} \bigr)\Big)\;.
\end{equs}
On the other hand, we have the identity
\begin{equ}
|\bar \CV| -1  =  \one_{0 \in \bar \CV} + |\bar \CV_{ex}| - 1
+ \sum_{i} |\bar H^{(i)}_{in}|\;.
\end{equ}
We claim that \eqref{e:wantedCondition} follows, but with an inequality that is not 
necessarily strict. Indeed, using the brutal bound $\one_{\bar H^{(i)} \subset H^{(i)}_{in}} \ge 0$ in
\eqref{e:firstinequ}, it could only fail if $0 \not\in \bar \CV$ and $\bar \CV_{ex} = \emptyset$. 
In this case however one has necessarily $\bar H^{(i)} \subset H^{(i)}_{in}$ for every $i$, so that
the inequality is restored by the fact that at least one of the $\bar H^{(i)}$ must be non-empty.
To see that the inequality is indeed strict, we note that the second inequality
in \eqref{e:firstinequ} is actually strict, unless $|\bar H^{(i)}| = 1$ for every
non-empty $\bar H^{(i)}_{in}$. 
In this case however, since $|\bar \CV| > 1$ by assumption,
it must be the case that either 1) $\bar H^{(i)} \subset H^{(i)}_{in}$ for every $i$ and therefore that at
least two of the $\bar H^{(i)}$ are non-empty, which again restores the strict inequality;
or 2)  $\bar{\CV}$ has one of the following simple
forms: $\{0,v\}$ for some  $v\neq 0$, $\{0,u, v\}$ for some in-vertex
$u$ and some ex-vertex $v$,  $\{u,v\}$  for some in-vertex
$u$ and some ex-vertex $v$, and in all these cases it is straightforward to check the desired bound.

We now turn to the case where $\bar\CV$ contains at least two ex-vertices, call them $v_1$ and $v_2$.
The idea is then to compare this to the graph $\hat \CV$ obtained by contracting
$v_1$ and $v_2$, with $\hat {\bar \CV} \subset \hat \CV$ the corresponding subgraph.
We also write $\hat a_e$ for the edge-weights of $\hat{\bar \CV}$. It then follows from
Lemma~\ref{lem:admis} that
\begin{equ}
\sum_{e \in \CE_0(\bar \CV)} a_e \le \sum_{e \in \CE_0(\hat{\bar \CV})} \hat a_e + |\s|\;.
\end{equ}
Since furthermore $|\hat {\bar \CV}| = |\bar \CV| - 1$, the required bound for $\bar \CV$ follows from
the corresponding bound for $\hat{\bar \CV}$. Successively contracting external vertices, we can therefore
reduce ourselves to the case where $\bar \CV$ contains only one ex-vertex. If the resulting graph still 
has at least two vertices, the claim follows. If not, then we have $\bar \CV \subset \CV_{ex}$, so that
$\CE_0(\bar \CV)$ is empty and the bound \eqref{e:wantedCondition} is trivially satisfied.
\end{proof}

\begin{remark} \label{rem:origin}
If $H$ is such that $\deg(0)=1$
(which is often the case in practice), 
then \eref{e:cond-AAall} only needs to be verified for graphs $\bar H$ with $0 \not\in\bar H$.
Indeed, let $\bar H$ be such that $0\in \bar H$, set $\bar H_0 = \bar H \setminus \{0\}$,
and assume that $\bar H_0$ satisfies \eref{e:cond-AAall}.
If $\bar H_{ex} = \emptyset$, then it is clear that the 
condition \eqref{e:cond-AAall} 
for $\bar H_0$ implies that 
for $\bar H$, so we only consider the case $\bar H_{ex} \neq \emptyset$.

Since $b_{e}(\bar H)$ is defined by \eref{e:def-be-gen} as an infimum over all choices of $\CV$,
it is in particular bounded by the case where $\CV$ is such that all vertices of $\bar H_{ex}$ are identified
to one single vertex $v$ which contains only one additional edge not belonging
to $\bar H$. It thus follows from 
\eref{e:adm-1} that
\begin{equ}
\sum_{e}  c_{e}(\bar H_0) = \sum_{e}  b_{e}(\bar H_0) \le 
\Big({ |\bar H_{ex}|+1 \over 2}-1\Big)\,|\s| 
\le {|\bar H_{ex}| \,|\s| \over 2}  = \sum_{e}  c_{e}(\bar H) \;,
\end{equ}
so that condition \eref{e:cond-AAall} indeed holds for $\bar H$ whenever it
holds for $\bar H_0$.
\end{remark}

\begin{proposition} \label{prop:Cond-B}
Given a partial  graph $H$,  fix an admissible $\eps$-allocation  rule such that
the numbers $b^{(v)}_e $ always satisfy $b^{(v)}_e \le |\s|/2$.
Assume that for every non-empty subgraph $\bar H\subset H\setminus H_\star$ one has 
\begin{equ}[e:cond-Bz]
\sum_{e \in \CE(\bar H)} m_e 
	> \Big( |\bar H_{in}| +{1\over 2} |\bar H_{ex}| \Big) \,|\s| \;.
\end{equ}
Then, for every graph $\CV$ obtained by Wick contracting several copies of $H$, 
and for every non-empty subgraph $\bar\CV\subset \CV\setminus \CV_\star$, one has 
\begin{equ}[e:cond-BBz]
\sum_{e \in \CE(\bar \CV)} a_e > |\s|\,|\bar \CV|  \;,
\end{equ}
where $a_e$ is defined in \eref{e:me2ae}.
\end{proposition}

\begin{remark}
It is easy to see that the KPZ allocation rule satisfies $b^{(v)}_e \le |\s|/2$.
\end{remark}

\begin{proof}
Given $\bar\CV\subset \CV\setminus \CV_\star$,
let $m=|\bar\CV_{in}|$ and $q=|\bar\CV_{ex}|$,
so that $|\bar\CV| =m+q$. Denote by $H^{(i)}$ the $i$th copy of $H$ in $\CV$ as before,
and  denote by $m_i$  the number of in-vertices of $\bar \CV$ belonging to $H^{(i)}$,
so that $m = \sum_i m_i$. We also denote by $n_i$  the number of external vertices in $H^{(i)}$
whose images under the quotient map are ex-vertices of $\bar\CV$.
With this notation, \eqref{e:cond-Bz} applied to $\bar\CV \cap H^{(i)}$ reads
\begin{equ}
\sum_{e\in \CE(\bar\CV) \cap \CE(H^{(i)})} m_e 
>
(m_i + n_i /2)\,|\s|\;,
\end{equ}
and the intersection here is well-defined because by construction
$\CE(\CV)=\cup_i \CE(H^{(i)})$. 
In fact, we actually have an even stronger bound. Indeed there may exist
an external edge $e=(u,v)$ of $H^{(i)}$ (say $u$ is an in-vertex and $v$ is an ex-vertex of $H^{(i)}$), such that $u\in\bar\CV$ but $v\notin \bar\CV$.
Suppose that there are $\tilde n_i$ such edges in $H^{(i)}$. We apply  \eqref{e:cond-Bz}  to the {\it union} of $\bar\CV \cap H^{(i)}$ {\it and} these $\tilde n_i$ ex-vertices:
\begin{equ}
\sum_{e\in \CE(\bar\CV) \cap \CE(H^{(i)})} m_e 
>
(m_i + n_i /2+\tilde n_i/2)\,|\s|\;.
\end{equ}

Summing over $i$,
one has 
\begin{equ} [e:bigger-cond]
\sum_{e\in \CE(\bar\CV) } m_e 
> (m+n/2+\tilde n/2)\,|\s|
\end{equ}
where $n \eqdef \sum_{i} n_i$ and $\tilde n \eqdef \sum_{i} \tilde n_i$.
Let now $d_1,\ldots ,d_q$ be the degrees of the $q$ ex-vertices (understood as $\CV$-degrees, namely the number of edges in $\CV$, not $\bar\CV$, attached to the vertex). 
By the procedure used to construct $\CV$ from $H$
and the fact that the degree of every external vertex
of $H$ is $1$, one has
$\sum_{i=1}^q d_i = n$, so that \eref{e:bigger-cond} is equivalent to
\begin{equ}
\sum_{e\in \CE(\bar\CV) } m_e  -\sum_{j=1}^q (d_j /2 -1 )\,|\s| - \tilde n |\s|/2
> (m+q)\,|\s| \;.
\end{equ}
By \eref{e:adm-1}, we know that for each of the $q$ ex-vertices $v$ we have 
$(d_j /2 -1 )\,|\s| = \sum_{e\ni v} b^{(v)}_e$, and every $e\ni v$ here 
belongs to $\CE(\bar\CV)$. All the other edges in $\CE(\bar\CV)$ that acquire non-zero
$ b^{(v)}_e$  are of the form $(u,v)$ described above, and there are $\tilde n$ of them, each satisfying $ b^{(v)}_e \le |\s|/2$.
Therefore, the left hand side of the above inequality is smaller or equal to $\sum_{e\in \CE(\bar\CV) } a_e $.
We then obtain
the desired bound \eref{e:cond-BBz}.
\end{proof}

\begin{corollary} \label{cor:ulti-bound}
Given a partial graph $H$ with labels $\{m_e\}_{e\in\CE(H)}$ and $p\ge 2$,  let $(\CV,\CE')$ be a graph obtained by Wick contracting $p$ copies of $H$.
%
If the conditions of Proposition~\ref{prop:Cond-A} and Proposition~\ref{prop:Cond-B}
are satisfied for $H$,
then 
the quantity defined in \eref{e:LambdaCV}
is bounded as
\begin{equ}
\Lambda^p_{\varphi_0^\lambda} (\CV) \lesssim \lambda^{\bar\alpha\, p} \;,
\end{equ}
where
\begin{equ}
\bar\alpha = |\s|\, \Big( |H_{ex}|/2 + |H_{in} \setminus H_\star| \Big)
		- \sum_{e \in \CE(H)} m_e  \;.
\end{equ}
\end{corollary}

\begin{proof}
As described in Definition~\ref{def:Big-graph},
the graph $(\CV,\CE')$  naturally inherits labels $\{m_e\}_{e\in\CE'}$
from $H$, and by definition \eref{e:me2ae} of $a_e$, one has labelled graph $(\CV,\CE',\{a_e\})$.
Let $(\CV,\CE)$ be the graph obtained from $(\CV,\CE')$ by
simply identifying each multi-edge as one single edge whose label is the sum of the labels $a_e$ of the multi-edge
as described above. 
With slight abuse of notation we still denote by $a_e$ the labels of $(\CV,\CE)$.
Then we are precisely in the setting of Theorem~\ref{theo:ultimate},
where each function $J_e(x)$ is simply $|x|^{-a_e}$ which has finite norm
$\|J_e\|_{a_e}$, except $J_e(x)=\varphi^\lambda_0(x)$
for the distinguished edges $e$.
Since the way we define $a_e$ in \eref{e:me2ae}
simply encodes the bound 
$\eps^{b_e^{(v)}} |x_v - x_{\bar v}|_\eps^{-m_e} 
	\lesssim |x_v - x_{\bar v}|^{-a_e}$
for each $e=\{v,\bar v\}$, and we precisely use up 
all the factors 
$\eps^{(\deg(v)/2-1)|\s|} $ in \eref{e:LambdaCV} due to the condition \eref{e:adm-1},
one has
\begin{equ}
\Lambda^p_{\varphi_0^\lambda} (\CV) \lesssim |\CI_\lambda(J)| \;.
\end{equ}
So it is enough to show that Assumption~\ref{ass:graph} is satisfied for $(\CV,\CE)$. 
In fact,
the conclusions of Proposition~\ref{prop:Cond-A} 
and Proposition~\ref{prop:Cond-B}
are obviously the same as the conditions \eref{e:assEdges} and \eref{e:assEdges2} of Assumption~\ref{ass:graph} respectively (note that the sum 
of labels $a_e$ on the left side of each of these conditions is the same after
we identify the multi-edges). So Assumption~\ref{ass:graph} is satisfied.

Finally the power $\bar\alpha$ is obtained from
Definition~\ref{def:Big-graph} of the procedure of constructing the graph $\CV$,
the condition \eref{e:adm-1} of the admissible $\eps$-allocation rule,
the relation \eref{e:me2ae} between $a_e$ and $m_e$,
and 
 the power $\alpha$ defined in Theorem~\ref{theo:ultimate}.
\end{proof}

%
%

\section{Identification of the limit} \label{sec:identify}

It was shown in \cite{KPZJeremy} (but see also \cite{KPZ} and \cite{FrizHairer} for 
very similar results) that if we replace
$\zeta_\eps$ by $\xi_\eps$ with
$\xi_\eps=\rho_\eps * \xi$ where $\xi$
is the Gaussian space-time white noise,
the renormalised models built from $\xi_\eps$ converge to a limit $\hat Z=(\hat \Pi, \hat\Gamma)$, which we call the KPZ model\footnote{This is a slight misnomer, since the constant parts
of $\hat \Pi_z \<4s>$ and $\hat \Pi_z\<211s>$ depend on the choice of cutoff function $K$. However,
the model does not depend on $\rho$ and, when comparing it to the model built from $\zeta_\eps$,
we will always implicitly assume that we make the same choice for $K$.}.
The goal of this section is to show that our renormalised models 
$(\hat \Pi^{(\eps)}, \hat\Gamma^{(\eps)})$
built from $\zeta_\eps$ defined above converge to the same limit.
We prove this by applying a ``diagonal argument'' along the lines of the one 
used in \cite{MourratWeber} in a similar context.
First of all, we have a central limit theorem for the random field $\zeta_\eps$.

\begin{proposition}\label{prop:conv}
Under Assumption~\ref{ass:approxField}, 
for every $\alpha < -{3\over 2}$, $\zeta_\eps$ converges in law to space-time 
white noise $\xi$ in $\CC^\alpha(\R \times S^1)$.
\end{proposition}

The proof of this proposition relies on the following result.

\begin{lemma}\label{lem:whitenoise}
Let $\zeta$ be a random variable in $\CC^\alpha(\R\times\T^d)$ 
for some $\alpha < 0$ 
which is stationary and such that, for any finite collection of test functions $\eta_1,\ldots,\eta_N$ 
with $\supp \eta_i$ not intersecting the hyperplane $\{x=vt\}$ 
(viewed as a subset of $\R\times\T^d$ in the natural way) for a constant vector $v\in\R^d$, the joint law of 
$\{\zeta(\eta_i)\}$ is Gaussian with covariance $\scal{\eta_i, \eta_j}$. Then $\zeta$ is 
space-time white noise.
\end{lemma}

\begin{proof}
Choose an orthonormal basis $\{e_n\}$ of $L^2(\R\times\T^d)$ such that each
$e_n$ is sufficiently smooth ($\CC^{|\alpha|}$ will do) and supported away from 
$\{x=vt\}$. Then, as a consequence of the assumption, $\zeta(e_n)$ is an i.i.d.\ sequence
of normal random variables. As a consequence, we can set $\xi = \sum_n \zeta(e_n)\,e_n$, 
which converges almost surely in the sense of distributions to a space-time white noise $\xi$.
It remains to show that $\xi = \zeta$ almost surely, as random variables in $\CC^\alpha$, 
which follows if we can show  that $\xi(\phi) = \zeta(\phi)$ almost surely, for any smooth
test function $\phi$ with support in some ball of radius $1/5$ (say). 

If $\phi$ is supported away from $\{x=vt\}$, this is immediate. Otherwise, write
$\phi = \lim_{n \to \infty} \phi_n$, where each $\phi_n$ is compactly supported away from $\{x=vt\}$
and the convergence is sufficiently fast so that $\xi(\phi_n) \to \xi(\phi)$ almost surely.
(It is easy to see that such an approximation always exists.) 
Let now $(T\phi)(t,x) = \phi(t,x-1/2)$ and note that one also has 
$\zeta(T\phi_n) \to \zeta(T\phi)$ almost surely since the supports of $T\phi_n$ avoid
$\{x=vt\}$ and $\zeta$ coincides with $\xi$ for such test functions. 
By stationarity, the collection of random variables
$\{\zeta(T\phi),\zeta(T\phi_n)\,:\,n \ge 0\}$ is equal in law to the collection 
$\{\zeta(\phi),\zeta(\phi_n)\,:\,n \ge 0\}$, so that $\zeta(\phi_n) \to \zeta(\phi)$ almost surely.
Since we already know that $\zeta(\phi_n) = \xi(\phi_n)$ almost surely, 
we conclude that $\zeta(\phi) = \xi(\phi)$ as required.
\end{proof}

\begin{remark}
The assumption that $\zeta$ is stationary is crucial here, otherwise the conclusion
doesn't hold in general. (Just add to $\zeta$ a Dirac measure located at the origin.)
\end{remark}

\begin{proof}[of Proposition~\ref{prop:conv}]
Setting $\zeta_\eps^{(0)}(t,x) = \eps^{-3/2}\zeta(t/\eps^2, (x-v_h t)/\eps)$, it is shown in
Lemma~\ref{lem:BerryEssen} below 
that, when testing  $\zeta_\eps^{(0)}$ against a finite number of test functions
$\eta_1,\ldots,\eta_N$, the law of the resulting $\R^n$-valued random variable converges 
to a Gaussian distribution with covariance $C_{ij} = \scal{\eta_i,\eta_j}$.

Following the same argument as in Lemma~\ref{lem:BerryEssen}, one shows that
 the sequence $\zeta_\eps$ is tight, so that it has some accumulation point
$\xi$. Since each of the $\zeta_\eps$ is stationary, $\xi$ is also stationary.
By Lemma~\ref{lem:whitenoise}, it therefore suffices to verify that it has the correct
finite-dimensional distributions when tested against test functions avoiding the
hyperplane $x = v_h t$. By stationarity, we can replace this by 
the hyperplane $x = v_h t\pm {1\over 2}$
(which is one single plane if we view the spatial variable as taking values on the circle).

For any $\delta > 0$, define the space-time domain
\begin{equ}
D_{\delta} = \{(t,x)\,:\, |t| \le T\;,\quad |x -v_h t| \le (1-\delta)/ 2\}\;,
\end{equ}
and consider a finite collection of test functions as above, but with supports 
contained in $D_\delta$.
Choosing the coupling mentioned in the assumption, we then have the bound
\begin{equs}
\E |(\zeta_\eps - \zeta_\eps^{(0)})(\eta_i)|^2 &\lesssim
\sup_{(t,x) \in \supp \eta_i} \E |\zeta_\eps(t,x) - \zeta_\eps^{(0)}(t,x)|^2 \\
&\lesssim
\eps^{-3} \sup_{|t| \le T\eps^{-2}}\sup_{|x| \le (1-\delta)/(2\eps)} 
	\E |\zeta(t,x) - \zeta^{(\eps)}(t,x)|^2\;,
\end{equs}
which converges to $0$ as $\eps \to 0$ by assumption. 
\end{proof}

 The only missing piece in the proof of Proposition~\ref{prop:conv} is the following simple lemma. Note that Assumptions~\ref{assump-0} and \ref{assump-mix}
can be also easily formulated on $\R\times \R^d$.

\begin{lemma} \label{lem:BerryEssen}
Let   $\zeta$ be a random field  on $\R\times\R^d$ equipped with scaling $\s=(2,1,\ldots,1)$ satisfying Assumptions~\ref{assump-0} and \ref{assump-mix}. 
Let 
\begin{equ}
\zeta_\eps^{(0)}(z) \eqdef \eps^{-(d+2)/2}
	\zeta\Big({t \over \eps^{2}}, {x-v^{(\eps)}t \over \eps}\Big) 
\end{equ}
where $v^{(\eps)}$ converges to a constant vector $v\in\R^d$ and $z=(t,x)\in\R\times\R^d$. 
The sequence of random fields $\zeta_\eps^{(0)}$ converges in distribution to  Gaussian white noise $\xi$ with scaling $\s$ as $\eps\to 0$
in the space $\CC^{\gamma}$ for every $\gamma<-{d+2\over 2}$.
\end{lemma}


\begin{proof}
We first show that $\E \| \zeta_\eps^{(0)} \|^p_{\CC^{-|\s|/2}}$ (where $|\s|=d+2$)
is bounded uniformly in $\eps>0$ for all $p>1$. For this it suffices  to show that
$ \lambda^{|\s|p/2} \E |\zeta_\eps^{(0)}(\varphi^\lambda_0)|^p \lesssim 1$ 
uniformly in $\eps,\lambda>0$ for all $p>1$, which follows in exactly the same way
as the proof of the case $\tau = \Xi$ in Theorem~\ref{thm:tightness}.

To identify the limit, it follows from the multidimensional version of Carleman's theorem
\cite{MR1574927,MR0184042} that it is enough to show that for any given finite  number of 
compactly supported smooth functions $\varphi_1,\ldots,\varphi_n$,
one has convergence of joint moments $\E(\zeta_\eps^{(0)}(\varphi_1)\cdots \zeta_\eps^{(0)}(\varphi_n))
 \to \E(\xi(\varphi_1)\cdots \xi(\varphi_n)) $ as $\eps \to 0$. 
This implies convergence $\zeta_\eps^{(0)} \to\xi$ in distribution, and the first part 
of the proof gives convergence in the desired topology.
 The calculation
 for the joint moment of $\zeta_\eps^{(0)}(\varphi_i)$ is the same as above,
except that $\varphi_0^\lambda$ is replaced by $\varphi_i$.
For $B\in\pi$, $|B|>2$, the integral is bounded by a positive power of $\eps$
and thus converges to zero. The non-vanishing terms are the partitions $\pi$
such that $|B|=2$ for every $B\in\pi$. Since $\kappa_2$ is normalised in
Assumption~\ref{assump-0}
such that the second cumulant of   $\zeta_\eps^{(0)}$ 
 converges to the Dirac distribution,
the joint moment converges to that of the Gaussian white noise
by standard Wick theorem.
\end{proof}

\begin{theorem} \label{theo:diagonal}
Let $\hat Z_\eps=(\hat \Pi^{(\eps)},\hat \Gamma^{(\eps)})$ be the renormalised model
built from $\zeta_\eps$ defined in the previous sections (with the choice of renormalisation
constants given by \eqref{e:choiceell}). 
Let $\hat Z=(\hat \Pi, \hat \Gamma)$ be the KPZ
random model.  Then, as $\eps\to 0$, one has
$
   \hat Z_\eps \to \hat Z
$
in distribution in the space $\MM$ of admissible models for $\TT$. 
\end{theorem}

\begin{proof}
For any fixed bounded domain $D \subset \R^2$, recall that we have an
associated ``norm'' $\$\cdot\$_D$ and pseudometric $\$\cdot,\cdot \$_D$ 
on the space $\MM$ of admissible models,
defined as in \cite[Eqs.~2.16--2.17]{Regularity}. Since the topology on $\MM$ is
generated by these pseudometrics, it is sufficient to show that 
\begin{equ}[e:wantedConv]
\lim_{\eps \to 0} \E \$\hat Z_\eps, \hat Z\$_D = 0\;,
\end{equ}
for every bounded domain $D$. 
Since $D$ does not matter much in the sequel, we henceforth omit it from our notations. 

In order to prove \eqref{e:wantedConv}, we proceed by going through the 
``intermediate'' model $\hat Z_{\eps,\bar\eps}$ previously defined in Section~\ref{sec:tightness}.
Recall that this is built in the same way as $\hat Z_\eps$, except that 
the field  $\zeta_\eps$  used to build the model is replaced by
$ \zeta_{\eps,\bar\eps} \eqdef \zeta_\eps*\rho_{\bar\eps}$,
where $\rho$ is a compactly supported  smooth function on $\R^2$ integrating to $1$ that is 
even in the space variable, and with
$\rho_{\bar\eps} = \bar\eps^{-3} \rho (\bar\eps^{-2} t, \bar\eps^{-1} x)$.
We furthermore define a model $\hat Z_{0,\bar\eps}$ obtained again in the same
way, but this time replacing
$\zeta_\eps$ by $ \zeta_{0,\bar\eps} \eqdef \xi*\rho_{\bar\eps}$, where $\xi$
is a realisation of space-time white noise.

With these notations at hand, we then have
\begin{equ}
\E \$\hat Z_\eps, \hat Z\$
\le \E \$\hat Z_\eps, \hat Z_{\eps,\bar \eps}\$ + \E \$\hat Z_{\eps,\bar \eps}, \hat Z_{0,\bar \eps}\$ + \E \$\hat Z_{0,\bar \eps}, \hat Z\$\;,
\end{equ}
which is valid for any fixed $\bar \eps > 0$ and for any coupling between $\xi$ and $\zeta_\eps$, so that 
\begin{equ}[e:threeTerms]
\lim_{\eps \to 0} \E \$\hat Z_\eps, \hat Z\$ \le \lim_{\bar \eps \to 0} \lim_{\eps \to 0} \bigl(\E \$\hat Z_\eps, \hat Z_{\eps,\bar \eps}\$ + \E \$\hat Z_{\eps,\bar \eps}, \hat Z_{0,\bar \eps}\$ + \E \$\hat Z_{0,\bar \eps}, \hat Z\$\bigr)\;.
\end{equ}
Note the order of limits: we \textit{first} take $\eps \to 0$ and \textit{then} only take $\bar \eps \to 0$.

The convergence of $\hat Z_{0,\bar \eps}$ to the KPZ model in the sense that 
\begin{equ}
\lim_{\bar \eps \to 0} \E \$\hat Z_{0,\bar \eps}, \hat Z\$^p = 0\;,
\end{equ}
 for every $p$
was 
already shown in \cite[Thm~6.1]{KPZJeremy} (see also \cite{KPZ,FrizHairer} 
in a slightly different setting).
Regarding the first term, it suffices to note that, by \cite[Thm~10.7]{Regularity}
and the second bound in Theorem~\ref{thm:tightness},
we obtain the bound
\begin{equ}
\E \$\hat Z_{\eps}, \hat Z_{\eps,\bar \eps}\$ \lesssim \bar \eps^{\kappa}\;,
\end{equ}
uniformly over $\eps$ sufficiently small, so that this term also vanishes.

It therefore remains to bound the second term. This term involves both $\zeta_\eps$
and $\xi$, so we need to specify a coupling between them. For this, we recall that, given any sequence
$X_n$ of random variables on a complete separable
metric space $\CX$, weak convergence to a limit $X$
and uniform boundedness of some moment of order strictly greater than $p$ of $d(0,X_n)$ implies
convergence in the $p$-Wasserstein metric, which in turn implies that, for every $n$, there exists 
a coupling between $X_n$ and $X$ such that 
\begin{equ}
\lim_{n \to \infty} \E d(X_n, X)^p = 0\;,
\end{equ}
see for example \cite{Cedric}. Taking for $\CX$ the space $\CC^{-2}$ (say) on any
bounded domain $\KK$ with the metric 
given by its norm, it follows from Proposition~\ref{prop:conv} and Theorem~\ref{thm:tightness}
(with $\tau = \Xi$)
that we can find couplings between $\zeta_\eps$ and $\xi$ such that 
\begin{equ}
\lim_{\eps \to 0} \E \|\xi - \zeta_\eps\|_{-2;\KK}^p = 0\;.
\end{equ}
At this point we use the fact that, for any fixed $\bar \eps > 0$, convolution with $\rho_{\bar \eps}$
maps $\CC^{-2}$ into $\CC^1$ (actually $\CC^\infty$), so that 
\begin{equ}[e:convC1]
\lim_{\eps \to 0} \E \|\zeta_{0,\bar \eps} - \zeta_{\eps,\bar \eps}\|_{1;\KK}^p = 0\;.
\end{equ}

Consider now the space $\CY$ of all stationary  
random processes $\eta$ on some given probability space with the following 
additional properties:
\begin{claim}
\item The process $\eta$ is almost surely periodic in space with period $1$.
\item One has $\E \|\eta\|_{1;\KK}^p < \infty$ for every $p \ge 1$. 
\end{claim}
We endow the space $\CY$ with the seminorm
\begin{equ}
\|\eta\|_\CY^p = \E \|\eta\|_{1;\KK}^p\;,
\end{equ}
and we denote by $\hat \Psi$ the map $\zeta_\eps \mapsto \hat Z_\eps$, viewed as a map
from $\CY$ into the space of
$\MM$-valued random variables.
Here, the requirement that the argument be a stationary stochastic process is needed
for the renormalisation constants to be well-defined by the formulae in Section~\ref{sec:values}.

It is then immediate from the definitions that, for some sufficiently large $p$, 
$\hat \Psi$ satisfies
\begin{equ}
\E \$\hat \Psi(\zeta),\hat \Psi(\eta)\$ \lesssim \|\zeta - \eta\|_{\CY}\bigl(1 + \|\zeta\|_{\CY}+ \|\eta\|_{\CY}\bigr)^p\;,
\end{equ}
Combining this with \eqref{e:convC1} and the fact that 
$\hat Z_{\eps,\bar\eps} = \hat\Psi(\zeta_{\eps,\bar \eps})$, 
it immediately follows that 
\begin{equ}
\lim_{\eps \to 0} \E \$\hat Z_{\eps,\bar \eps}, \hat Z_{0,\bar \eps}\$ = 0\;,
\end{equ}
for every fixed (sufficiently small) $\bar \eps > 0$, so that 
the second term in \eqref{e:threeTerms} also vanishes, thus concluding the proof.
\end{proof}

%

\begin{remark} \label{rem:Phi4}
Our proof also extends to the dynamical $\Phi^4_3$ model 
driven by a general noise $\zeta$ under the same assumptions as here (except that 
the spatial variable takes values in $\T^3$ instead of $S^1$). 
Besides the  constant $C_1^{(\eps)}\approx \eps^{-1}$
and the logarithmically divergent constant $C_2^{(\eps)}$ as in the Gaussian case, we then have
the following new renormalisation constants appearing:
\begin{equ}
\bar C_{1}^{(\eps)} =
\begin{tikzpicture}  [baseline=10] 
\node[root]	(root) 	at (0,0) {};
\node at (0,1) (a) {};  
	\node[cumu3,rotate=180] at (a) {};
\draw[kernel] (a.east) node[dot]   {} to [bend left = 60] (root);
\draw[kernel] (a.west) node[dot]   {} to [bend right = 60] (root);
\draw[kernel] (a.south) node[dot]   {} to (root);
\end{tikzpicture}
\approx \eps^{-3/2} \;,
\quad
\bar C_{2}^{(\eps)} =
\begin{tikzpicture}  [baseline=15]
\node[root]	(root) 	at (0,0) {};
\node[dot]		(top)  	at (0,1.3) {};
\node		(right)  	at (1,1.15) {};
	\node[cumu5,rotate=36]	(right-cumu) 	at (right) {};
\draw[->,kernel] (top) to  (root);
\draw[->,kernel,bend left=40] (right.south east) node[dot] {} to  (root);
\draw[->,kernel] (right.south) node[dot] {} to  (root);
\draw[->,kernel] (right.north west) node[dot] {} to  (top);
\draw[->,kernel,bend right=60] (right.north east) node[dot] {} to  (top);
\draw[->,kernel,bend left=30] (right.south west) node[dot] {} to  (top);
\end{tikzpicture}
\approx \eps^{-1/2}
\;,
\quad
\bar C_{3}^{(\eps)} =
\begin{tikzpicture}  [baseline=15]
\node[root]	(root) 	at (0,0) {};
\node[dot]		(top)  	at (0,1.3) {};
\node		(rt)  	at (0.5,0.7) {};
	\draw[cumu2] (rt) ellipse (4pt and 8pt);
\node		(right)  	at (1,1.15) {};
	\node[cumu3,rotate=180]	(right-cumu) 	at (1,1.2) {};
\draw[->,kernel] (top) to  (root);
\draw[->,kernel,bend left=40] (right.south) node[dot] {} to  (root);
\draw[->,kernel] (rt.south) node[dot] {} to  (root);
\draw[->,kernel] (right.north west) node[dot] {} to  (top);
\draw[->,kernel,bend right=60] (right.north east) node[dot] {} to  (top);
\draw[->,kernel] (rt.north) node[dot] {} to  (top);
\end{tikzpicture}
\approx \eps^{-1/2}
\;,
\end{equ}
\begin{equ}
\bar C_{4}^{(\eps)} =
\begin{tikzpicture}  [baseline=15]
\node[root]	(root) 	at (0,0) {};
\node[dot]		(left)  	at (0,1.3) {};
\node		(right)  	at (1,1.15) {};		
	\node[cumu4]	(right-cumu) 	at (right) {};
\draw[kernel] (left) to  (root);
\draw[kernel]   (right.south east) node[dot]   {}  to (root) ;
\draw[kernel]  (right.north east) node[dot]  {}  [bend right=60] to (left);
\draw[kernel]   (right.north west) node[dot]  {} to (left);
\draw[kernel]  (right.south west) node[dot]  {}  [bend left=20] to (left);
\end{tikzpicture} 
\; \stackrel{\eps\to 0}{\longrightarrow} c_4 \;,
\qquad
\bar C_{5}^{(\eps)} =
\begin{tikzpicture}  [baseline=15]
\node[root]	(root) 	at (0,0) {};
\node[dot]		(top)  	at (0,1.3) {};
\node		(right)  	at (1,1.15) {};
	\node[cumu4]	(right-cumu) 	at (right) {};
\draw[->,kernel] (top) to  (root);
\draw[->,kernel,bend left=40] (right.south east) node[dot] {} to  (root);
\draw[->,kernel] (right.south west) node[dot] {} to  (root);
\draw[->,kernel] (right.north west) node[dot] {} to  (top);
\draw[->,kernel,bend right=60] (right.north east) node[dot] {} to  (top);
\end{tikzpicture}
\; \stackrel{\eps\to 0}{\longrightarrow} c_5
\;,
\end{equ}
where $c_4,c_5$ are some finite constants.
The renormalised equation is then given by
\begin{equs}
\partial_t  \Phi_\eps =  \Delta  \Phi_\eps  & - \lambda \Phi_\eps^3 
 + \Big( 3 \lambda C_1^{(\eps)}-9 \lambda^2 C_2^{(\eps)} 
	- 6 \lambda^2 c_4 - 9 \lambda^2 c_5 \Big) \Phi_\eps   \\
 & + \Big(\lambda \bar C_{1}^{(\eps)} -3\lambda^2 \bar C_{2}^{(\eps)}
 	-18\lambda^2 \bar C_{3}^{(\eps)} \Big)	+ \zeta_\eps \;.
\end{equs}
Using the method developed in this article, one can show that $\Phi_\eps$ converges to a limit which 
coincides with the solution to the dynamical $\Phi^4_3$ equation driven by Gaussian space-time white noise
as constructed in \cite{Regularity,CC,Kupiainen}.
\end{remark}

\bibliographystyle{./Martin}
\bibliography{./refs}

\end{document}